\documentclass[a4paper]{article}

\usepackage[english]{babel}
\usepackage[dvipsnames]{xcolor}
\usepackage{amssymb}
\usepackage{amsthm}
\usepackage{caption}
\usepackage{amsmath}  
\usepackage{color} 
\usepackage{epsfig}
\usepackage{graphicx}
\graphicspath{{images/}}
 \usepackage{hyperref}
\usepackage{mathtools}

\DeclarePairedDelimiter\floor{\lfloor}{\rfloor}

\usepackage[shortlabels]{enumitem}
\mathtoolsset{showonlyrefs=true}

\theoremstyle{plain}
\newtheorem{theorem}{Theorem}[section]
\newtheorem{corollary}[theorem]{Corollary}
\newtheorem{lemma}[theorem]{Lemma}

\newtheorem{remark}[theorem]{Remark}

\newtheorem{assumption}{Assumption}

\newcommand{\oU}{\overline{U}}

\newcommand{\R}{\mathbb{R}}
\newcommand{\K}{\mathcal{K}}
\newcommand{\N}{\mathbb{N}}
\newcommand{\cL}{\mathcal{L}}

\newcommand{\bydef}{\stackrel{\mbox{\tiny\textnormal{\raisebox{0ex}[0ex][0ex]{def}}}}{=}}

\title{
Recent advances about the rigorous integration of parabolic PDEs via fully spectral Fourier-Chebyshev expansions
}

\author{
Matthieu Cadiot
\footnote{McGill University, Department of Mathematics and Statistics, 805 Sherbrooke Street West, Montreal, QC, H3A 0A9, Canada. {\tt matthieu.cadiot@mail.mcgill.ca}}
\and
Jean-Philippe Lessard \footnote{McGill University, Department of Mathematics and Statistics, 805 Sherbrooke Street West, Montreal, QC, H3A 0A9, Canada. {\tt jp.lessard@mcgill.ca}}
}

\date{}

\begin{document}

\maketitle

\begin{abstract}
This paper presents a novel approach to rigorously solving initial value problems for semilinear parabolic partial differential equations (PDEs) using fully spectral Fourier-Chebyshev expansions. By reformulating the PDE as a system of nonlinear ordinary differential equations and leveraging Chebyshev series in time, we reduce the problem to a zero-finding task for Fourier-Chebyshev coefficients. A key theoretical contribution is the derivation of an explicit decay estimate for the inverse of the linear part of the PDE, enabling larger time steps. This allows the construction of an approximate inverse for the Fréchet derivative and the application of a Newton-Kantorovich theorem to establish solution existence within explicit error bounds. Building on prior work, our method is extended to more complex partial differential equations, including the 2D Navier-Stokes equations, for which we establish global existence of the solution of the IVP for a given nontrivial initial condition.
\end{abstract}

\begin{center}
\end{center}

\begin{center}
{\bf \small Mathematics Subject Classification (2020)}  \\ \vspace{.05cm}
{\small 65M99 $\cdot$ 65G20 $\cdot$ 65M70 $\cdot$ 65Y20 $\cdot$ 35Q30} 
\end{center}

\section{Introduction}

The study of the time evolution of solutions to semilinear parabolic partial differential equations (PDEs) is a central topic in applied mathematics, with numerous applications in fields such as fluid mechanics, materials science, pattern formation, population dynamics, and autocatalytic chemical reactions, among others. A key challenge in this area is the development of rigorous methods for solving initial value (Cauchy) problems. This is a particularly difficult task due to the presence of nonlinearities and the infinite-dimensional nature of the phase space in which these PDEs are defined. In this paper, we introduce a novel computer-assisted methodology, improving significantly the fully spectral approach \cite{jacek_integration_cheb} to tackle this problem. Our method is part of a broader collaborative effort aimed at advancing rigorous computational techniques to study the flow of dissipative PDEs. This global initiative is motivated by the fact that the infinite-dimensional dynamics of parabolic PDEs represent a fundamental problem in mathematics, as exemplified by the investigation of global existence and behavior of solutions in Cauchy problems for the Navier-Stokes equations. Recent years have seen significant progress in the field of computer-assisted proofs in nonlinear analysis, with the development of various approaches in this direction. These include topological methods based on covering relations and self-consistent bounds \cite{MR2049869,MR2788972,chaos_KS,MR3167726,MR3773757,Cy,CyZ}, the $C^1$ rigorous integrator approach from \cite{MR2728184}, the semi-group methodology of \cite{MR3639578,MR3683781,TLJO,JLT_Schrodinger,MR4356641,integrator_gabriel}, the finite element discretization methods \cite{nakao1,MR4015319,nakao3}, as well as the Chebyshev interpolation and domain decomposition strategy from \cite{MR4777935}, highlighting the rich diversity of strategies employed in this field.

The focus of the present paper is on studying the following class of initial value problems
\begin{equation}\label{eq : initial pde_intro} 
\partial_t u = \mathbb{L} u + \mathbb{Q}(u) + \phi, \quad
    u(0,x) = b(x),
\end{equation}
where $m \in \N$, $x = (x_1, \dots, x_m) \in \Omega \bydef (-\pi,\pi)^m$, and where $u = u(t,x)$, $\phi = \phi(t,x)$, and $b$ are functions that are $2\pi$-periodic in $x$. Under the conditions specified in Assumptions~\ref{ass : assumption on L} and \ref{ass : assumption on Q}, we assume that \( \mathbb{L} \) is a spatial elliptic differential operator, and that the PDE is parabolic. Following the set-up of \cite{jacek_integration_cheb}, our method is structured as follows: we express the solution of the PDE \eqref{eq : initial pde_intro} as a Fourier series in space, reducing the problem to an infinite system of nonlinear ordinary differential equations (ODEs) over a finite time interval $[0, h]$, for some $h > 0$. Using the Fourier coefficients $\beta$ of the initial condition, we reformulate the ODEs as Picard integral equations and expand the solution as a Chebyshev series in time. This process leads to an equivalent zero-finding problem of the form $F(U) =  U + \frac{h}{2} \mathcal{L}^{-1} \left( \mathbf{\Lambda} (Q(U) + \Phi) - \beta \right)$, where $U$ represents an infinite two-index sequence of Fourier-Chebyshev coefficients corresponding to the solution of the Cauchy problem, $\cL$ and $\mathbf{\Lambda}$ are linear operators, and $Q(U)$, $\Phi$ and $\beta$ represent the Fourier-Chebyshev representation of the terms $\mathbb{Q}(u)$, $\phi$ and $b$ in \eqref{eq : initial pde_intro}, respectively. After projecting the problem onto a finite-dimensional subspace, we apply Newton's method to compute a numerical approximation $\oU$ of $F=0$. The idea is then to invoke a Newton-Kantorovich type theorem to prove that $F$ has a true zero close to $\oU$, which relies on the fundamental problem of this paper, namely to construct a good enough approximate inverse for the Fréchet derivative $DF(\oU)$. This construction begins by proving the invertibility of the operator $\cL$ on a suitable Banach space and to derive explicit and quantitative control on its inverse. An important feature of the operator $\cL$ is that it can be viewed as a block diagonal operator $\cL = (\cL_k)_{k}$ where $k$ is the (space) Fourier order and where $\cL_k$ is an operator acting on an infinite sequence of Chebyshev (time) coefficients. Moreover, $\cL_k$ consists of the sum of an infinite-dimensional tridiagonal operator and a rank one operator, and becomes less and less diagonal dominant as $k$ goes to infinity, a feature complicating significantly the analysis. It is important to note that tridiagonal operators frequently appear in computer-assisted proofs involving differential equations, as demonstrated in \cite{cyranka_mucha}, which examines elliptic systems, and \cite{tridiagonal}, which analyzes non-autonomous differential operators that result in tridiagonal operators with unbounded off-diagonal entries.

In studying the operator $\cL$ and proving its invertibility, we prove a novel result, namely that there exists an explicit constant $C$ such that 
\begin{equation} \label{eq:new_tail_estimate}
\| \cL_k^{-1} \Lambda \| \le \frac{C}{2+\mu_k},
\end{equation}
where $\mu_k = -\frac{h}{2} \lambda_k$ with $\lambda_k \to -\infty$ being the eigenvalues of the spatial differential operator $\mathbb{L}$ in \eqref{eq : initial pde_intro} (see Assumption~\ref{ass : assumption on L}). Note that $\Lambda$ consists of the constant blocks diagonal operators appearing in $\mathbf{\Lambda}$, that is $\mathbf{\Lambda} = (\Lambda)_k$,
where $k$ is the (space) Fourier order and where $\Lambda$ is an operator acting on an infinite sequence of Chebyshev (time) coefficients. The central estimate \eqref{eq:new_tail_estimate} is a novel result that offers valuable insights into fully spectral methods for solving Cauchy problems in PDEs, and it represents one of the primary theoretical contributions of this paper. As we will demonstrate, the ability to control this decay has two major advantages. First, it establishes that the operator $\mathcal{L}^{-1} \mathbf{\Lambda}$ is compact, allowing us to conclude that the Fréchet derivative $DF(U)=I_d + \frac{h}{2} \mathcal{L}^{-1} \mathbf{\Lambda} DQ(U)$ is a compact perturbation of the identity. This property is particularly valuable for computations, as it enables the Fréchet derivative to be approximated using large finite-dimensional matrices. Second, this control over the decay permits the use of larger time steps in numerical integration, enhancing computational efficiency.

Utilizing the newly discovered decay property \eqref{eq:new_tail_estimate} of the inverse of the differential operator in the Fourier tail, we construct an approximate inverse $A$ for $DF(\oU)$. By invoking the above mentioned Newton-Kantorovich theorem, we demonstrate that the operator $T(U) \bydef U - A F(U)$ is a contraction on a closed ball $B_r(\oU)$ of radius $r > 0$, centered at the numerical approximation $\oU$. Applying the contraction mapping theorem yields a unique solution to $F = 0$ within $B_r(\oU)$, which corresponds to the solution of the Cauchy problem over an explicit time interval. The explicit radius $r > 0$ provides a rigorous $C^0$-error bound between the true solution of the Cauchy problem \eqref{eq : initial pde_intro} and its numerical approximation.

While the approach of the present paper is strongly influenced by the previous work \cite{jacek_integration_cheb}, it offers two significant advances. The first one is that this new approach  enables the use of larger time steps in our computations, thanks to the theoretical advancement \eqref{eq:new_tail_estimate} and to its effect on the construction of a good enough approximate inverse $A$. Furthermore, the work in \cite{jacek_integration_cheb} acknowledged certain limitations of the approach, particularly in its applicability to PDEs with derivatives in the nonlinearity. In \cite{jacek_integration_cheb}, it is speculated--based on numerical experiments--that the method could be extended to models such as the Kuramoto-Sivashinsky equation and the phase-field crystal (PFC) equation, where the derivative order of the nonlinearity is lower than that of the linear part. However, the authors did not anticipate how to adapt the method to more complex models, including Burgers' equation, the Cahn-Hilliard equation, or the Ohta-Kawasaki model. In this paper, we overcome these limitations, applying our method to these more general classes of models. 

This paper is organized as follows. In Section \ref{sec:formulation}, we present the set-up of the methodology, introducing the required notations for our analysis. In particular, assuming some properties of $\mathcal{L}^{-1}\mathbf{\Lambda}$ (aka \eqref{eq : ineq C0 C1n CA0n}), we demonstrate that our approach allows the development of a Newton-Kantorovish argument for proving existence of solutions to \eqref{eq : initial pde}. In Section \ref{sec : computation of C0 C1}, we verify the aforementioned properties of  $\mathcal{L}^{-1}\mathbf{\Lambda}$  by establishing an explicit control of $\mathcal{L}^{-1}$. Finally, we illustrate our approach in Section \ref{sec : applications} with IVPs in the 2D Navier-Stokes equations, in the Swift-Hohenberg PDE and in the Kuramoto-Sivashinski PDE. Notably, we establish global-in-time existence for a specific IVP in the 2D Navier-Stokes equations by explicitly constructing a trapping region that confines trajectories to the zero solution. All computer-assisted proofs, including the requisite codes, are accessible on GitHub at \cite{julia_cadiot}.

\section{Set-up, Newton-Kantorovich bounds and Regularity} \label{sec:formulation}

In this section, we first formulate the zero-finding problem of the form $F(U) = 0$ in a Banach space of Fourier-Chebyshev coefficients, where the solution $U$ corresponds to that of the initial value problem. In Section~\ref{ssec : Newton Kantorovich}, we introduce the Newton-Kantorovich theorem, which we employ to establish computer-assisted proofs of the existence of a solution to $F = 0$. Section~\ref{sec : computation of the bounds} then presents the analytical framework necessary for computing the bounds appearing in Theorem~\ref{th : radii polynomial}. Finally, in Section~\ref{sec:regularity}, we demonstrate that solving $F = 0$ yields strong solutions to the partial differential equation (that is, they possess sufficient regularity in both space and time to satisfy the original Cauchy problem.).

Given $m \in \mathbb{N}$, we study the following class of initial value problems 
\begin{equation}\label{eq : initial pde}
    \partial_t u = \mathbb{L} u + \mathbb{Q}(u) + \phi, \quad
    u(0,x) = b(x), 
\end{equation}
where $x=(x_1,\dots,x_m) \in \Omega \bydef (-\pi,\pi)^m$ and where $u = u(t,x)$, $\phi = \phi(t,x)$ and $b$ are given functions which are spatially  $2\pi$-periodic in the variable $x$. In particular, we assume that 
\begin{align}
    \phi(t,x) = \sum_{k \in \mathbb{Z}^m}\phi_k(t) e^{ik\cdot x} ~~ \text{ and } ~~ b(x) = \sum_{k \in \mathbb{Z}^m} b_k e^{ik\cdot x}.
\end{align}
In addition, we look for a solution $u$ of \eqref{eq : initial pde} of the form
\begin{equation}\label{eq : periodic ansatz in space}
    u(t,x) = \sum_{k \in \mathbb{Z}^m} u_k(t) e^{ik \cdot x},
\end{equation}
that is $u$ is periodic in space. Now, we complement  \eqref{eq : initial pde} with assumptions on the operators $\mathbb{L}$ and $\mathbb{Q}$. First, we assume that $\mathbb{L}$ is a spatial elliptic differential operator satisfying the following property.
\begin{assumption}\label{ass : assumption on L}
    For all $k \in \mathbb{Z}^m$, there exists $\lambda_k \in \R$ such that
    \begin{align*}
    \mathbb{L} e^{i k\cdot x} = \lambda_k e^{i k\cdot x}
\end{align*}
for all $x \in \Omega$. Moreover, assume that $\lambda_k \to -\infty$ as $|k| \to \infty$ and denote
\begin{align}\label{eq : positivity condition of nuk}
   \lambda_{\sup} \bydef \sup_{k \in \mathbb{Z}^m} \lambda_k <  \infty.
\end{align}
\end{assumption}

Then, we assume that $\mathbb{Q}$ is an autonomous polynomial non-linearity satisfying the following.
\begin{assumption}\label{ass : assumption on Q}
    There exists $p \in \mathbb{N}$ such that 
\begin{align*}
    \mathbb{Q}(u) =  \mathbb{Q}_1 (u) \cdots \mathbb{Q}_p(u)
\end{align*}
where $\mathbb{Q}_j$ ($j \in \{1,\dots, p\}$) is a linear  operator such that
\begin{align*}
    \mathbb{Q}_j(e^{i k\cdot x}) = q_{j,k} e^{i k\cdot x}
\end{align*}
for all $x \in \Omega$ and all $k \in \mathbb{Z}^m$. Moreover, we assume that \eqref{eq : initial pde} is semi-linear, that is, for each $j \in \{1,\dots, p\}$
\begin{align}\label{eq : assumption on non linear term}
    \frac{q_{j,k}}{\lambda_k} \to 0~~ \text{ as } |k| \to \infty.
\end{align}
\end{assumption}

\begin{remark}\label{rem : general non linear term}
In the cases for which $\mathbb{Q}$ is a general polynomial non-linearity, then $\mathbb{Q}$ can be written as a linear combination of operators of the form $\mathbb{Q}_1 (u) \cdots \mathbb{Q}_p(u)$. Then, the analysis of this paper applies by linearity. Note that, more generally, one could also consider analytic non-linear operators. We choose to present the method for a general non-linear term $\mathbb{Q}_1 (u) \cdots \mathbb{Q}_p(u)$ for simplicity. 
\end{remark}

Using the above assumption and substituting the ansatz \eqref{eq : periodic ansatz in space} in \eqref{eq : initial pde}, we obtain that
\begin{equation}\label{eq : pde in Fourier series}
    \frac{d u_k}{dt} = \lambda_ku_k + \left(\mathbb{Q}(u)\right)_k + \phi_k, \quad
    u_k(0) = b_k
\end{equation}
for all $k \in \mathbb{Z}^m$, where $\left(\mathbb{Q}(u)\right)_k$ is the $k$-th Fourier coefficient of $\mathbb{Q}(u).$ In particular, using Assumption~\ref{ass : assumption on Q}, we can use the discrete convolution to express $\left(\mathbb{Q}(u)\right)_k$. Indeed, we have
\begin{align}
\left(\mathbb{Q}(u)\right)_k = \sum_{k_1 + \dots + k_p = k} q_{1,k_1}u_{k_1}\dots q_{p,k_p}u_{k_p}    
\end{align}
for all $k \in \mathbb{Z}^m.$ 
Now, as presented in \cite{jacek_integration_cheb}, we fix a {\em time step} $h>0$ and look for a solution to \eqref{eq : initial pde} on $[0,h] \times \Omega$.
First, using the rescaling $t \mapsto \frac{2t}{h}-1$, we equivalently look for a solution to \eqref{eq : pde in Fourier series} as 
\begin{equation}\label{eq : pde in Fourier series with rescaling}
    \frac{du_k}{dt} =  \frac{h}{2}f_k(u), \quad
      u_k(-1) =   b_k
\end{equation}
for all $t \in [-1,1]$, where 
\[
f_k(u) \bydef  \lambda_k u_k + \left(\mathbb{Q}(u)\right)_k + \phi_k.
\]
Moreover, we expand each $u_k=u_k(t)$ as a Chebyshev series of the first kind, that is 
\[
u_k(t) = U_{k,0} +  2\sum_{n \in \mathbb{N}} U_{k,n}T_n(t),
\]
where $T_n$ ($n \ge 0$) are the Chebyshev polynomials of the first kind.
Let $\mathbb{N}_0 \bydef \mathbb{N} \cup \{0\}$ and denote $U \bydef (U_{k,n})_{(k,n) \in \mathbb{Z}^m \times \mathbb{N}_0}$. Moreover, for all $k \in \mathbb{Z}^m$, denote $U_k \bydef (U_{k,n})_{n \in \mathbb{N}_0}$. In other terms, we look for a solution $u$ of the form
\begin{equation}\label{eq : expansion solution}
u(t,x) =  \sum_{k \in \mathbb{Z}^m} U_{k,0}  e^{ik\cdot x} + 2\sum_{n \in \mathbb{N}} \sum_{k \in \mathbb{Z}^m} U_{k,n} T_n\left(t\right) e^{ik\cdot x}.
\end{equation}
Then, define $(\mu_k)_{k \in \mathbb{Z}^m}$ as 
\begin{equation}\label{eq : translation spectrum}
     \mu_k \bydef -\frac{h}{2}(\lambda_k - \lambda_{\sup})
\end{equation}
for all $k \in  \mathbb{Z}^m$, where $\lambda_{\sup}$ is given in \eqref{eq : positivity condition of nuk}. In particular, 
\begin{equation}\label{eq : positivity of muk}
    \mu_k \geq  0 \text{ for all } k \in \mathbb{Z}^m.
\end{equation}
This fact will notably be useful for the inversion of the linear operator  $\mathcal{L}$ defined below.
Letting $\phi_k(t) = \Phi_{k,0} +  2\sum_{n \in \mathbb{N}} \Phi_{k,n}T_n(t)$ and using \cite{jacek_integration_cheb}, we obtain that $U = (U_{k,n})$ solves
\begin{equation}\label{eq : pde with Fourier and cheb pointwise}
\begin{aligned}
\mu_kU_{k,n-1} +  2n U_{k,n} - \mu_kU_{k,n+1} + \frac{h}{2} \left(Q_{k,n+1}(U) - Q_{k,n-1}(U)\right) + \frac{h}{2}(\Phi_{k,n+1}-\Phi_{k,n+1}) &= 0\\
      U_{k,0} + 2   \sum_{n \in \mathbb{N}} (-1)^n U_{k,n} &=   b_k,
\end{aligned}
\end{equation}
where, using the discrete convolution and Assumption \ref{ass : assumption on Q}, we have 
\begin{align}
    Q_{k,n}(U) \bydef  \lambda_{\sup} U_{k,n} +    \sum\limits_{\substack{n_1+\dots + n_p = n \\ k_1+\dots + k_p = k\\ k_i,n_i \in \mathbb{Z}}} q_{1,k_1} U_{k_1,|n_1|} \dots  q_{p,k_p}U_{k_p,|n_p|}.
\end{align}
Given $k \in \mathbb{Z}^m$, we define $\mathcal{L}_k$ along with $\Lambda$ and $\Sigma$ 
{\tiny
\begin{equation}
\mathcal{L}_k \bydef \begin{pmatrix}
     1  & -2   & 2   &  -2     & \dots     &~ &~\\
     \mu_k  & 2  & -\mu_k       & ~      &~ &~\\
    ~  & \mu_k   & 4    & -\mu_k      &~ &~\\
    ~  & ~   & \ddots& \ddots  & \ddots  &~&~ \\
    ~  & ~ & ~ & \mu_k   & 2j    & -\mu_k &~\\
     ~  & ~   & ~ & ~ &  \ddots& \ddots  & \ddots
    \end{pmatrix}, ~~
    \Lambda \bydef \begin{pmatrix}
    0   & 0   & 0 & \dots  & ~\\
    - 1  & 0  & 1       & ~      \\
    ~  & -1  & 0    & 1      \\
    ~  & ~   & \ddots& \ddots  & \ddots \\
    \end{pmatrix},~~ 
\Sigma \bydef \begin{pmatrix}
     0    & 0    &\dots       & ~        \\
     1    & ~    & ~       & ~           \\
    ~     & 1    & ~     & ~      & ~  \\
    ~     & ~    & \ddots & ~       \\
    \end{pmatrix}\label{eq : def of sigma and lambda},
\end{equation}
}
as operators on sequences indexed on $\mathbb{N}_0.$ Using such notations, we obtain that \eqref{eq : pde with Fourier and cheb pointwise} is equivalent to
\begin{align}
    \mathcal{L}_k U_k +  \frac{h}{2}\Lambda Q_k(U) + \frac{h}{2} \Lambda \Phi_k - \beta_k = 0
\end{align}
for all $k \in \mathbb{Z}^m$ where $Q_k(U) \bydef \left(Q_{k,n}(U)\right)_{n \in \mathbb{N}_0}$,  $\Phi_k \bydef (\phi_{k,n})_{n \in  \mathbb{N}_0}$  and $\beta_k \bydef (\beta_{k,n})_{n \in  \mathbb{N}_0}$ with
\begin{equation}\label{def : definition of beta}
    \beta_{k,n} \bydef \begin{cases}
      b_k &\text{ if } n=0\\
    0 &\text{ if } n \geq 1.
\end{cases}
\end{equation}
Using \eqref{eq : positivity of muk}, we will show in Section \ref{sec : computation of C0 C1} that $\mathcal{L}_k$ is invertible for all $k \in \mathbb{Z}^m$. In particular, this allows us to define the map $F$ as 
\begin{align*}
    F_k(U) \bydef  U_k +  \frac{h}{2} \mathcal{L}_k^{-1}\Lambda Q_k(U) + \frac{h}{2}\mathcal{L}_k^{-1}\Lambda \Phi_k - \mathcal{L}_k^{-1}\beta_k
\end{align*}
for all $k \in \mathbb{Z}^m$. Then, solutions to \eqref{eq : pde with Fourier and cheb pointwise} are equivalently zeros of $F$. In particular, $F$ can be written in the following condensed form
\begin{align}\label{def : definition of F}
    F(U) = U   + \frac{h}{2}\mathcal{L}^{-1}\mathbf{\Lambda} Q(U)  + \frac{h}{2} \mathcal{L}^{-1}\mathbf{\Lambda} \Phi - \mathcal{L}^{-1}\beta
\end{align}
where 
\begin{align}
    \left(\mathcal{L}U\right)_k \bydef \mathcal{L}_k U_k ~~ \text{ and } ~~ \left(\mathbf{\Lambda} U\right)_k \bydef \Lambda U_k, \qquad \text{for all } k \in \mathbb{Z}^m.
\end{align}

Now that we wrote an equivalent problem in terms of the coefficients of $U$, we define an Banach space for sequences on which we will define $F$. Let $s_1,s_2 \geq 0$ and $\nu \geq 1$, then define the weight $\omega=(\omega_{k,n})$ as 
\begin{align}\label{def : weights}
    \omega_{k,n} \bydef \begin{cases}
        (1+ s_1|k|_1 )^{s_2}\nu^{|k|_1} &\text{ if } n=0\\
        2(1+ s_1|k|_1 )^{s_2}\nu^{|k|_1} &\text{ if } n \geq 1
    \end{cases}
\end{align}
for all $k \in \mathbb{Z}^m$ and $n \in \mathbb{N}_0$, where $|k|_1 \bydef |k_1| + |k_2| + \dots + |k_m|$. In particular, we have
\begin{align}\label{eq : weights inequality}
    \omega_{k+l,n} \leq  \omega_{l,n} \omega_{k,n}, \quad \text{for all } l,k \in \mathbb{Z}^m.
\end{align}

First, let us define $\ell^1 \bydef \ell^1(\mathbb{Z}^m \times \mathbb{N}_0)$ be the usual Lebesgue space for sequences indexed on $\mathbb{Z}^m \times \mathbb{N}_0$ associated to its usual norm $\|\cdot\|_1$ given by
\begin{align*}
    \|V\|_1 \bydef \sum_{(k, n) \in \mathbb{Z}^m \times \mathbb{N}_0} |V_{k,n}|,
\end{align*}
for all $V = (V_{k,n})_{(k, n) \in \mathbb{Z}^m \times \mathbb{N}_0} \in \ell^1.$ Moreover, denote $\ell^1(\mathbb{N}_0)$ the usual Lebesgue space for sequences indexed on $\mathbb{N}_0$. In particular, for each $V \in \ell^1$, we have that $V = (V_k)_{k \in \mathbb{Z}^m}$ and $V_k \in \ell^1(\mathbb{N}_0)$ for all $k \in \mathbb{Z}^m$.
Then, define the Banach space $X$ as 
\begin{align}\label{def : banach space X}
    X \bydef \left\{U = (U_{k,n})_{(k,n) \in \mathbb{Z}^m\times \mathbb{N}_0} \in \ell^1 : \|U\|_{1,\omega} < \infty\right\}, ~~ \text{ where } \|U\|_{1,\omega} \bydef \sum_{k \in \mathbb{Z}^m}\sum_{n \in \mathbb{N}_0}   \left|U_{k,n}\right| \omega_{k,n}.
\end{align}
In particular,  defining $W_0 : \ell^1(\mathbb{N}_0) \to \ell^1(\mathbb{N}_0)$ as the following bounded linear operator 
\begin{align}\label{def : definition of W0}
    (W_0v)_n = \begin{cases}
        v_0 &\text{ if } n=0\\
        2v_n &\text{ if } n>0,
    \end{cases}
\end{align}
we have that 
\begin{equation}\label{eq : indentity on the 1 omega norm}
    \|U\|_{1,\omega} \bydef \sum_{k \in \mathbb{Z}^m} \omega_{k,0} \|W_0U_k\|_1
\end{equation}
for all $U = (U_k)_{k \in \mathbb{Z}^m} \in X$. 
Using \eqref{eq : weights inequality}, we obtain that $X$ is a Banach algebra under the discrete convolution, that is
\begin{align*}
    \|U*V\|_{1,\omega} \leq \|U\|_{1,\omega} \|V\|_{1,\omega} ~ \text{ for all } U, V \in X, ~~ \text{ where }  \left(U*V\right)_{k,n} \bydef \sum_{j \in \mathbb{Z}^m}\sum_{i \in \mathbb{Z}} U_{k-j,|n-i|}V_{j,|i|}.
\end{align*}
Finally, given a bounded linear operator $B : X \to X$,  $\|\cdot\|_{1,\omega}$ provides a natural operator norm for $B$:
\begin{align}\label{eq : operator norm on X}
    \|B\|_{1,\omega} \bydef \sup_{U \in X, \|U\|_{1,\omega} =1} \|BU\|_{1,\omega} = \max_{k \in \mathbb{Z}^m, n \in \mathbb{N}_0} \sum_{j=0}^\infty \sum_{p \in \mathbb{Z}^m} \frac{\omega_{p,j}}{\omega_{k,n}} |B_{(p,j),(k,n)}|,
\end{align}
where $B$ is seen as a matrix $(B_{(p,j),(k,n)})_{(p,j),(k,n) \in \mathbb{Z}^m \times \mathbb{N}_0}.$

Using the above properties, we wish to solve \eqref{eq : initial pde} by equivalently looking for a zero of $F$ in the Banach space $X$. In order to do so, we want to derive a Newton-Kantorovich approach (cf. Section \ref{ssec : Newton Kantorovich}) and look for solutions as fixed-point of a well-chosen map. In particular, the success of such an approach depends on the computation of specific quantities.



\subsection{A proof of existence using a Newton-Kantorovich approach}\label{ssec : Newton Kantorovich}

Let us assume that we have access to an approximate zero $\overline{U} \in X$ of $F$ and that we want to prove the existence of a true zero in a vicinity of $\overline{U}$. The theorem below provides sufficient conditions under which such a statement can be verified.

\begin{theorem}[\bf Newton-Kantorovich Theorem] \label{th : radii polynomial}
     Let $A$ be an injective linear operator such that $AF : X \to X$ is Fréchet differentiable. Let $Y, {Z}_1, r^*$   be non-negative constants and ${Z}_2 : [0,r^*] \to [0,\infty)$ be a non-negative function such that for all $0 \leq r \leq r^{*}$
  \begin{align}\label{eq: definition Y0 Z1 Z2}
  \nonumber
   \|AF(\overline{U})\|_{1,w} &\leq {Y}\\
    \|I_d - ADF(\overline{U})\|_{1,w} &\leq {Z}_1\\\nonumber
   \|A(DF(\overline{U}+V)-DF(\overline{U}))\|_{1,w} &\leq {Z}_2(r)r, ~~ \text{for all } V \in \overline{B_r(0)}.
\end{align}  
If there exists $0 \leq r \leq r^*$ such that
\begin{equation}\label{eq : condition contraction}
    \frac{1}{2}{Z}_2(r)r^2 - (1-{Z}_1)r + {Y}_0 < 0 \text{ and } {Z}_1 + {Z}_2(r) r < 1
 \end{equation}
then there exists a unique zero of $F$ in $\overline{B_r(\overline{U})}$.
\end{theorem}

In practice, the operator $A$ (introduced in the previous result) is chosen as an approximate inverse of $DF(\oU)$. Indeed, note that if \eqref{eq : condition contraction} is satisfied, then $Z_1 < 1$ and, using a Neumann series argument, we obtain that $DF(\oU)$ is invertible.
In what follows, we describe its construction, combining a finite numerical projection and an infinite tail.

Let us fix $N =  (N_1, N_2) \in \mathbb{N}^2$. $N_1$ represents the numerical size of our Fourier series expansion and $N_2$ the one of the Chebyshev series part.  Then,  define the following projection operators
 \begin{align}\label{def : projection operators}
    \left(\pi^{\leq N}(V)\right)_{k,n}  =  \begin{cases}
          V_{k,n},  & (k,n) \in I^{N} \\
              0, & (k,n) \notin I^{N}
    \end{cases} ~~ \text{ and } ~~
     \left(\pi^{> N}(V)\right)_{k,n}  =  \begin{cases}
          0,  & (k,n) \in I^{N} \\
              V_{k,n}, & (k,n) \notin I^{N}
    \end{cases}
 \end{align}
  for all $V =  (V_{k,n})_{(k, n) \in \mathbb{Z}^m \times \mathbb{N}_0} \in \ell^1$, where   $$I^{N} \bydef \left\{(k,n) \in 
 \mathbb{Z}^m \times  \mathbb{N}_0 :  |k|_\infty \leq N_1\text{ and }  n \leq N_2\right\},$$
 and $|k|_\infty \bydef \max\{|k_1|, \dots, |k_m|\}$.
In addition, we define the projectors $\pi^{\leq N_1}$, $\pi^{> N_1}$, $\pi^{\leq N_2}$ and $\pi^{> N_2}$ as 
 \begin{align*}\label{def : projection operators special cases}
    \left(\pi^{\leq {N}_1}(V)\right)_{k,n}  =  \begin{cases}
         V_{k,n} ,  & \text{ if } n \leq N_1 \\
            0 , & \text{ if } n > N_1
    \end{cases} ~~ &\text{ and } ~~
     \left(\pi^{\leq N_2}(V)\right)_{k,n}  =  \begin{cases}
          V_{k,n},  & \text{ if } |k|_\infty \leq N_2 \\
              0,  & \text{ if } |k|_\infty > N_2
    \end{cases}\\
    \pi^{>N_1} \bydef I_d - \pi^{\leq {N}_1} ~~ &\text{ and }  ~~ \pi^{>N_2} \bydef I_d - \pi^{\leq {N}_2}.
 \end{align*}
 In particular, notice that 
 \begin{equation}\label{eq : identity on the projectors}
     \pi^{>N} = \pi^{>N_1} + \pi^{\leq N_1}\pi^{>N_2}.
 \end{equation}
 Furthermore, given a linear operator $B : X \to X$, we denote by $B_{col(k,n)}$ ($(k,n) \in \mathbb{Z}^m \times \mathbb{N}_0$), the $(k,n)^{th}$ column of $B$, that is 
\begin{equation*}
    (B_{col(k,n)})_{j,i} = B_{(j,i),(k,n)}
\end{equation*}
for all $(j,i) \in \mathbb{Z}^m \times \mathbb{N}_0$, where $B$ is seen as an infinite matrix with entries $(B_{(j,i),(k,n)})$. Similarly, for a linear operator  $B : \ell^1(\mathbb{N}_0) \to \ell^1(\mathbb{N}_0)$, we denote by $B_{col(n)}$  the $n^{th}$ column of $B$. Finally, we denote by $(e_i)_{i \in \mathbb{N}_0}$ the standard basis for $\ell^1(\mathbb{N}_0)$, that is $(e_i)_j = \delta_{i,j}$ where $\delta_{i,j}$ is the usual Kronecker delta.

We know focus on the construction of the approximate inverse $A$. First, notice that 
\begin{equation}\label{def : operator K}
    DF(\oU) = I_d + \mathcal{K}, ~~ \text{ where } ~ \mathcal{K} \bydef \frac{h}{2}\mathcal{L}^{-1}\mathbf{\Lambda} DQ(\oU).
\end{equation}
One of the main objective of the present manuscript is to prove that $\mathcal{K} : X \to X$ is a compact operator, and consequently $A$, the approximate inverse of $DF(\oU)$ in Theorem \ref{th : radii polynomial},  can be chosen as a compact perturbation of the identity. Specifically, we construct numerically $A_0$ such that 
\begin{align*}
    A_0 \approx \pi^{\leq N}\left(I_d + \mathcal{K}\right)^{-1}\pi^{\leq N}.
\end{align*}
We choose $A_0$ such that $A_0 = \pi^{\leq N}A_0\pi^{\leq N}$, that is $A_0$ can be seen as a matrix. Then, we define $A_\infty$ as 
\begin{equation}
    A_\infty \bydef \pi^{>N}(I_d - \mathcal{K})\pi^{>N}.
\end{equation}
In particular, $A_\infty$ is an approximate inverse for $\pi^{> N}\left(I_d + \mathcal{K}\right)\pi^{> N}$. Indeed, we have 
\begin{equation}\label{eq : identity A infinity}
    A_\infty \pi^{> N}\left(I_d + \mathcal{K}\right)\pi^{> N} = \pi^{>N} - \left(\pi^{>N}\mathcal{K}\pi^{>N}\right)^2,
\end{equation}
and if $\pi^{>N}\mathcal{K}\pi^{>N}$ is small, then the above justifies the choice for $A_\infty.$
Using $A_0$ and $A_\infty$, we construct $A$ as follows
\begin{align}\label{def : operator A}
    A \bydef A_0 + A_\infty - A_\infty \mathcal{K} A_0 - A_0\mathcal{K}A_\infty. 
\end{align}
Note that the ``tail'' of the operator $A$ is given by $A_\infty$ and is a compact perturbation of the identity.

Finally, we assume that $\oU$ satisfies
\begin{align}
    \oU = \pi^{\leq N}\oU,
\end{align}
that is $\oU$ can be seen as a finite-dimensional vector. In fact, $\oU$ is obtained numerically as an approximate zero of $F$.
We demonstrate in the next section that the specific structures of $\overline{U}$ and $A$ provide computable bounds $Y, Z_1$ and $Z_2$. Specifically, such bounds will be computed thanks to rigorous numerics.

\subsection{Derivation of the bounds for the Newton-Kantorovich Theorem}\label{sec : computation of the bounds}

In this subsection, we present some analysis required for the computations for the bounds involved in Theorem~\ref{th : radii polynomial}. Specifically, we emphasize that, under the construction of $\overline{U}$ and $A$ as described earlier, these bounds can be determined through calculations involving finite-dimensional objects. These computations can be rigorously performed using numerical methods, enabling verification of condition \eqref{eq : condition contraction} and providing a basis for an existence proof. Explicit formulas are provided for estimating each bound, along with detailed quantities related to the operator $\mathcal{L}^{-1}\mathbf{\Lambda}$, which plays a crucial role in our computer-assisted analysis. In particular, we show that the derived estimates can be explicitly computed, subject to controlling the decay induced by the operator $\mathcal{L}^{-1}\mathbf{\Lambda}$. 

More specifically, we are interested in non-negative and uniformly bounded functions $C_0, C_{1,n}, C_{A_0,n} : [0,\infty) \to (0,\infty)$ (cf. Section \ref{sec : computation of C0 C1}) such that 
\begin{equation}\label{eq : ineq C0 C1n CA0n}
    \begin{aligned}
    \|W_0 \mathcal{L}_k^{-1} \Lambda W_0^{-1}\|_1 \leq \frac{C_0(\mu_k)}{2+\mu_k}, ~~ &\text{ for all } k \in \mathbb{Z}^m \\
    \frac{1}{2-\delta_{0,n}}\sum_{j = 0}^\infty \left|\left(W_0\mathcal{L}_k^{-1} \Lambda\right)_{j,n}\right| \leq \frac{C_{1,n}(\mu_k)}{n+ 1+ \sqrt{n^2+\mu_k^2}},  ~~ &\text{ for all } |k|_\infty \leq N_1\\
    \sum_{j = 0}^\infty  \sum_{p \in \mathbb{Z}^m, |p|_\infty \leq N_1} \frac{\omega_{p,j} }{\omega_{k,n} }\left|\left(A_0\mathcal{L}^{-1} \mathbf{\Lambda}\right)_{(p,j),(k,n)}\right| \leq \frac{C_{A_0,n}(\mu_k)}{n+ 1+ \sqrt{n^2+\mu_k^2}}, ~~ &\text{ for all } |k|_\infty \leq N_1 
\end{aligned} 
\end{equation}
for  all $n \in \mathbb{N}_0$, and where $\delta_{0,n}$ is the usual Kronecker delta. Let
$\mathcal{D}_{C_0}$, $\mathcal{D}_{C_1}$ and $\mathcal{D}_{A_0}$ be diagonal linear operators defined as 
\begin{equation}\label{def : operators DC0 and DC1}
    (\mathcal{D}_{C_0}U)_{k,n} \bydef \frac{C_0(\mu_k) U_{k,n}}{2+\mu_k}, ~~
    (\mathcal{D}_{C_{1}}U)_{k,n} \bydef \frac{C_{1,n}(\mu_k)  U_{k,n}}{n+1+\sqrt{n^2+\mu_k^2}}, ~~
    (\mathcal{D}_{A_0}U)_{k,n} \bydef \frac{C_{A_0,n}(\mu_k) U_{k,n}}{n+1+\sqrt{n^2+\mu_k^2}}      
\end{equation}
for all $k \in \mathbb{Z}^m$ and $n \in \mathbb{N}_0$, where we fix $C_{1,n}(\mu_k) = C_{A_0,n}(\mu_k) = 1$ for all $|k|_\infty > N_1$. Such diagonal operators will be useful in our analysis as they allow to estimate the operator $\mathcal{L}^{-1}\mathbf{\Lambda}$ column-wise. In particular, combining the properties the norm $\|\cdot\|_{1,\omega}$ in \eqref{eq : indentity on the 1 omega norm} and \eqref{eq : operator norm on X} with \eqref{eq : ineq C0 C1n CA0n} and \eqref{def : operators DC0 and DC1}, we have 
\begin{align}\label{eq : identity on DC0 and DC1}
    \|\mathcal{L}^{-1}\mathbf{\Lambda} \mathcal{D}_{C_0}^{-1}\|_{1,\omega} \leq 1, ~\|\mathcal{L}^{-1}\mathbf{\Lambda} \pi^{\leq N_1}\pi^{>N_2} \mathcal{D}_{C_1}^{-1}\|_{1,\omega} \leq 1 ~ \text{ and } ~ \|A_0\mathcal{L}^{-1}\mathbf{\Lambda} \mathcal{D}_{A_0}^{-1}\|_{1,\omega} \leq 1. 
\end{align}

The above quantities, once computed explicitly, provide a computer-assisted strategy for the application of Theorem \ref{th : radii polynomial}.
This process is first illustrated by the following result, which provides an explicit and computable bound $Z_1$ satisfying \eqref{eq: definition Y0 Z1 Z2}.

\begin{lemma}\label{lem : Z1 bound}
Let $Z_0, Z_\infty, Z_{\leq N}$ and $Z_{A_0}$ be defined as 
\begin{align*}
Z_0 &\bydef \| \pi^{\leq N} - A_0(I_d +\mathcal{K})\pi^{\leq N}\|_{1,\omega} \\
Z_\infty & \bydef \frac{h}{2}\max\bigg\{\|\pi^{>N_1}\mathcal{D}_{C_0}DQ(\oU)\pi^{>N_1}\|_{1,\omega} + \|\pi^{\leq N_1}\pi^{>N_2}\mathcal{D}_{C_1}DQ(\oU)\pi^{>N_1}\|_{1,\omega},\\
    &~~~~~~~~~~~~~~~\|\pi^{>N_1}\mathcal{D}_{C_0}DQ(\oU)\pi^{\leq N_1}\pi^{>N_2}\|_{1,\omega} + \|\pi^{\leq N_1}\pi^{>N_2}\mathcal{D}_{C_1}DQ(\oU)\pi^{\leq N_1}\pi^{>N_2}\|_{1,\omega}\bigg\}\\
     Z_{\leq N} &\bydef \frac{h}{2}\|\pi^{>N_1}\mathcal{D}_{C_0}DQ(\oU)\pi^{\leq N}\|_{1,\omega} + \frac{h}{2}\|\pi^{\leq N_1}\pi^{>N_2}\mathcal{D}_{C_1}DQ(\oU)\pi^{\leq N}\|_{1,\omega}\\
     Z_{A_0} &\bydef \frac{h}{2}\max\left\{ \|\pi^{\leq N_1}\mathcal{D}_{A_0} DQ(\oU)\pi^{>N_1}\|_{1,\omega}, ~ \|\pi^{\leq N_1}\mathcal{D}_{A_0}  DQ(\oU)\pi^{\leq N_1}\pi^{>N_2}\|_{1,\omega}\right\}.
\end{align*}
Then 
\begin{equation}\label{ineq : Z infinity Z N ZA0}
   \|\pi^{>N}\K\pi^{> N}\|_{1,\omega} \le Z_\infty, \quad
\|\pi^{>N}\K\pi^{\leq N}\|_{1,\omega} \le Z_{\leq N}, \quad \text{and} \quad
\|A_0\K\pi^{>N}\|_{1,\omega} \le Z_{A_0}, 
\end{equation}
and letting $Z_1$ be defined as 
\begin{align}\label{def : definition of Z1}
    Z_1 \bydef \max\left\{Z_\infty^2 + (1+Z_\infty)Z_{\leq N}Z_{A_0} +  Z_{A_0}Z_\infty^2, ~ Z_{0}(1+Z_{\leq N} + Z_{\leq N}Z_{\infty}) + Z_{\leq N}Z_{A_0}(1+Z_\infty)  \right\}
\end{align}
 we have $\|I_d - ADF(\overline{U})\|_{1,\omega} \leq Z_1.$
\end{lemma}
\begin{proof}
We start by proving \eqref{ineq : Z infinity Z N ZA0}. Using the definition of the operator norm \eqref{eq : operator norm on X} and \eqref{eq : identity on the projectors}, we have 
\begin{align*}
    \|\pi^{>N}\K\pi^{>N}\|_{1,\omega} = \max\left\{\|\pi^{>N}\K\pi^{>N_1}\|_{1,\omega}, \|\pi^{>N}\K\pi^{\leq N_1}\pi^{>N_2}\|_{1,\omega}\right\}.
\end{align*}
Focusing on the first term $\|\pi^{>N}\K\pi^{>N_1}\|_{1,\omega}$, we get
\footnotesize{
\begin{align*}
    &\|\pi^{>N}\K\pi^{>N_1}\|_{1,\omega} = \frac{h}{2}\|\pi^{>N}\mathcal{L}^{-1}\mathbf{\Lambda}DQ(\oU)\pi^{>N_1}\|_{1,\omega} \\
    &\leq \frac{h}{2}\|\pi^{>N_1}\mathcal{L}^{-1}\mathbf{\Lambda}DQ(\oU)\pi^{>N_1}\|_{1,\omega} + \frac{h}{2}\|\mathcal{L}^{-1}\mathbf{\Lambda}\pi^{\leq N_1}\pi^{>N_2}DQ(\oU)\pi^{>N_1}\|_{1,\omega}\\
    &\leq  \frac{h}{2}\|\pi^{>N_1}\mathcal{L}^{-1}\mathbf{\Lambda}\mathcal{D}_{C_0}^{-1}\|_{1,\omega}\|\pi^{>N_1}\mathcal{D}_{C_0}DQ(\oU)\pi^{>N_1}\|_{1,\omega} + \frac{h}{2}\|\mathcal{L}^{-1}\mathbf{\Lambda}\mathcal{D}_{C_1}^{-1}\pi^{\leq N_1}\pi^{>N_2}\|_{1,\omega}\|\pi^{\leq N_1}\pi^{>N_2}\mathcal{D}_{C_1}DQ(\oU)\pi^{>N_1}\|_{1,\omega}
\end{align*}
}
\normalsize
where we used that $\mathcal{L}^{-1}$, $\mathbf{\Lambda}$ and $\mathcal{D}_{C_0}$ all commute with $\pi^{>N_1}$ and $\pi^{\leq N_1}$.
But using \eqref{eq : identity on DC0 and DC1}, we obtain that 
\begin{align*}
    \|\pi^{>N}\K\pi^{>N_1}\|_{1,\omega}  \leq \frac{h}{2}\|\pi^{>N_1}\mathcal{D}_{C_0}DQ(\oU)\pi^{>N_1}\|_{1,\omega} + \frac{h}{2}\|\pi^{\leq N_1}\pi^{>N_2}\mathcal{D}_{C_1}DQ(\oU)\pi^{>N_1}\|_{1,\omega}.
\end{align*}
On the other hand, we have 
\begin{align*}
    \|\pi^{>N}\K\pi^{\leq N_1}\pi^{>N_2}\|_{1,\omega} &\leq \|\pi^{>N_1}\K\pi^{\leq N_1}\pi^{>N_2}\|_{1,\omega} + \|\pi^{\leq N_1} \pi^{>N_2}\K\pi^{\leq N_1}\pi^{>N_2}\|_{1,\omega}\\
    &\leq \frac{h}{2}\|\pi^{>N_1}\mathcal{D}_{C_0}DQ(\oU)\pi^{\leq N_1}\pi^{>N_2}\|_{1,\omega} + \frac{h}{2}\|\pi^{\leq N_1}\pi^{>N_2}\mathcal{D}_{C_1}DQ(\oU)\pi^{\leq N_1}\pi^{>N_2}\|_{1,\omega}.
\end{align*}
This implies that $\|\pi^{>N}\K\pi^{> N}\|_{1,\omega} \leq Z_\infty$.
A similar reasoning can be applied for the term  $\|\pi^{>N}\K\pi^{\leq N}\|_{1,\omega}$ and  one can easily prove that $\|\pi^{>N}\K\pi^{\leq N}\|_{1,\omega} \leq Z_{\leq N}$. Finally, concerning the term $\|A_0\K\pi^{>N}\|_{1,\omega}$, we have the following
\begin{align*}
    \|A_0\K\pi^{>N}\|_{1,\omega} &=\frac{h}{2} \max\left\{ \|A_0\mathcal{L}^{-1}\mathbf{\Lambda} DQ(\oU)\pi^{>N_1}\|_{1,\omega}, ~ \|A_0\mathcal{L}^{-1}\mathbf{\Lambda} DQ(\oU)\pi^{\leq N_1}\pi^{>N_2}\|_{1,\omega}\right\}\\
    &\leq \frac{h}{2}\max\left\{ \|\pi^{\leq N_1}\mathcal{D}_{A_0} DQ(\oU)\pi^{>N_1}\|_{1,\omega}, ~ \|\pi^{\leq N_1}\mathcal{D}_{A_0}  DQ(\oU)\pi^{\leq N_1}\pi^{>N_2}\|_{1,\omega}\right\} = Z_{A_0}
\end{align*}
using the definition of $\mathcal{D}_{A_0}$ and \eqref{eq : identity on DC0 and DC1}.
\vspace{0.2cm}

Now, we prove that $Z_1$, given in \eqref{def : definition of Z1}, satisfies $\|I_d - A DF(\oU)\|_{1,\omega} \leq Z_1$.
Using the definition of the operator norm $\|\cdot\|_{1,\omega}$ in \eqref{eq : operator norm on X} and recalling \eqref{def : operator K}, we have
\begin{equation}\label{eq : splitting Z1 max}
    \|I_d - ADF(\oU)\|_{1,\omega} = \left\|I_d - A(I_d + \frac{h}{2}\mathcal{L}^{-1}\mathbf{\Lambda} DQ(\oU))\right\|_{1,\omega} = \max\left\{ Z_1^{\le N}, Z_1^{> N} \right\}
\end{equation}
where 
\begin{equation}
Z_1^{\le N} \bydef \left\|\pi^{\leq N} - A(I_d + \mathcal{K})\pi^{\leq N}\right\|_{1,\omega}
\quad \text{and} \quad 
Z_1^{> N} \bydef  \left\|\pi^{>N} - A(I_d + \mathcal{K})\pi^{>N}\right\|_{1,\omega}.
\end{equation}
In the rest of the proof, we bound separately the two quantities $Z_1^{\le N}$ and $Z_1^{>N}$. 
To handle $Z_1^{\le N}$, consider the splitting
\[
\pi^{\leq N} - A(I_d + \mathcal{K})\pi^{\leq N} = 
\left( \pi^{\leq N} - \pi^{\leq N} A(I_d + \mathcal{K})\pi^{\leq N} \right) - \pi^{> N} A(I_d + \mathcal{K})\pi^{\leq N}
\]
where each term is written as 
\begin{align*}
    \pi^{\leq N} - \pi^{\leq{N}}A(I_d + \mathcal{K})\pi^{\leq N} &= \pi^{\leq N} - A_0(I_d + \mathcal{K})\pi^{\leq N}  + A_0 \mathcal{K}A_\infty(I_d + \mathcal{K})\pi^{\leq N}\\
    &= \pi^{\leq N} - A_0(I_d + \mathcal{K})\pi^{\leq N}  + A_0 \mathcal{K}A_\infty\mathcal{K}\pi^{\leq N}\\
    & = \pi^{\leq N} - A_0(I_d + \mathcal{K})\pi^{\leq N} + A_0 \mathcal{K}\pi^{>N}\mathcal{K}\pi^{\leq N} - A_0 \mathcal{K}\pi^{>N}\K\pi^{>N}\mathcal{K}\pi^{\leq N}
\end{align*}
and
{ \small
\begin{align*}
     \pi^{>{N}}A(I_d + \mathcal{K})\pi^{\leq N} &=  A_\infty(I_d + \mathcal{K})\pi^{\leq N} -  A_\infty \mathcal{K} A_0(I_d + \mathcal{K})\pi^{\leq N}\\
     &= \pi^{>N}(I_d - \mathcal{K})\pi^{>N}\mathcal{K}\pi^{\leq N} -  \pi^{>N}(I_d - \mathcal{K})\pi^{>N}\mathcal{K} A_0(I_d + \mathcal{K})\pi^{\leq N} \\
     & = \pi^{>N}(I_d - \mathcal{K})\pi^{>N}\mathcal{K}\pi^{\leq N} 
     - \pi^{>N}(I_d - \mathcal{K})\pi^{>N}\mathcal{K} \pi^{\leq N} \\
     & \quad
     - \pi^{>N}(I_d - \mathcal{K})\pi^{>N}\mathcal{K}(A_0(I_d + \mathcal{K})\pi^{\leq N}-\pi^{\leq N})
     \\
     & = \pi^{>N}(I_d - \mathcal{K})\pi^{>N}\mathcal{K}(\pi^{\leq N} - A_0(I_d+ \mathcal{K})\pi^{\leq N}).
\end{align*}
}
In particular, using \eqref{ineq : Z infinity Z N ZA0}, we find that 
\small{
\begin{align*}
    &\| \pi^{\leq N} - \pi^{\leq{N}}A(I_d + \mathcal{K})\pi^{\leq N}\|_{1,\omega} \\
    &\leq \|\pi^{\leq N} - A_0(I_d + \mathcal{K})\pi^{\leq N}\|_{1,\omega} + \|A_0 \mathcal{K}\pi^{>N}\|_{1,\omega}\|\pi^{>N}\mathcal{K}\pi^{\leq N}\|_{1,\omega} + \|A_0 \mathcal{K}\pi^{>N}\|_{1,\omega}\|\pi^{>N}\K\pi^{>N}\|_{1,\omega}\|\pi^{>N}\mathcal{K}\pi^{\leq N}\|_{1,\omega}\\
    &= Z_0 + Z_{A_0}Z_{\leq N} + Z_{A_0}Z_{\infty}Z_{\leq N}
\end{align*}
}
\normalsize
and similarly
\begin{align*}
     \|\pi^{>{N}}A(I_d + \mathcal{K})\pi^{\leq N}\|_{1,\omega} &\leq \| \pi^{>N}(I_d - \mathcal{K})\pi^{>N}\|_{1,\omega}\|\pi^{>N}\mathcal{K}\pi^{\leq N}\|_{1,\omega}\|\pi^{\leq N} - A_0(I_d+ \mathcal{K})\pi^{\leq N}\|_{1,\omega}\\
     &=  (1+Z_\infty)Z_{\leq N} Z_0.
\end{align*}
In particular, we have obtained that 
\begin{equation}\label{eq : Z1 finite}
    Z_{1}^{\leq N} \leq Z_{0}(1+Z_{\leq N} + Z_{\leq N}Z_{\infty}) + Z_{\leq N}Z_{A_0}(1+Z_\infty).
\end{equation}
To handle $Z_1^{>N}$, we consider the splitting
\[
\pi^{>N} - A(I_d + \mathcal{K})\pi^{>N} = \pi^{>N} - \pi^{>N}A(I_d + \mathcal{K})\pi^{>N}  - \pi^{\leq N}A(I_d + \mathcal{K})\pi^{>N}.
\]
Focusing first on the part $ \pi^{>N} - \pi^{>N}A(I_d + \mathcal{K})\pi^{>N}$ and  using the definition of $A$ in \eqref{def : operator A}, we get 
{\small{
\begin{align*}
    \pi^{>N} - \pi^{>N}A(I_d + \mathcal{K})\pi^{>N} &= \pi^{>N} - A_\infty(I_d + \mathcal{K})\pi^{>N} + A_\infty\mathcal{K}A_0 (I_d + \mathcal{K})\pi^{>N}\\
   & = \pi^{>N} - \pi^{>N}(I_d - \mathcal{K})\pi^{>N}(I_d +\mathcal{K})\pi^{>N} + \pi^{>N}(I_d - \mathcal{K})\pi^{>N}\mathcal{K}A_0 (I_d + \mathcal{K})\pi^{>N}\\
   & = \pi^{>N}\K \pi^{>N}\K\pi^{>N} + \pi^{>N}(I_d - \mathcal{K})\pi^{>N}\mathcal{K}A_0 (I_d + \mathcal{K})\pi^{>N}\\
   & = \pi^{>N}\K \pi^{>N}\K\pi^{>N} + \pi^{>N}(I_d - \mathcal{K})\pi^{>N}\mathcal{K}A_0 \mathcal{K}\pi^{>N},
\end{align*}}}
where we used that $A_0 \pi^{>N} = A_0\pi^{\leq N} \pi^{>N} = 0$ for the last identity. 
On the other hand, using \eqref{eq : identity A infinity},
\begin{align*}
      \pi^{\leq N}A(I_d + \mathcal{K})\pi^{>N} &= A_0(I_d + \mathcal{K})\pi^{>N} - A_0\mathcal{K}A_\infty(I_d + \mathcal{K})\pi^{>N}\\
     & = A_0(I_d + \mathcal{K})\pi^{>N} - A_0\mathcal{K}\pi^{>N}(I_d -\mathcal{K}\pi^{>N} \mathcal{K})\pi^{>N}\\
      &=  A_0\mathcal{K}\pi^{>N}\mathcal{K}\pi^{>N}\mathcal{K}\pi^{>N}
\end{align*}
where we used again that $A_0\pi^{>N} = 0.$ Similarly as the reasoning achieved for $Z^{\leq N}_1$, and using \eqref{ineq : Z infinity Z N ZA0}, we get
\begin{equation}\label{eq : ineq Z1 infinity}
    Z_{1}^{>N} \leq Z_\infty^2 + (1+Z_\infty)Z_{\leq N}Z_{A_0} +  Z_{A_0}Z_\infty^2.
\end{equation}
We conclude the proof combining \eqref{eq : splitting Z1 max}, \eqref{eq : Z1 finite} and \eqref{eq : ineq Z1 infinity}.
\end{proof}

\begin{remark}\label{rem : remark Z1}
    First, notice that $Z_{0}$ is an upper bound for a matrix norm. Such a computation is easily achieved using rigorous numerics and is expected to be very small as $A_0$ is chosen as an approximate inverse for $\pi^{\leq N}+ \pi^{\leq N}\K\pi^{\leq N}$. Then, notice that the bounds $Z_{\leq N}$, $Z_\infty$ and $Z_{A_0}$ all boil down to the computation of an operator norm of a diagonal matrix with decay along the diagonal (namely $\mathcal{D}_{C_0}$, $\mathcal{D}_{C_1}$ or $\mathcal{D}_{A_0}$) multiplied by $DQ(\oU)$. Such a computation has been extensively studied in computer-assisted analysis and we refer to \cite{ unbounded_domain_cadiot, MR3623202,navier-stokes} for illustrations of its treatment.

    Finally, notice that $Z_{\leq N}$, $Z_\infty$ and $Z_{A_0}$ all depend linearly in $h$. In particular, if $Z_0 \ll 1$, then   $Z_1 = \mathcal{O}(h^2)$. This implies that $Z_1$ depends quadratically on $h$, which can be useful when verifying \eqref{eq : condition contraction}.
\end{remark}

\begin{lemma}\label{lem : Y bound}
Let $Z_{\leq N}, Z_{A_0}$ and $Z_\infty$ be given as in Lemma \ref{lem : Z1 bound}. Moreover, let $Y_{i}$ ($i \in \{0,\dots,3\}$) be defined as follows
    \begin{align*}
        Y_{0} &\bydef \|A_0F(\oU)\|_{1,\omega}\\
        Y_1 & \bydef \frac{h}{2}\sum_{|k|_\infty > N_1} \frac{C_0(\mu_k)}{2+\mu_k} \|W_0 (Q(\oU)_k+\Phi_k)\|_1\\
        Y_2 & \bydef \sum_{|k|_\infty > N_1} \omega_{0,k}|b_k|\\
        Y_3 & \bydef  \big\|\pi^{\leq N_1}\pi^{>N_2}\mathcal{L}^{-1}\big(\frac{h}{2}\Lambda (Q(\oU)+\Phi) - \beta\big)\big\|_{1,\omega}
     \end{align*}
    Then defining 
    \begin{align}\label{def : definition of Y0}
        Y \bydef Y_{0}(1 + Z_{\leq N}(1+Z_\infty)) + (Y_1 + Y_2 + Y_3)(1+Z_\infty)( Z_{A_0}+1),
    \end{align}
    we obtain $\|AF(\overline{U})\|_{1,\omega} \leq Y$.
\end{lemma}

\begin{proof}
    First, using the definition of $A$, we obtain the following
    \begin{align*}
        \|AF(\oU)\|_{1,\omega} &= \|A_0 F(\oU) - A_0\mathcal{K} A_\infty F(\oU)\|_{1,\omega} + \|A_\infty F(\oU) - A_\infty\mathcal{K} A_0 F(\oU)\|_{1,\omega}\\
        &\leq \|A_0 F(\oU)\|_{1,\omega}(1+\|A_\infty \mathcal{K}\pi^{\leq N}\|_{1,\omega}) + \|A_\infty F(\oU)\|_{1,\omega}(1 + \|A_0\mathcal{K}\pi^{>N}\|_{1,\omega}).
    \end{align*}
    First, using \eqref{ineq : Z infinity Z N ZA0}, notice that 
    $
    \|A_\infty\|_{1,\omega} = \|\pi^{>N} - \pi^{>N}\K\pi^{>N}\|_{1,\omega} \leq 1+ Z_\infty.
    $
Moreover, using the above and \eqref{ineq : Z infinity Z N ZA0} again, we obtain that 
\begin{align*}
     \|AF(\oU)\|_{1,\omega} \leq  \|A_0 F(\oU)\|_{1,\omega}(1+ Z_{\leq N}(1+Z_{\infty})) + \|\pi^{>N} F(\oU)\|_{1,\omega}(1+Z_{\infty})(1 + Z_{A_0}).
\end{align*}
Since $\oU$ satisfies $\oU = \pi^{\leq N}\oU$ by assumption, we have that $\pi^{>N}\oU = 0.$ In particular, using \eqref{eq : identity on the projectors}, 
\small{
\begin{align}\label{eq : proof Y step 0}
    \|\pi^{>N}F(\oU)\|_{1,\omega} &= \frac{h}{2}\|\pi^{>N}\mathcal{L}^{-1}\left(\Lambda (Q(\oU)+\Phi) - \beta\right)\|_{1,\omega} \\
    &\leq  \frac{h}{2}\|\pi^{>N_1}\mathcal{L}^{-1}\Lambda (Q(\oU)+\Phi)\|_{1,\omega} +  \| \pi^{>N_1}\mathcal{L}^{-1}\beta\|_{1,\omega} + \big\|\pi^{\leq N_1}\pi^{>N_2}\mathcal{L}^{-1}\big(\frac{h}{2}\Lambda (Q(\oU)+\Phi) - \beta\big)\big\|_{1,\omega}.
\end{align}
}
\normalsize
  By definition of $\beta$ in \eqref{def : definition of beta}, we have 
    \begin{align}\label{eq : proof Y2}
        \|\pi^{>N_1}\mathcal{L}^{-1} \beta\|_{1,\omega} = \sum_{|k|_\infty > N_1} \omega_{0,k}|b_k|\|W_0\left(\mathcal{L}_k^{-1}\right)_{col(0)}\|_1 = \sum_{|k|_\infty > N_1} \omega_{0,k}|b_k| = Y_2,
    \end{align}
    where we also used \eqref{eq : formula for the first column} for the last identity.
    Moreover, using \eqref{eq : ineq C0 C1n CA0n}, we obtain
    \begin{align}\label{eq : proof Y1}
       \frac{h}{2} \|\pi^{>N_1}\mathcal{L}^{-1}\Lambda (Q(\oU)+\Phi)\|_{1,\omega} \leq   \frac{h}{2} 
 \sum_{|k|_\infty > N_1} \frac{C_0(\mu_k)}{2+\mu_k} \|W_0 (Q(\oU)_k+\Phi_k)\|_1 = Y_1. 
    \end{align}
   To conclude the proof, we combine \eqref{eq : proof Y step 0}, \eqref{eq : proof Y2} and \eqref{eq : proof Y1}, and get
    \[
     \|\pi^{>N}F(\oU)\|_{1,\omega} \leq Y_1 + Y_2 + Y_3. \quad \qedhere
    \]
\end{proof}

\begin{remark}\label{rem : quantity L Lambda V}
In the bounds $Y_0$ and $Y_2$ of the above lemma, we require the rigorous computation of $\mathcal{L}^{-1}V$
for  a given sequence $V = (V_{k,n})\in X$ with a finite number of nonzero coefficients.  Since $\mathcal{L}$ is diagonal by block in the Fourier component, this is equivalent to computing $\mathcal{L}_k^{-1}V_k$ for a finite number of $k$'s. We present in  Appendix \ref{sec : appendix} such a computation. 
\end{remark}

\begin{lemma}
 Let $\mathcal{D}_{C_0}$ be defined as in \eqref{def : operators DC0 and DC1} and let $r^*$ be fixed. Moreover, let $Z_{\leq N}, Z_{A_0}$ and $Z_{\infty}$ be defined  in Lemma \ref{lem : Z1 bound}. Let $0 \leq r \leq r^*$ and $V \in \overline{B_{r}(0)} \subset X$,  if 
 \small{
\begin{align}\label{def : definition of Z2}
    Z_{2}(r)r &\geq  \frac{h}{2}\|\mathcal{D}_{C_0}\left(DQ(\oU + V) - DQ(\oU)\right)\|_{1,\omega} \max\left\{\|A_0\|_{1,\omega}(1+Z_{\leq N}(1+Z_{\infty})), (1+Z_\infty)(1+Z_{A_0})\right\}
\end{align}
}
\normalsize
    then $\|A\left(DF(\overline{U} + V) - DF(\oU)\right)\|_{1,\omega} \leq Z_2(r)r$.
\end{lemma}

\begin{proof}
Let $0\leq r \leq r^*$ and let $V \in \overline{B_{r}(0)}$.  
Recalling  $DF(\oU + V) = I_d  + \frac{h}{2}\mathcal{L}^{-1}\mathbf{\Lambda} DQ(\oU + V)$, we get
\begin{align*}
    \|A\left(DF(\oU + V) - DF(\oU)\right)\|_{1,\omega}& = \frac{h}{2}\|A\mathcal{L}^{-1}\mathbf{\Lambda}\left(DQ(\oU + V) - DQ(\oU)\right)\|_{1,\omega}\\
    &\leq  \frac{h}{2}\|A\|_{1,\omega} \| \mathcal{L}^{-1}\mathbf{\Lambda} \mathcal{D}_{C_0}^{-1}\|_{1,\omega} \|\mathcal{D}_{C_0}\left(DQ(\oU + V) - DQ(\oU)\right)\|_{1,\omega}\\
    &\leq \frac{h}{2}\|A\|_{1,\omega} \|\mathcal{D}_{C_0}\left(DQ(\oU + V) - DQ(\oU)\right)\|_{1,\omega}
\end{align*}
where we have used \eqref{eq : identity on DC0 and DC1}. It remains to compute $\|A\|_{1,\omega}$. By definition of the operator norm on $X$ \eqref{eq : operator norm on X}, 
\begin{align*}
    \|A\|_{1,\omega} = \max\left\{ \|A\pi^{\leq N}\|_{1,\omega},  \|A\pi^{>N}\|_{1,\omega}\right\}.
\end{align*}
First focusing on $A\pi^{\leq N}$, we use \eqref{def : operator A} to get
\begin{align*}
    \|A\pi^{\leq N}\|_{1,\omega} = \|(I_d - A_\infty \K) A_0\|_{1,\omega} &\leq \|A_0\|_{1,\omega} (1 + \|A_\infty\|_{1,\omega}\|\pi^{>N}\K\pi^{\leq N}\|_{1,\omega})\\
    &\leq \|A_0\|_{1,\omega}(1+ (1+Z_\infty)Z_{\leq N}),
\end{align*}
and where we used \eqref{ineq : Z infinity Z N ZA0} for the last inequality.
Similarly, using \eqref{ineq : Z infinity Z N ZA0} again, we have
\begin{align*}
     \|A\pi^{\leq N}\|_{1,\omega} = \|(\pi^{>N} - A_0 \K )A_\infty\|_{1,\omega} \leq (1 + Z_{A_0})(1+Z_\infty). \qquad \qedhere
\end{align*}
\end{proof}

\subsection{A posteriori regularity} \label{sec:regularity}

In practice, we will look for a zero $\tilde{U}$ of $F$ in $X$, and establish its existence thanks to Theorem \ref{th : radii polynomial}. Notice that the space $X$ ensures some spatial regularity for the associated function $\tilde{u}$ of $\tilde{U}$, but it only provides continuity in time. In this section, we prove that $\tilde{u}$ is actually continuously differentiable in time, leading to a strong solution of the IVP.

Let $X^{1}$ be the Banach space given as 
\begin{equation}
X^{1} \bydef \left\{ U \in X:  \|U\|_{X^{1}} < \infty\right\}, \quad \text{where }
\|U\|_{X^{1}} \bydef \sum_{n \in \mathbb{N}_0}\sum_{\mathbb{Z}^m} (|n| + \omega_{k,n}) |U_{k,n}|,
\end{equation}
and $\omega_{k,n}$ is given in \eqref{def : weights}.
In particular, given $U \in X^{1}$, we have that $U$ has a function representation $u$, given as in \eqref{eq : expansion solution}, which is continuously differentiable in time. 
Looking for a zero $\tilde{U} \in X$ of $F$, we prove in the next lemma that $\tilde{U} \in X^{1}$, establishing the continuous differentiability of $\tilde{u}$ in time. 
\begin{lemma}\label{lem : regularity}
    Let $\tilde{U} \in X$ be a zero of $F$. Then, $\tilde{U} \in X^{1}$ and $\tilde{U}$ has a  function representation $\tilde{u}$ which is continuously differentiable in time. 
\end{lemma}

\begin{proof}
    Using \eqref{eq : proof Y2} and Lemma \ref{lem : norm of first column}, we obtain that $\mathcal{L}^{-1}\beta \in X^{1}$. Now, using Lemma \ref{lem : C0 and C1} and the definition of $C_1^*$, we have that $\mathcal{L}^{-1}\mathbf{\Lambda} V \in X^{1}$ for all $V \in X$. Consequently, $\mathcal{L}^{-1}\mathbf{\Lambda}Q(U) \in X^{1}$ for all $U \in X$ by Assumption \ref{ass : assumption on Q} and $\mathcal{L}^{-1}\mathbf{\Lambda}\Phi \in X^{1}$. Since $\tilde{U} $ is a zero of $F$, we obtain that 
\[
\tilde{U} = - \mathcal{L}^{-1}\Lambda (Q(\tilde{U} )+\Phi)  + \mathcal{L}^{-1}\beta  \in X^{1}. \qedhere
\]
\end{proof}

\begin{remark}\label{rem : regularity}
   If  $\lambda_k = \mathcal{O}(|k|_1^s)$ for some $s \in \mathbb{N}$ and if $U \in X$ is a zero of $F$ with $s_1>0$ and $s_2=s$ in \eqref{def : weights}, then $U$ is the unique zero of $F$ in $X$. This conclusion arises from the uniqueness in  Cauchy problems for parabolic PDEs with elliptic spatial linear part (see for instance \cite{evans2010partial} or \cite{mclean2000strongly}). 
   Moreover, suppose that $\phi=0$ and that $\mathbb{Q}(u)$ can be written as 
   \[
   \mathbb{Q}(u) = \mathbb{L}_Q \mathbb{Q}_0(u),
   \]
   where $\mathbb{L}_Q$ is a linear differential operator with constant coefficients of order smaller than $\mathbb{L}$ and $\mathbb{Q}_0$ only contains products of $u$ (without spatial derivatives involved). Let $\mathcal{L}_Q$ be the Fourier-Chebyshev representations of $\mathbb{L}_Q$, then given $k \in \mathbb{Z}^m$, there exists $q_k \in \mathbb{C}$ such that $(\mathcal{L}_QU)_k = q_k U_k$. In particular, using Assumption \ref{ass : assumption on Q} combined with \eqref{eq : ineq C0 C1n CA0n}, we obtain that there exists $C>0$ such that
   \[
   \|\mathcal{L}_k^{-1}\Lambda q_k\|_1 \leq \frac{C}{|k|_1}
   \]
   for all $k \in \mathbb{Z}^m$. In particular, if $b \in C^\infty(\Omega)$, a bootstrapping argument provides that zeros of $F$ in $\ell^1$ are equivalently smooth ($C^\infty$) solutions to \eqref{eq : initial pde}. We apply this reasoning in the applications of Section \ref{sec : applications}.
\end{remark}

\section{Computation of the functions  \boldmath$C_0 ~$\unboldmath  and \boldmath$C_{1,n}$\unboldmath}\label{sec : computation of C0 C1}

In this section we present our computations for the functions  $C_0$ and $C_{1,n}$ satisfying \eqref{eq : ineq C0 C1n CA0n}. In fact we separate this analysis into two parts : a first part handling a bounded range of values of $\mu$ and an asymptotic part handling values of $\mu$ big enough.

 Let $\mu \geq 0$ be fixed and define  $\mathcal{L}_\mu$ as 
\begin{align}\label{def : definition of L mu}
    \mathcal{L}_\mu \bydef -\mu   \Lambda +  D +   e_0 \theta^T
\end{align}
where $e_0 = (1, 0,0, \dots)^T$, $ \theta = (1,-2,2,-2,2,\dots)^T$. In particular $\mathcal{L}_k = \mathcal{L}_{\mu_k}$.

\subsection{Computation for \boldmath$\mu \in [0,\mu^*]$\unboldmath}\label{ssec : computation for a bounded range of mu}

For this first part, we consider that $\mu \in [0,\mu^*]$ for some $\mu^*$ of our choice. In this case, the computation of $C_0$ and  $C_{1,n}$ is achieved combining the approximation for $\mathcal{L}_k$ proposed in Section 3.1 in \cite{jacek_integration_cheb} and the arithmetic on intervals (e.g. see \cite{MR0657002,MR2807595}).

\begin{figure}[h!]
\begin{center}
\includegraphics[width=12cm]{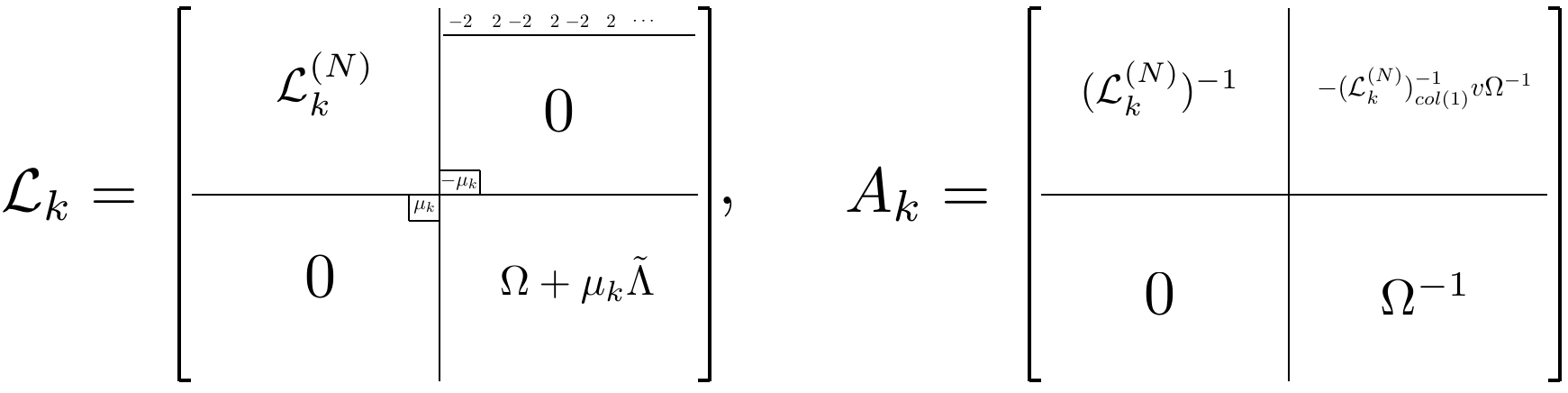}
\end{center}
\vspace{-.2cm}
\caption{The operator $\cL_k$ and its approximate inverse operator $A_k$ (figure taken from \cite{jacek_integration_cheb})}
\label{fig:operators_cL_kand_A_k}
\end{figure}

In particular, given $N$ even, we can decompose $\mathcal{L}_k$ as displayed in Figure \ref{fig:operators_cL_kand_A_k} where 
{\small 
\[
\tilde \Lambda \bydef
\begin{pmatrix} 
0&-1&0&\cdots&\ \\
1&0&-1&0&\cdots\\
0&1&0&-1&\cdots\\
&\ddots&\ddots&\ddots&\ddots&\ddots\\
&\dots&0&1&0&-1 \\
&\ &\dots&\ddots&\ddots&\ddots
\end{pmatrix}
\quad \text{and} \quad
\Omega \bydef
\begin{pmatrix} 
2(N+1)& 0 &0&\cdots  \\
0&2(N+2)&0&0 \\
0&0&2(N+3)&0 \\
\vdots &\vdots& &\ddots&\ddots  \\
\end{pmatrix}.
\]}

In fact, if $N$ is big enough, then one can construct an approximate inverse $A_k$ for $\mathcal{L}_k$ as presented in Figure \ref{fig:operators_cL_kand_A_k}. The value of $N$ can be determined using the following lemma, which we recall from \cite{jacek_integration_cheb}.

\begin{lemma} \label{lem:bounding_invLk_small_k}
Let $N \in \N$ be even and assume that $\cL_k^{(N)} \in M_{N+1}(\R)$ is invertible, and consider the operator $\cL_k$ as in Figure~\ref{fig:operators_cL_kand_A_k}. Let
\begin{align*}
\rho^{(1)} &\bydef \frac{1}{N+1} \left( \|W_0(\cL_k^{(N)})^{-1}_{col(1)}\|_{1} + 1 \right)
\\
\rho^{(2)} &\bydef  \| W_0(\cL_k^{(N)})^{-1}_{col(N+1)} + \frac{1}{N+2} W_0(\cL_k^{(N)})^{-1}_{col(1)} \|_{1} + \frac{1}{N+2} 
\\
\rho^{(3)} &\bydef \frac{2}{(N+1)(N+3)} \|W_0(\cL_k^{(N)})^{-1}_{col(1)} \|_{1} + \frac{1}{N+1} + \frac{1}{N+3} \\
\rho_k &\bydef \frac{|\mu_k|}{2} \max \{ \rho^{(1)}, \rho^{(2)}, \rho^{(3)} \},
\end{align*}
where $(\cL_k^{(N)})^{-1}_{col(1)} \in \R^{N+1}$ denotes the first column of the matrix $(\cL_k^{(N)})^{-1}$. 
%
    If $\rho_k<1$, then $\cL_k$ is  boundedly invertible with
\begin{equation}
 \cL_k^{-1} =  \sum_{j \ge 0} ( I - A_k \cL_k)^j  A_k \quad \text{ and } \quad  \|W_0(I_d - A_k \mathcal{L}_k^{-1})W_0^{-1}\|_1 \leq \rho_k.
\end{equation}
\end{lemma}
Using the above lemma, given $\mu_1 \geq 0$ and $N$ sufficiently big, one can compute a value for $C_0(\mu_1)$ and $C_1(\mu_1)$ based on matrix computations. In particular such computations can easily be performed using the arithmetic on intervals. We expose in \cite{julia_cadiot} the numerical computational details for such analysis.

Now, given $\mu_2 \geq \mu_1$ with $\mu_2$ close to $\mu_1$, we would like to be able to deduce values for $C_0(\mu)$ and $C_1(\mu)$ for all $\mu \in [\mu_1, \mu_2]$. This is achieved in the next lemma.

\begin{lemma}\label{lem : range of values of mu}
    Let $0 \leq \mu_1 \leq \mu_2$ and let $n \in \mathbb{N}_0$.  If $C_0(\mu_1)$ and $C_{1,n}(\mu_1)$ satisfy \eqref{eq : ineq C0 C1n CA0n} for $\mu=\mu_1$ and if
    \begin{align}\label{eq : condition inverse}
        (\mu_2-\mu_1)\frac{C_0(\mu_1)}{2+\mu_1}< 1, 
    \end{align}
   then for all $\mu \in [\mu_1, \mu_2]$ we can define $C_0(\mu)$ and $C_{1,n}(\mu)$ as 
    \begin{align*}
        C_0(\mu) = \frac{\mu+2}{1-\frac{(\mu-\mu_1)C_0(\mu_1)}{\mu_1+2}} \frac{C_0(\mu_1)}{2+\mu_1} ~ \text{ and } ~
         C_{1,n}(\mu)  = \frac{n+ 1+ \sqrt{n^2+\mu^2}}{n+ 1+ \sqrt{n^2+\mu_1^2}} \frac{C_{1,n}(\mu_1)}{1-\frac{(\mu-\mu_1)C_0(\mu_1)}{\mu_1+2}} .
    \end{align*}
    In particular $C_0(\mu)$ and $C_{1,n}(\mu)$ satisfy \eqref{eq : ineq C0 C1n CA0n} for all $\mu \in [\mu_1,\mu_2]$.
\end{lemma}

\begin{proof}
    Given $\mu \in [\mu_1,\mu_2]$, and using the definition of $\mathcal{L}_\mu$ in \eqref{def : definition of L mu}, we have that
    \begin{align*}
        \mathcal{L}_\mu = \mathcal{L}_{\mu_1} + (\mu_1-\mu) \Lambda = \mathcal{L}_{\mu_1} \left(I_d - (\mu-\mu_1) \mathcal{L}_{\mu_1}^{-1}\Lambda\right).
    \end{align*}
    In particular, we have $(\mu-\mu_1) \|W_0\mathcal{L}_{\mu_1}^{-1}\Lambda W_0^{-1}\|_1 \leq \frac{(\mu-\mu_1) C_0(\mu_1)}{2+\mu_1}$ by assumption on $C_0(\mu_1)$. Then under \eqref{eq : condition inverse} we obtain using a Neumann series argument that $\mathcal{L}_\mu$ is boundedly invertible. Moreover, we have 
    \begin{align*}
        \mathcal{L}_{\mu}^{-1} = \sum_{j=0}^\infty (\mu-\mu_1)^j (\mathcal{L}_{\mu_1}^{-1}\Lambda)^{j} \mathcal{L}_{\mu_1}^{-1}. 
    \end{align*}
    This implies that 
    \begin{align}\label{eq : proof neumann step 1}
         (\mu + 2)\|W_0 \mathcal{L}^{-1}_{\mu} \Lambda W_0^{-1}\|_1  &= (\mu + 2)\|W_0 \sum_{j=0}^\infty (\mu-\mu_1)^j (\mathcal{L}_{\mu_1}^{-1}\Lambda)^{j} \mathcal{L}_{\mu_1}^{-1} \Lambda W_0^{-1}\|_1 \nonumber \\
         & = (\mu + 2)\| \sum_{j=0}^\infty (\mu-\mu_1)^j (W_0\mathcal{L}_{\mu_1}^{-1}\Lambda W_0^{-1})^{j} W_0 \mathcal{L}_{\mu_1}^{-1} \Lambda W_0^{-1}\|_1 \nonumber\\
         & \leq \frac{(\mu+ 2)}{1-(\mu-\mu_1)\|W_0 \mathcal{L}^{-1}_{\mu_{1}} \Lambda W_0^{-1}\|_1}\| W_0 \mathcal{L}_{\mu_1}^{-1} \Lambda W_0^{-1}\|_1 \nonumber\\
         & \leq \frac{\mu+2}{1-\frac{(\mu-\mu_1)C_0(\mu_1)}{\mu_1+2}} \frac{C_0(\mu_1)}{2+\mu_1}.
    \end{align}
    Using a similar reasoning, we have 
\begin{align*}
     \frac{n+ 1+ \sqrt{n^2+\mu^2}}{2-\delta_{0,n}}\sum_{j = 1}^\infty \left|\left(W_0\mathcal{L}_\mu^{-1} \Lambda\right)_{j,n}\right| &=  \frac{n+ 1+ \sqrt{n^2+\mu^2}}{2-\delta_{0,n}}\sum_{j = 1}^\infty \left|\left(W_0\sum_{j=0}^\infty (\mu-\mu_1)^j (\mathcal{L}_{\mu_1}^{-1}\Lambda)^{j} \mathcal{L}_{\mu_1}^{-1} \Lambda\right)_{j,n}\right|\\
     &\leq \frac{n+ 1+ \sqrt{n^2+\mu^2}}{2-\delta_{0,n}} \frac{1}{1-\frac{(\mu-\mu_1)C_0(\mu_1)}{\mu_1+2}} \sum_{j = 1}^\infty \left|\left(W_0\mathcal{L}_{\mu_1}^{-1} \Lambda\right)_{j,n}\right|\\
     &\leq \frac{n+ 1+ \sqrt{n^2+\mu^2}}{n+ 1+ \sqrt{n^2+\mu_1^2}} \frac{1}{1-\frac{(\mu-\mu_1)C_0(\mu_1)}{\mu_1+2}} C_{1,n}(\mu_1).
\end{align*}
    where we used \eqref{eq : proof neumann step 1} for the last step.
\end{proof}

Combining Lemma \ref{lem:bounding_invLk_small_k} with Lemma \ref{lem : range of values of mu}, we are able to compute values for $C_0$ and $C_{1,n}$ for a range $[0,\mu^*]$. In practice, we compute an upper bound for $C_{1,n}(\mu_k)$ for a finite values of $n$'s and $k$'s and determine $C_1(\mu_k)$ such that
\[
C_1(\mu_k) \geq  \sup_{n \in \mathbb{N}_0} C_{1,n}(\mu_k).
\]
Then, using the above, we have that given a value for $C_1(\mu_1)$, we can define $C_1(\mu)$ as 
\[
C_1(\mu) = \frac{1+\mu}{1+\mu_1} \frac{C_1(\mu_1)}{1-\frac{(\mu-\mu_1)C_0(\mu_1)}{\mu_1+2}}
\]
for all $\mu \in [\mu_1,\mu_2]$ as long as \eqref{eq : condition inverse} is satisfied. This allows to compute a value for $C_{1,n}(\mu)$ for all $\mu \leq \mu^*$.
Such values are stored numerically and can easily be used when computing the bounds $Y, Z_1$ and $Z_2$ as exposed in the previous section. We display in Figure \ref{fig : C0 C1} the obtained graphs for $C_0$ and $C_1$ using $\mu^* = 3000.$ Note that the values of $C_0$ and $C_1$ for $\mu \geq 3000$ are obtained thanks to Lemma \ref{lem : C0 and C1}.
 \begin{figure}[h!]
  \centering
  \begin{minipage}{.52\textwidth}
   \centering
  \includegraphics[clip,width=1\textwidth]{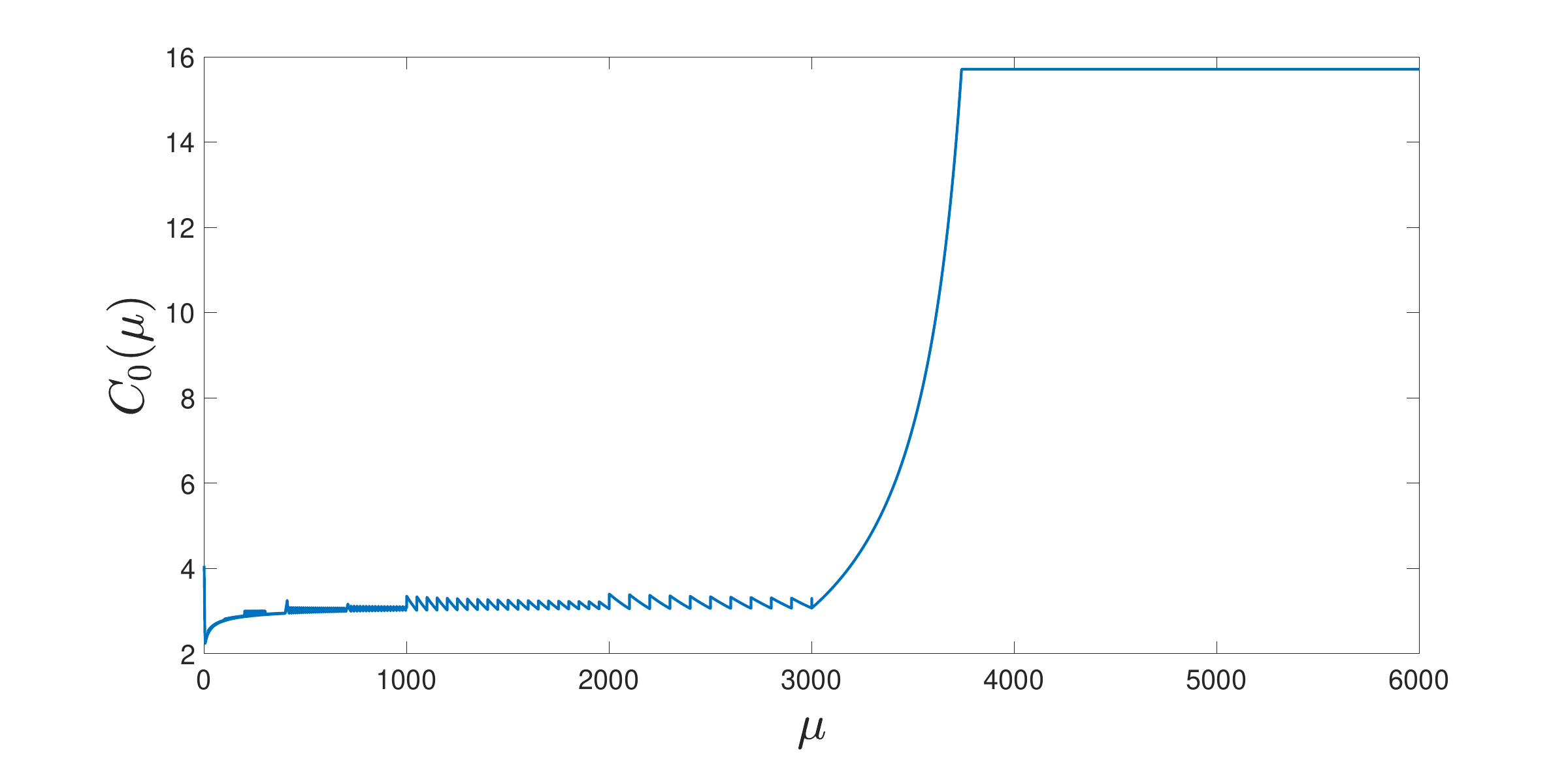}
  \end{minipage}%
  \begin{minipage}{.52\textwidth}
    \centering
   \includegraphics[clip,width=1\textwidth]{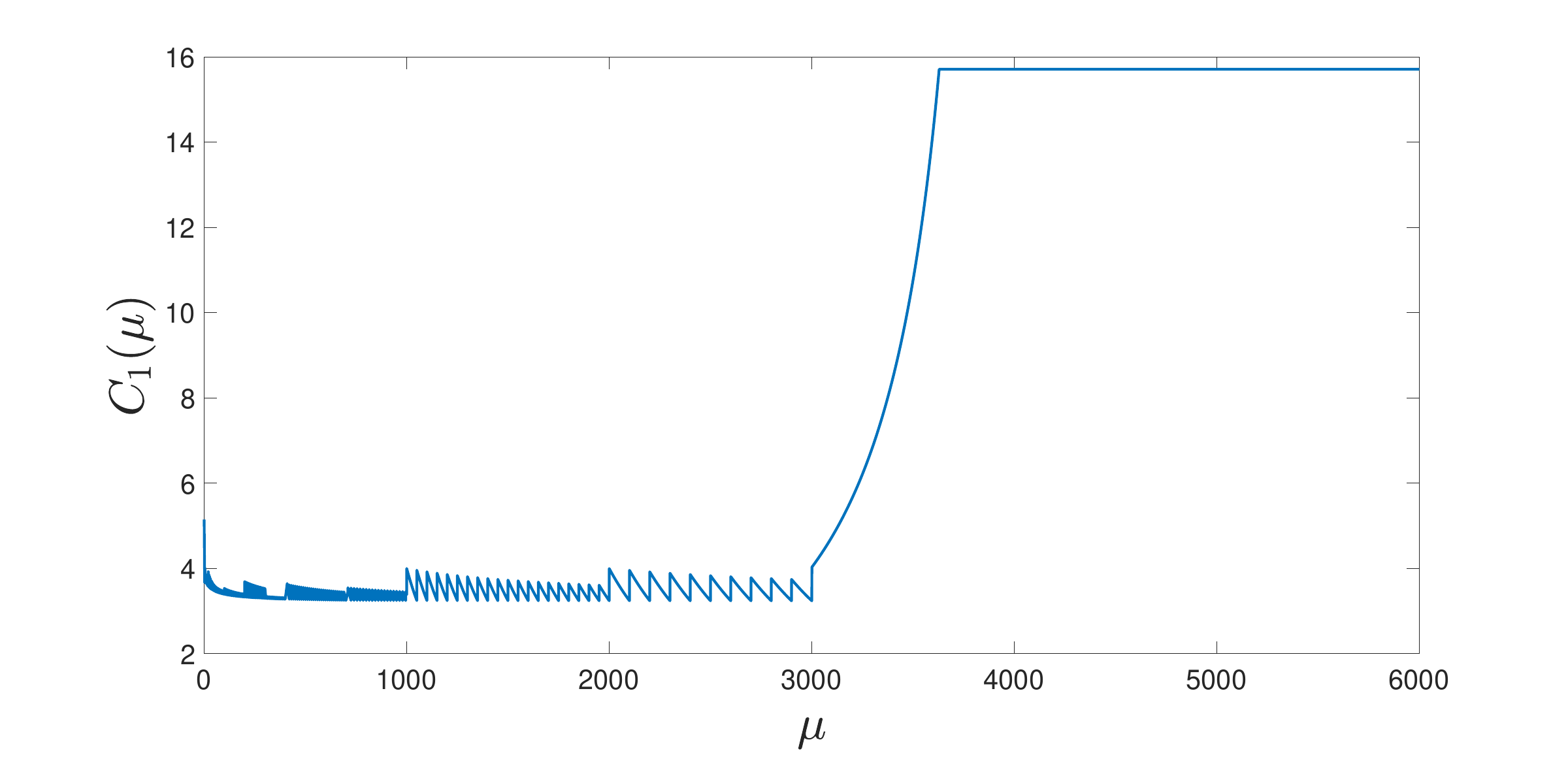}
  \end{minipage}
  \vspace{-.2cm}
  \caption{Representation of the functions $C_0$ (left) and $C_1$ (right).}
    \label{fig : C0 C1}
  \end{figure}
\subsection{Computation for \boldmath$\mu \in (\mu^*,\infty)$\unboldmath}


Let $\mu^* \geq 4$ be fixed. The goal of this section is to prove that there exists constants $C_0^*, C_1^* >0$ such that if we define $C_0(\mu) = C_0^*$ and $C_{1,n}(\mu) = C_1^*$ for all $\mu \geq \mu^*$ and all $n \in \mathbb{N}_0$, then $C_0$ and $C_{1,n}$ satisfy \eqref{eq : ineq C0 C1n CA0n} for all $\mu \geq \mu^*$. In order to do so, we need to analyze the structure of ${\mathcal{L}}_k$. 

Let $\mathcal{D}_\mu$ and $\mathcal{D}_{0,\mu}$ be defined as  
\begin{equation}\label{def : operator D mu}
    (\mathcal{D}_{\mu}U)_{n} = (n+1+\sqrt{n^2+\mu^2})U_n \text{ and }   (\mathcal{D}_{0,\mu}V)_{n} = (n+1+\sqrt{n^2+\mu^2})V_n
\end{equation}
for all $U \in \ell^1(\mathbb{N}_0)$ and all $V \in \ell^1(\mathbb{N})$. In other words, $\mathcal{D}_{0,\mu}$ is the restriction of $\mathcal{D}_\mu$ to $\mathbb{N}$.  Notice that 
\begin{align*}
     \frac{n+ 1+ \sqrt{n^2+\mu^2}}{2-\delta_{0,n}}\sum_{j = 1}^\infty \left|\left(W_0\mathcal{L}_\mu^{-1} \Lambda\right)_{j,n}\right| \leq \|W_0 \mathcal{L}_{\mu}^{-1}\Lambda \mathcal{D}_{\mu} W_0^{-1}\|_1
\end{align*}
for all $n \in \mathbb{N}_0$. In other words, by computing an upper bound for $\|W_0 \mathcal{L}_{\mu}^{-1}\Lambda \mathcal{D}_{\mu} W_0^{-1}\|_1$, we obtain an upper bound for $\sup_{n \in \mathbb{N}_0} C_{1,n}(\mu)$. We use such a property to compute $C_1^*.$

We start by decomposing $\mathcal{L}_\mu$ as follows
{\small
\begin{align*}
    \mathcal{L}_\mu = \begin{pmatrix}
         1 & -2   & 2    & \dots\\
        \mu &   ~\\
        0                & ~\\
        0                & ~ & T_\mu \\
        \vdots           & ~
\end{pmatrix}, \quad \text{where}~~
    T_\mu \bydef \begin{pmatrix}
        2  & -\mu     & ~      &~ &~\\
     \mu   & 4    & -\mu      &~ &~\\
     ~   & \ddots& \ddots  & \ddots  &~&~ \\
 ~ & ~ & \mu   & 2j    & -\mu &~\\
      ~   & ~ & ~ &  \ddots& \ddots  & \ddots
    \end{pmatrix} = -\mu \Lambda_{0} + D_0
\end{align*}}
and where $D_0$ is the diagonal matrix with entries $(2n)_{n \in \mathbb{N}}$ on its diagonal and $\Lambda_{0} = \frac{1}{\mu}\left(D_0-T_\mu\right).$

The inverse of the operator $\mathcal{L}_\mu$ can be expressed using the inverse of $T_\mu$. Moreover, the next lemma provides estimates on the quantities in \eqref{eq : ineq C0 C1n CA0n} depending on norms involving $T_{\mu}^{-1}$. Consequently, this reduces our estimations to the control of the inverse of an infinite tridiagonal matrix.

\begin{lemma}\label{lem : general bound Linv}
Let $\mu \geq \mu^*$ and define $\xi_\mu$ as 
\begin{equation}\label{def : definition of xi mu}
    \xi_\mu \bydef \sum_{i=1}^\infty (-1)^{i+1}\left(T_\mu^{-1}\right)_{i,1} = \|(T^{-1}_\mu)_{col(1)}\|_1.
\end{equation}
We obtain that
    $\mathcal{L}_\mu$ is invertible  and 
   \begin{equation}\label{eq : general inequalities for C0 and C1}
    \begin{aligned}
       (\mu+2) \|W_0 \mathcal{L}_\mu^{-1}  {\Lambda} W_0^{-1}\|_1 &\leq  \frac{\mu+2}{\mu }\max\left\{ 2, ~   \left(\frac{1}{1+2\mu \xi_\mu} + 2\right) \mu\|T^{-1}_\mu\Lambda_0\|_1\right\}\\
       \| W_0  \mathcal{L}_\mu^{-1}  \Lambda \mathcal{D}_{\mu} W_0^{-1}\|_1 &\leq\max\left\{ \frac{2(\mu+1)}{\mu}, ~   \left(\frac{1}{1+2\mu \xi_\mu} + 2\right) \|T^{-1}_\mu\Lambda_0 \mathcal{D}_{0,\mu}\|_1\right\}. 
    \end{aligned}
    \end{equation}  
    Moreover, 
    \begin{align}\label{eq : formula for the first column}
        (\mathcal{L}_{k}^{-1})_{col(0)} = \begin{pmatrix}
        \frac{1}{1+ 2\xi_\mu \mu} \\
        -\frac{\mu}{1+ 2\xi_\mu \mu}(T_\mu^{-1})_{col(1)}
    \end{pmatrix} ~~ \text{ and } ~~ \| W_0(\mathcal{L}_{k}^{-1})_{col(0)}\|_1 = 1.
    \end{align}
\end{lemma}

\begin{proof}
First, observe that the identity \eqref{def : definition of xi mu} is given by \eqref{eq : inverse for Tn}.
Then, notice that $\mathcal{L}_\mu$ can be written as 
\begin{align*}
    \mathcal{L}_\mu = \begin{pmatrix}
        1 & \theta_0^T\\
        \mu {e_0} & T_\mu
    \end{pmatrix}
\end{align*}
where $e_0 = (1,0,0,\dots)^T$ and $\theta_0 = (-2,2,-2,2, \dots)^T$. In fact, the Schur complement of $\mathcal{L}_\mu$ is given by $$1 - \mu \theta_0^T T_\mu^{-1}{e_0} = 1 + 2\mu \sum_{i=1}^\infty (-1)^{i+1} \left(T^{-1}_\mu\right)_{i,1} = 1 + 2 \mu \|(T^{-1}_\mu)_{col(1)}\|_1  = 1+ 2\xi_\mu \mu>0.$$ We  obtain that $\mathcal{L}_\mu$ is invertible and we have the following identity 
\begin{align}\label{eq : inverse of L second edition}
    \mathcal{L}_\mu^{-1} = \begin{pmatrix}
        \frac{1}{1+ 2\xi_\mu \mu} & -\frac{1}{1+ 2\xi_\mu \mu}\theta_0^T T^{-1}_\mu\\
        -\frac{\mu}{1+ 2\xi_\mu \mu}(T_\mu^{-1})_{col(1)}  & T^{-1}_\mu + \frac{\mu}{1+ 2\xi_\mu \mu} (T_\mu^{-1})_{col(1)}\theta_0^T T^{-1}_\mu
    \end{pmatrix}.
\end{align}
In particular, the above yields the formula for the first column of $\mathcal{L}_\mu^{-1}$ in \eqref{eq : formula for the first column}. The norm identity in  \eqref{eq : formula for the first column} is a direct consequence of the definition of $\xi_\mu$.
Now, note that $\Lambda$ can be decomposed as follows 
\begin{align*}
    \Lambda = \begin{pmatrix}
        0 & 0 \\
        -e_0 & \Lambda_0
    \end{pmatrix}.
\end{align*}
Now, using that 
$
\theta_0^T T^{-1}_\mu e_0 = \theta_0^T (T^{-1}_\mu)_{col(1)} = -2\xi_\mu,
$
we obtain
\begin{align}\label{eq : Linv times Lambda}
    \mathcal{L}_\mu^{-1}\Lambda = \begin{pmatrix}
         -\frac{2\xi_\mu}{1+2\xi_\mu\mu} & -\frac{1}{1+2\xi_\mu\mu}\theta_0^T T^{-1}_\mu\Lambda_0\\
         -(1 - \frac{2\mu\xi_\mu}{1+2\xi_\mu\mu}) (T^{-1}_\mu)_{col(1)} &  \left(I_d + \frac{\mu}{1+2\xi_\mu\mu} (T^{-1}_\mu)_{col(1)}\theta_0^T\right) T^{-1}_\mu\Lambda_0
     \end{pmatrix}
\end{align}
Therefore, using that $1-\frac{2\mu \xi_\mu}{1+2\xi_\mu\mu} = \frac{1}{1+2\xi_\mu\mu}$, it yields
\begin{align*}
   W_0 \mathcal{L}_\mu^{-1}\Lambda W_0^{-1} = \begin{pmatrix}
         -\frac{2\xi_\mu}{1+2\xi_\mu\mu} & -\frac{1}{2(1+2\xi_\mu\mu)}\theta_0^T T^{-1}_\mu\Lambda_0\\
         -\frac{2}{1+2\xi_\mu\mu}(T^{-1}_\mu)_{col(1)} &  \left(I_d + \frac{\mu}{1+2\xi_\mu\mu} (T^{-1}_\mu)_{col(1)}\theta_0^T\right) T^{-1}_\mu\Lambda_0
     \end{pmatrix}.
\end{align*}
Then, using that $\|\theta_0\|_{\infty} =2$, the above provides 
\begin{align*}
    \mu \|  W_0 \mathcal{L}_\mu^{-1}\Lambda W_0^{-1}\|_1 \leq \max\left\{\frac{4\mu\xi_\mu}{1+2\xi_\mu\mu}, ~ \left(\frac{1}{1+2\xi_\mu\mu} + \|I_d + \frac{\mu}{1+2\xi_\mu\mu} (T^{-1}_\mu)_{col(1)}\theta_0^T\|_1\right) \mu\|T^{-1}_\mu\Lambda_0\|_1 \right\}\\
    \leq \max\left\{2, ~ \left(\frac{1}{1+2\xi_\mu\mu} + \|I_d + \frac{\mu}{1+2\xi_\mu\mu} (T^{-1}_\mu)_{col(1)}\theta_0^T\|_1\right) \mu\|T^{-1}_\mu\Lambda_0\|_1 \right\}
\end{align*}
since $\frac{4\mu\xi_\mu}{1+2\xi_\mu\mu} \leq 2.$
Moreover, we get
\begin{align}\label{eq : ineq operator norm}
   \| I_d + \frac{\mu}{1+2\xi_\mu\mu} (T^{-1}_\mu)_{col(1)}\theta_0^T\|_1 
   &= \max_{j \in \mathbb{N}} \left\{ |1 + 2\mu (-1)^j (T^{-1}_\mu)_{j,1}| + \sum_{i\neq j} \frac{2\mu}{1+2\xi_\mu\mu} |(T^{-1}_\mu)_{i,1}| \right\} \nonumber\\
   &\leq 1 + \sum_{i \in \mathbb{N}} \frac{2\mu}{1+2\xi_\mu\mu} |(T^{-1}_\mu)_{i,1}|  = 1 + \frac{2\xi_\mu \mu}{1+2\xi_\mu \mu}
   \leq 2.
\end{align}
We conclude the proof using a similar reasoning for $ W_0  \mathcal{L}_\mu^{-1}  \Lambda \mathcal{D}_{\mu} W_0^{-1}$.
\end{proof}

The above lemma provides explicit estimates for \eqref{eq : ineq C0 C1n CA0n} if one is able to compute different matrix norms involving $T_\mu^{-1}$ (e.g. $\|T^{-1}_\mu \Lambda_0\|_1$ or $\| T^{-1}_\mu \Lambda_0 D_{0,\mu}\|_1$). This is achieved in the next section. In particular, we derive sharp estimates for the entries of  $T^{-1}_\mu.$

In the rest of Section \ref{sec : computation of C0 C1}, we fix some $\mu \geq \mu^* \geq 4$ and remove the dependency of $\mu$ of our objects of interest for simplicity. That is we write $T = T_\mu$ and $\xi = \xi_\mu$ for instance.

\subsection{Analysis of the tridiagonal infinite matrix \boldmath$T$\unboldmath}\label{sec : tridiagonal Tk}

We first consider a finite truncation of the matrix $T$, that is, given $n \in \mathbb{N}$, we define the $n \times n$ matrix 
{\small
\begin{equation}\label{def : T(n)}
T(n) \bydef \begin{pmatrix}
     2  & -\mu  & ~       & ~      &~ \\
    \mu  & 4  & -\mu    & ~      &~ \\
    ~     & \ddots& \ddots  & \ddots &~ \\
    ~     & ~     & \mu& 2(n-1)& -\mu \\
    ~     & ~     & ~       & \mu   & 2n
    \end{pmatrix}.
\end{equation}
}
Now, let $(d_j)$ and $(\delta_j)$ be the sequences defined as 
\begin{align}
        \delta_1 &= 2, ~~~~
        \delta_j = 2j + \frac{\mu^2}{\delta_{j-1}}, ~~~~ j = 2, \dots, n\\
    d_n &= 2n, ~~~~
        d_j = 2j + \frac{\mu^2}{d_{j+1}}, ~~~~ j = n-1, \dots, 1.
\end{align}
In particular, note that $d_i, \delta_i >0$ for all $i = 1, \dots, n$. Using \cite{DAFONSECA2000_explicit_inverses_of_some_tridiag} and the fact that $\mu>0$, we obtain that $T(n)$ has an  inverse given by
\begin{align}\label{eq : inverse for Tn}
    \left(T(n)^{-1}\right)_{i,j} =\displaystyle \begin{cases}
    \displaystyle \mu^{j-i} \frac{d_{j+1}\dots d_n}{\delta_i \dots \delta_n} &\text{ if } i \leq j\\
      \displaystyle (-1)^{i+j} \mu^{i-j} \frac{d_{i+1}\dots d_n}{\delta_j \dots \delta_n} &\text{ if } i > j
    \end{cases}
\end{align}
for all $i,j \in \{1, \dots,n\}.$   Moreover, using \cite{DAFONSECA2000_explicit_inverses_of_some_tridiag} again, we notice that 
\begin{align*}
    T(n) = U\mathcal{D}^{-1} U^\dagger
\end{align*}
where $\mathcal{D} = diag((d_i)_{i=1}^n)$ and 
{\small
\begin{align*}
U \bydef \begin{pmatrix}
     d_1  & -\mu  & ~       & ~      &~ \\
    ~  & d_2  & -\mu    & ~      &~ \\
    ~     &~& \ddots  & \ddots &~ \\
    ~     & ~     & ~& d_{n-1}& -\mu \\
    ~     & ~     & ~       & ~   & d_n
    \end{pmatrix},~~
    U^\dagger \bydef \begin{pmatrix}
     d_1  & ~  & ~       & ~      &~ \\
    \mu  & d_2  & ~    & ~      &~ \\
    ~     & \ddots& \ddots  & ~ &~ \\
    ~     & ~     & \mu& d_{n-1}& ~\\
    ~     & ~     & ~       & \mu   & d_n
    \end{pmatrix}.
\end{align*}
}
Consequently, the entries of $T(n)^{-1}$ can also be obtained using the recursive sequence $(d_j)$ only. In particular, $U$ and $U^\dagger$ can be inverted explicitly and direct computations lead to
\begin{align}\label{eq : inverse of T second expression}
    \left(T(n)^{-1}\right)_{i,j} = \begin{cases}
       \displaystyle \sum_{k=1}^i \frac{(-\mu)^{i-k}}{d_k \dots d_i} \frac{\mu^{j-k}}{d_{k+1}\dots d_j} &\text{ if } i \leq j\\
        \displaystyle \sum_{k=1}^j \frac{(-\mu)^{i-k}}{d_k \dots d_i} \frac{\mu^{j-k}}{d_{k+1}\dots d_j} &\text{ if } i > j.
    \end{cases}
\end{align}
In particular, we wish to construct explicit estimates on the elements of the sequence $(d_j)$ in order to control the entries of $T(n)^{-1}$. This is achieved in the next section.

\subsubsection{Explicit estimates for the entries of \boldmath$T(n)^{-1}$\unboldmath}

In this section we first derive a few properties of the sequence $(d_j)$ which will be useful in our analysis.
First, we prove that $(d_j)$ is increasing and we provide an enclosure for each element $d_j.$ In particular, we suppose in this section that 
\begin{equation}\label{eq : assumption on n}
    n > \frac{\mu^2}{4}.
\end{equation}

\begin{lemma}\label{lem : lemma on sequence dn}
   For all $1 \leq j \leq n-1$, we have 
    \begin{align}\label{eq : estimates on dn}
        0 \leq d_{j+1}-d_{j} \leq 2  ~~ \text{ and } ~~
        \alpha_j(\mu) \leq d_{j+1} \leq \alpha_{j+1}(\mu).
    \end{align}
    where 
    \begin{equation}\label{eq : def of alpha j}
        \alpha_j(\mu) \bydef j + \sqrt{j^2 + \mu^2}.
    \end{equation}
    Moreover,
    \begin{align*}
     \sqrt{4+\mu^2}  \leq d_1 \leq  1+ \sqrt{1+\mu^2}.
    \end{align*}
\end{lemma}
\begin{proof}
Let $j \in \{1, \dots, n-1\}$, then observe that
$d_{j} = 2j + \frac{\mu^2}{d_{j+1}}$ and therefore
    \begin{align}\label{ineq : product dj dj+1}
        d_{j}d_{j+1} = 2j d_{j+1} + \mu^2 \geq \mu^2
    \end{align}
    since $d_{j+1} >0.$
    For $j=n-1$, we have
    \begin{align*}
        d_n - d_{n-1} = 2n - \left(2(n-1) + \frac{\mu^2}{2n}\right) =  2 - \frac{\mu^2}{2n}. 
    \end{align*}
    Therefore, it is clear that $0\leq d_n - d_{n-1}\leq 2$ if $n > \frac{\mu^2}{4}$, which is satisfied by \eqref{eq : assumption on n}.  Suppose that 
    \begin{align*}
        0 \leq d_{j+2}-d_{j+1} \leq 2
    \end{align*}
    and for some $1 \leq j \leq n-2$. Then, notice that
    \begin{align*}
        d_{j+1}-d_{j} = 2 - \mu^2 \frac{d_{j+2}-d_{j+1}}{d_{j+2}d_{j+1}}.
    \end{align*}
    This readily implies that $d_{j+1} -d_j \leq 2$ as $d_{j+2} \geq d_{j+1}$ by assumption. Moreover,
    \begin{align*}
        2 - \mu^2 \frac{d_{j+2}-d_{j+1}}{d_{j+2}d_{j+1}} \geq 2\left(1 - \frac{\mu^2}{d_{j+1}d_{j+2}}\right) \geq 0
    \end{align*}
    using \eqref{ineq : product dj dj+1}. This proves that 
    \[
    0 \leq d_{j+1} - d_j \leq 2
    \]
    and we obtain that $0\leq d_{j+1} - d_j \leq 2$ for all $j \in \{1, \dots, n-1\}$ by induction.
Now, given $j \in \{1,\dots,n-1\}$ and combining \eqref{ineq : product dj dj+1} with the fact that $d_j \leq d_{j+1}$, we get
\begin{align*}
   d_{j+1}^2 \geq d_{j+1}d_j = 2jd_{j+1} + \mu^2.
\end{align*}
Studying the positivity of the polynomial $x \to x^2 - 2jx - \mu^2$, we obtain that
\begin{align}\label{ineg : lower bound for dj}
    d_{j+1} \geq j + \sqrt{j^2+\mu^2}
\end{align}
since $d_{j+1}$ is positive. Similarly, using that $d_j \geq d_{j+1} -2$, we obtain that
\begin{align*}
    d_{j+1} \leq j+1 + \sqrt{(j+1)^2 + \mu^2}.
\end{align*}  
Using  $2+\frac{\mu^2}{2+\sqrt{4+\mu^2}} = \sqrt{4+\mu^2}$, the bounds for $d_1$ are obtained thanks to $d_1 = 2 + \frac{\mu^2}{d_2}$ along with \eqref{eq : estimates on dn}.
\end{proof}

\begin{lemma}\label{lem : estimates entries of Tinv with dn}
For all $i,j \in \{1, \dots, n\}$, we have
    \begin{align*}
     \left|\left(T(n)^{-1}\right)_{i,j}\right| \leq  \begin{cases}
         \displaystyle\frac{\mu^{j-i}}{d_i\dots d_j} &\text{ if } i \leq j\\
         \displaystyle\frac{\mu^{i-j}}{d_j\dots d_i} &\text{ if } i > j.
     \end{cases}
\quad \text{and} \quad 
     \left|\left(T(n)^{-1}\right)_{i,j}\right| \geq  \begin{cases}
         \displaystyle\frac{2(i-1)\mu^{j-i}}{d_{i-1}\dots d_j} &\text{ if } 2 \leq i \leq j\\
         \displaystyle\frac{2(j-1)\mu^{i-j}}{d_{j-1}\dots d_i} &\text{ if } 2 \leq j < i.
     \end{cases}
\end{align*}
\end{lemma}
\begin{proof}
    Let $j \in \{1,\dots,n\}$ and let $i \leq j$. Then
\begin{align}
\nonumber
   (T(n)^{-1})_{i+2,j+2} & = \sum_{k=1}^{i+2} \frac{(-\mu)^{i+2-k}}{d_k \dots d_{i+2}} \frac{\mu^{j+2-k}}{d_{k+1}\dots d_{j+2}} \\
   \nonumber
   & =  \frac{\mu^{j-i}}{d_{i+2} \dots d_{j+2}} 
 -\frac{\mu^2}{d_{i+2}d_{j+2}}\sum_{k=1}^{i+1} \frac{(-\mu)^{i+1-k}}{d_k \dots d_{i+2}} \frac{\mu^{j+1-k}}{d_{k+1}\dots d_{j+2}}\\
 &= \frac{\mu^{j-i}}{d_{i+2} \dots d_{j+2}} 
 -\frac{\mu^2}{d_{i+2}d_{j+2}}(T(n)^{-1})_{i+1,j+1}\label{eq : lemma ineq Tn step 1}.
\end{align}
    Since $i \leq j$, we know from \eqref{eq : inverse for Tn} that $(T(n)^{-1})_{i+1,j+1} \geq 0$. This implies that 
    $
    |(T(n)^{-1})_{i,j}| \leq  \frac{\mu^{j-i}}{d_{i} \dots d_{j}}
    $
    for all $i \leq j$. 
Now, going back to \eqref{eq : lemma ineq Tn step 1}, we get
\begin{align*}
    (T(n)^{-1})_{i+2,j+2} &=   \frac{\mu^{j-i}}{d_{i+2} \dots d_{j+2}} 
 -\frac{\mu^2}{d_{i+2}d_{j+2}}(T(n)^{-1})_{i+1,j+1}\\
 &=  \frac{\mu^{j-i}}{d_{i+2} \dots d_{j+2}} 
 -\frac{\mu^2}{d_{i+2}d_{j+2}}\left( \frac{\mu^{j-i}}{d_{i+1} \dots d_{j+1}} 
 -\frac{\mu^2}{d_{i+1}d_{j+1}}(T(n)^{-1})_{i,j} \right)\\
 & =  \frac{\mu^{j-i}}{d_{i+2} \dots d_{j+2}} \left(1-\frac{\mu^2}{d_{i+1}d_{i+2}}\right) + \frac{\mu^4}{d_{i+1}d_{i+2}d_{j+1}d_{j+2}}(T(n)^{-1})_{i,j}.
\end{align*}
Using that $1-\frac{\mu^2}{d_{i+1}d_{i+2}} = \frac{2(i+1)}{d_{i+1}}$ and that $(T(n)^{-1})_{i,j} \geq 0$, it yields
\[
|(T(n)^{-1})_{i,j}| \geq \frac{2(i-1)\mu^{j-i}}{d_{i-1} \dots d_{j}}
\]
  for all $2 \leq i \leq j$. A similar reasoning can be established in the case $i > j$.   
\end{proof}

With the above lemma available, we can provide explicit estimations for the entries of $T(n)^{-1}$. This is achieved in the lemma below.

\begin{lemma}\label{lem : estimates entries of Tinv with g}
    Let $i \geq 2$ and $j \geq 2$ such that $i \leq j$, then
    \begin{equation}
        \frac{2(i-1)}{\alpha_{i}(\mu)\alpha_{j}(\mu)} \exp\left(g\left(\frac{i}{\mu}\right)-g\left(\frac{j}{\mu}\right)\right) \leq   |(T(n)^{-1})_{i,j}| \leq 
             \frac{1}{\alpha_{j-1}(\mu)} \exp\left(g\left(\frac{i-2}{\mu}\right)-g\left(\frac{j-2}{\mu}\right)\right)  
    \end{equation}
    where
\begin{equation}   \label{eq : definition of the function g}
g(x) \bydef \mu\left(x\ln\left(\sqrt{x^2 +1} + x\right) - \sqrt{x^2+1} +1  \right).
\end{equation}
Similarly, for all $j \geq 2$, we have
 \begin{equation}
       \frac{1}{\alpha_{j}(\mu)} \exp\left(g\left(\frac{1}{\mu}\right)-g\left(\frac{j}{\mu}\right)\right)   \leq  |(T(n)^{-1})_{1,j}| \leq 
            \frac{\mu}{ \alpha_{j-1}(\mu)\sqrt{4+\mu^2}} \exp\left(-g\left(\frac{j-2}{\mu}\right)\right)  
    \end{equation}
    and $\frac{1}{\gamma_1(\mu)} \leq |(T(n)^{-1})_{1,1}| \leq  \frac{1}{\sqrt{4+\mu^2}}.$ Moreover, we have
    \begin{equation}\label{eq : ineq for g}
      \frac{x^2\mu}{2} \geq  g(x) \geq \begin{cases}
            \frac{\mu}{2}x &\text{ if }~ x \geq 1.1\\
            0.46 \mu x^2 &\text{ if } ~ 0.75 \leq x \leq 1.1\\
            0.479 \mu x^2 & \text{ if } ~ 0.5 \leq x \leq 0.75\\
            0.49 \mu x^2 &\text{ if } ~ 0 \leq x \leq 0.5.
        \end{cases}
    \end{equation}
    Finally, denoting $I \bydef \floor{\frac{3}{4}\mu} + 1$, if $j \geq I$, then we have the following estimate
    \begin{align}\label{eq : estimate with powers}
        |(T(n)^{-1})_{i,j}| \leq \frac{1}{\alpha_{I-1}(\mu)} \exp\left(g\left(\frac{i-2}{\mu}\right)-g\left(\frac{I-2}{\mu}\right)\right) \left(\frac{\mu}{\alpha_{I}(\mu)}\right)^{j-I}.
    \end{align}
\end{lemma}

\begin{proof}
For simplicity, denote $T(n)$ by $T$.
Suppose that $i \leq j$, then
    recall from Lemma~\ref{lem : estimates entries of Tinv with dn} that 
    \begin{align*}
        |T^{-1}_{i,j}| \leq \frac{\mu^{j-i}}{d_i \dots d_j}.
    \end{align*}
    Moreover, defining $x_l \bydef \frac{l}{\mu}$ and using Lemma~\ref{lem : lemma on sequence dn}, we obtain
    \begin{align}\label{eq : estimates of Tij with di}
         \frac{\mu^{j-i}}{d_i \dots d_j} \leq \frac{1}{d_j} \prod_{l=i-1}^{j-2} \frac{1}{x_l + \sqrt{x_l^2 + 1}},
    \end{align}
    where we use the convention that $\prod_{l=i-1}^{j-2} \frac{1}{x_l + \sqrt{x_l^2 +1}} =1$ if $i=j$.
    Using that $x \to \ln(x + \sqrt{x^2 +1})$ is an increasing function, we obtain that 
{\small \begin{align*}
    \sum_{l=i-1}^{j-2}\ln(x_l + \sqrt{x_l^2 +1}) \geq  \int_{i-2}^{j-2}\ln\left(\frac{x}{\mu} + \sqrt{\frac{x^2}{\mu^2}+1}\right) dx &=  \mu \int_{\frac{i-2}{\mu}}^{\frac{j-2}{\mu}}\ln\left(y + \sqrt{y^2+1}\right) dy
    = g(x_{j-2})-g(x_{i-2})
\end{align*}}
where  $g$ is defined in \eqref{eq : definition of the function g}. Using \eqref{eq : estimates of Tij with di}, we obtain that
\[
 |T^{-1}_{i,j}| \leq 
             \frac{1}{d_j} \exp\left(g\left(\frac{i-2}{\mu}\right)-g\left(\frac{j-2}{\mu}\right)\right).
\]
For the upper bound when  $i=1$, we proceed as above. Indeed,
\begin{align}\label{eq : estimate firt column above}
     |T^{-1}_{1,j}| = \frac{\mu^{j-1}}{d_1 \dots d_j} \leq \frac{\mu}{d_1d_j} \prod_{l=1}^{j-2} \frac{1}{x_l + \sqrt{x_l^2 + 1}} \leq \frac{\mu}{d_1 d_j}  \exp\left(-g\left(\frac{j-2}{\mu}\right)\right).
\end{align}
For the lower bound, notice that  
\begin{align}\label{eq : estimate first column below}
    |T^{-1}_{1,j}| = \frac{\mu^{j-1}}{d_1 \dots d_j} \geq  \frac{1}{d_j} \prod_{l=1}^{j-1} \frac{1}{x_l + \sqrt{x_l^2 + 1}}
\end{align}
using Lemma \ref{lem : estimates entries of Tinv with dn}.  Using that $x \to \ln(x + \sqrt{x^2 +1})$ is an increasing function, we obtain that 
{\small
\begin{align*}
    \sum_{l=1}^{j-1}\ln(x_l + \sqrt{x_l^2 +1}) \leq  \int_{1}^{j}\ln\left(\frac{x}{\mu} + \sqrt{\frac{x^2}{\mu^2}+1}\right) dx 
    = g(x_{j})-g(x_{1}).
\end{align*}}
This implies that 
$
    |T^{-1}_{1,j}| \geq \frac{1}{d_j} \exp( g(x_{1})-g(x_{j}) ).
$
A similar reasoning allows to conclude for the lower bound of $|T^{-1}_{i,j}|$ in the case $i \geq 2.$ We obtain the desired formulas for the estimates of $ |T^{-1}_{i,j}| $ using  Lemma \ref{lem : lemma on sequence dn}. The proof of \eqref{eq : ineq for g} is obtained using tedious but straightforward computations.
Finally, we focus on \eqref{eq : estimate with powers}. We notice that 
\begin{align*}
     |T^{-1}_{i,j}| \leq \frac{\mu^{j-i}}{d_i \dots d_j} =   \frac{\mu^{I-i}}{d_i \dots d_I} \frac{\mu^{j-I}}{d_{I+1} \dots d_j} \leq   \frac{1}{\alpha_{I-1}(\mu)} \exp\left(g\left(\frac{i-2}{\mu}\right)-g\left(\frac{I-2}{\mu}\right)\right) \frac{\mu^{j-I}}{d_{I+1} \dots d_j}
\end{align*}
using the above computations. Then, for all $l \geq I+1$, we have 
\begin{align*}
    \frac{\mu}{d_l} \leq   \frac{\mu}{\alpha_{l-1}(\mu)} \leq \frac{\mu}{\alpha_I(\mu)}
\end{align*}
since $\alpha_l(\mu)$ is increasing in $l$. This concludes the proof.
\end{proof}

\begin{remark}
    From \eqref{eq : inverse for Tn} we have that
    \begin{align*}
        |(T(n)^{-1})_{i,j}| = |(T(n)^{-1})_{j,i}|.
    \end{align*}
    Consequently, the above lemma also provides estimates for the entries of $T(n)^{-1}$ when $i > j$.
\end{remark}

\subsection{Computation of explicit estimates}

\begin{lemma}\label{lem : norm of first column}
Let $N_2 \geq 0$, then we have the following inequalities
{\small \begin{align*}
\sum_{i=1}^\infty \left|(T^{-1})_{i,1}\right| &\geq \frac{\mu \exp(g(\frac{2}{\mu}))}{1+\sqrt{1+\mu^2}}  \left(\frac{\sqrt{\pi}}{\sqrt{2\mu}} - \frac{2}{\mu}\right)  \\
\sum_{i=N_2+1}^\infty \left|(T^{-1})_{i,1}\right| &\leq 
{\tiny
\begin{cases}
\frac{\exp\left(-g\left(\frac{N_2}{\mu}\right)\right)}{\sqrt{4+\mu^2}} +   \frac{\mu}{\sqrt{4+\mu^2}}\left(\exp\left(\frac{0.49N_2^2}{\mu}\right) \min\left\{\frac{1}{2} ; \frac{\sqrt{\pi}}{2\sqrt{0.49\mu}}\right\} + 0.6\exp(-0.115\mu) + \frac{2\exp(-0.55\mu)}{\mu}\right), &\text{ if } \frac{N_2}{\mu} \leq \frac{1}{2}\\
\frac{\exp\left(-g\left(\frac{N_2}{\mu}\right)\right)}{\sqrt{4+\mu^2}} +   \frac{\mu}{\sqrt{4+\mu^2}}\left(\exp\left(\frac{0.46N_2^2}{\mu}\right) \min\left\{0.6 ; \frac{\sqrt{\pi}}{2\sqrt{0.46\mu}}\right\} + \frac{2\exp(-0.55\mu)}{\mu}\right), &\text{ if } \frac{1}{2} \leq \frac{N_2}{\mu} \leq 1.1\\
\frac{\exp\left(-g\left(\frac{N_2}{\mu}\right)\right)}{\sqrt{4+\mu^2}} +   \frac{2\exp(-0.55N_2)}{\sqrt{4+\mu^2}}, &\text{ if } 1.1 \leq \frac{N_2}{\mu}
\end{cases}
}
\\
\sum_{i=N_2+1}^\infty 2i \left|(T^{-1})_{i,1}\right| &\leq 
{\tiny
\begin{cases}
\frac{\mu}{0.479d_1} \left( e^{-0.479\frac{N_2^2}{\mu}} -  e^{-0.479\frac{9\mu}{16}}\right) +  \frac{2(N_2+1)e^{-\frac{0.479N_2^2}{\mu}}}{d_1} +  \frac{\sqrt{\mu \pi}}{\sqrt{0.479}d_1}e^{-\frac{0.479N_2^2}{\mu}} + \frac{\mu}{d_1} e^{-0.479\frac{9\mu}{16}}, &\text{ if } \frac{N_2}{\mu} \leq \frac{3}{4}\\
\frac{\mu}{d_1} e^{-0.479\frac{9\mu}{16}} \left(\frac{\mu}{1+2\mu}\right)^{N_2-\frac{3\mu}{4}} \frac{1+2\mu}{1+\mu}, &\text{ if }\frac{N_2}{\mu} > \frac{3}{4}.
\end{cases}
}
\end{align*}}
\end{lemma}

\begin{proof}
For the computation of the estimates for the norm of $(T^{-1})_{col(1)}$, we first derive new estimates for each 
 $T^{-1}_{i,1}$ based on Lemmas \ref{lem : estimates entries of Tinv with dn} and \ref{lem : estimates entries of Tinv with g}.  Using \eqref{eq : estimate firt column above}, we get
\begin{align}\label{eq : step 0 proof column 1}
     |T^{-1}_{i,1}| = \frac{\mu^{i-1}}{d_1 \dots d_i} \leq \frac{1}{d_1} \prod_{l=1}^{i-1} \frac{1}{x_l + \sqrt{x_l^2 + 1}} \leq \frac{1}{\sqrt{4+\mu^2}}  \exp\left(-g\left(\frac{i-1}{\mu}\right)\right)
\end{align}
for all $i \geq 1$.
Then, using \eqref{eq : estimate first column below}, we have  
\begin{align*}
    |T^{-1}_{i,1}| = \frac{\mu^{i-1}}{d_1 \dots d_i} \geq  \frac{1}{d_1} \prod_{l=2}^{i} \frac{1}{x_l + \sqrt{x_l^2 + 1}}
\end{align*}
using Lemma \ref{lem : estimates entries of Tinv with dn}.  Using that $x \to \ln(x + \sqrt{x^2 +1})$ is an increasing function, we obtain that 
\begin{align*}
    \sum_{l=2}^{i}\ln(x_l + \sqrt{x_l^2 +1}) \leq  \int_{2}^{i+1}\ln\left(\frac{x}{\mu} + \sqrt{\frac{x^2}{\mu^2}+1}\right) dx 
    = g(x_{i+1})-g(x_{2}).
\end{align*}
This implies that 
\begin{align}\label{eq : step 1 proof column 1}
    |T^{-1}_{1,j}| \geq \frac{1}{1+\sqrt{1+\mu^2}} \exp( g(x_{2})-g(x_{i+1}) ).
\end{align}
Now, going back to \eqref{eq : step 0 proof column 1} and using that $g$ is increasing, we have
\begin{align*}
       \sum_{i=N_2+1}^\infty |T^{-1}_{i,1}| &\leq  \frac{\exp\left(-g\left(\frac{N_2}{\mu}\right)\right)}{\sqrt{4+\mu^2}} +   \frac{\mu}{\sqrt{4+\mu^2}}\int_{\frac{N_2}{\mu}}^\infty \exp\left(-g(x)\right)dx.
\end{align*}
Now, if $\frac{N_2}{\mu} \leq \frac{1}{2}$, using \eqref{eq : ineq for g}, we get 
{\small
\begin{align*}
    \int_{\frac{N_2}{\mu}}^\infty {\exp\left(-g(x)\right)} dx &\leq \int_{\frac{N_2}{\mu}}^{\frac{1}{2}}{\exp\left(-0.49\mu x^2\right)} dx + \int_{\frac{1}{2}}^{1.1} {\exp\left(-0.46\mu x^2\right)}dx+ \int_{1.1}^\infty {\exp\left(-\frac{\mu x}{2}\right)} dx\\
    &\leq \int_{\frac{N_2}{\mu}}^{\frac{1}{2}}{\exp\left(-0.49\mu x^2\right)} dx + 0.6\exp(-0.115\mu) + \frac{2\exp(-0.55\mu)}{\mu}.
\end{align*}}
First, notice that 
{\small \[
\int_{\frac{N_2}{\mu}}^{\frac{1}{2}}{\exp\left(-0.49\mu x^2\right)} dx \leq \frac{1}{2}\exp\left(-\frac{0.49N_2^2}{\mu}\right).
\]}
Moreover,
{\small \begin{align*}
    \int_{\frac{N_2}{\mu}}^{\frac{1}{2}}{\exp\left(-0.49\mu x^2\right)} dx = \frac{1}{\sqrt{0.49\mu}}\int_{\frac{N_2}{\mu}\sqrt{0.49\mu}}^{\frac{\sqrt{0.49\mu}}{2}}{\exp\left(-x^2\right)} dx &\leq \frac{\exp\left(-\frac{0.49N_2^2}{\mu}\right)}{\sqrt{0.49\mu}}\int_{0}^\infty{\exp\left(- x^2\right)} dx\\
    &= \frac{\sqrt{\pi}}{2\sqrt{0.49\mu}}\exp\left(-\frac{0.49N_2^2}{\mu}\right).
\end{align*}}
Using the above, we get
{\small \[
    \int_{\frac{N_2}{\mu}}^\infty {\exp\left(-g(x)\right)} dx \leq \exp\left(\frac{-0.49N_2^2}{\mu}\right) \min\left\{\frac{1}{2} ~,~ \frac{\sqrt{\pi}}{2\sqrt{0.49\mu}}\right\} + 0.6\exp(-0.115\mu) + \frac{2\exp(-0.55\mu)}{\mu}.
\]}
Similarly, if $\frac{1}{2} < \frac{N_2}{\mu} \leq 1.1$, then
{\small \[
    \int_{\frac{N_2}{\mu}}^\infty {\exp\left(-g(x)\right)} dx 
    \leq \exp\left(\frac{-0.46N_2^2}{\mu}\right) \min\left\{0.6 ~,~ \frac{\sqrt{\pi}}{2\sqrt{0.46\mu}}\right\} + \frac{2\exp(-0.55\mu)}{\mu}.
\]}
Finally, if $\frac{N_2}{\mu} > 1.1$, we get
{\small \[
    \int_{\frac{N_2}{\mu}}^\infty {\exp\left(-g(x)\right)} dx 
    \leq  \frac{2\exp(-0.55N_2)}{\mu}.
\]}


Then,  using \eqref{eq : step 1 proof column 1}, we have
{\small \[
\sum_{i=1}^\infty |T^{-1}_{i,1}| \geq \frac{\exp(g(\frac{2}{\mu}))}{1+\sqrt{1+\mu^2}} \sum_{i=1}^\infty \exp\left(-g\left(\frac{i+1}{\mu}\right) \right).
\]}
  Combining \eqref{eq : ineq for g} with the fact that $g$ is increasing, we get
  \begin{align*}
       \sum_{i=1}^\infty |T^{-1}_{i,1}| \geq \frac{\exp(g(\frac{2}{\mu}))}{1+\sqrt{1+\mu^2}}  \mu \int_{2}^\infty \exp(-g(x)) dx \geq \frac{\exp(g(\frac{2}{\mu}))}{1+\sqrt{1+\mu^2}}  \mu \int_{\frac{2}{\mu}}^\infty \exp\left(-\frac{x^2\mu}{2}\right) dx.
  \end{align*}
  Now, we have 
{\small \begin{align*}
      \int_{\frac{2}{\mu}}^\infty \exp\left(-\frac{x^2\mu}{2}\right) dx &= \frac{\sqrt{2}}{\sqrt{\mu}}\int_{\frac{\sqrt{2}}{\sqrt{\mu}}}^\infty \exp\left(-x^2\right) dx\\
      &= \frac{\sqrt{2}}{\sqrt{\mu}}\int_{0}^\infty \exp\left(-x^2\right) dx - \frac{\sqrt{2}}{\sqrt{\mu}}\int_0^{\frac{\sqrt{2}}{\sqrt{\mu}}} \exp\left(-x^2\right) dx
      \geq \frac{\sqrt{\pi}}{\sqrt{2\mu}} - \frac{2}{\mu}.
  \end{align*}}
  Finally, the estimates for the term $\sum_{n = N_2+1}^\infty 2i |T^{-1}_{i,1}|$ is based on \eqref{eq : estimate with powers} combined with \eqref{eq : ineq for g}, and follows similar computations as the one exposed above.
\end{proof}

\begin{remark}
    Notice that when $\mu$ is large, the above lemma provides
    \[
  \frac{\sqrt{\pi}}{\sqrt{2}}  \lesssim \sqrt{\mu} \sum_{i=1}^\infty \left|(T^{-1})_{i,1}\right| \lesssim \frac{\sqrt{\pi}}{\sqrt{2}\sqrt{0.98}},
    \]
    leading to a sharp enclosure of $\|(T^{-1})_{col(1)}\|_1$. In particular, this allows to estimate $\xi_\mu$ in Lemma \ref{lem : general bound Linv}.
\end{remark}

Recall that Lemma \ref{lem : general bound Linv} provides explicit upper bounds for the quantities in \eqref{eq : ineq C0 C1n CA0n} depending on $T^{-1}$. We want to use the results established in Section \ref{sec : tridiagonal Tk} in order to derive explicit values for $C_0^*$ and $C_1^*$. We present a series of technical lemmas allowing the computation of the quantities involved in \eqref{eq : general inequalities for C0 and C1}.

\begin{lemma}\label{lem : norm colums of lambda Tinv}
Let $j \in \mathbb{N}$, then we have the following identity
    \begin{align}
        \sum_{i=j+1}^\infty \left| \left(\Lambda_{0} T^{-1}\right)_{i,j}\right| = \left(T^{-1}\right)_{j,j} - \left(T^{-1}\right)_{j+1,j}.
    \end{align}
\end{lemma}

\begin{proof}
Let $n > \frac{\mu^2}{4}$ and let $i,j \in \{1, \dots, n\}$ with $i>j$. Then, using \eqref{eq : inverse for Tn}, we get
    \begin{align*}
        \left(T(n)^{-1}\right)_{i-1,j} - \left(T(n)^{-1}\right)_{i+1,j} = (-1)^{j+i+1}\mu^{i+1-j}\frac{d_{i+2}\dots d_n}{\delta_{j}\dots \delta_n}\left(\frac{d_id_{i+1}}{\mu^2}- 1\right).
    \end{align*}
    Recalling that $d_i = 2i + \frac{\mu^2}{d_{i+1}}$, we have that 
    $
        \frac{d_id_{i+1}}{\mu^2}-1 = \frac{2id_{i+1}}{\mu^2}>0. 
    $
    Taking the limit as $n \to \infty$, we obtain that $(-1)^{i+j+1}\left(\left(T^{-1}\right)_{i-1,j} - \left(T^{-1}\right)_{i+1,j}\right) \geq 0$ for all $i > j$.
   Consequently, using the above and a  telescopic series argument, we have
      \begin{align*}
        \sum_{i=j+1}^\infty \left| \left(\Lambda_{0} T^{-1}\right)_{i,j}\right|   & = \sum_{i=j+1}^{\infty}(-1)^{i+j+1}\left(\left(T^{-1}\right)_{i-1,j} - \left(T^{-1}\right)_{i+1,j}\right) \\
        & = \left(T^{-1}\right)_{j,j} - \left(T^{-1}\right)_{j+1,j}. \quad \qedhere
    \end{align*}
\end{proof}

\begin{lemma}\label{lem : operator norm Tinv lambda}
Let $j \in \mathbb{N}$ and let $\mu \geq 3000$, then
    \begin{align*}
       \|(T^{-1}\Lambda_0)_{col(j)}\|_1 \leq  \frac{7.82}{j-1+\sqrt{(j-1)^2+\mu^2}}
    \end{align*}
\end{lemma}

\begin{proof}
Let $j \geq 4$ and notice that 
\[
\mu T^{-1}\Lambda_0  =  T^{-1}D_0 - I_d. 
\]
The above identity implies that
\begin{align*}
    \mu \|(T^{-1}\Lambda_0)_{col(j)}\|_1 = \sum_{i\neq j} 2j |T^{-1}_{i,j}| + |1 - 2jT^{-1}_{j,j}|.
\end{align*}
Now, using that $\mu \Lambda_0 T^{-1} =  D_0T^{-1} -I_d$, we have that 
\begin{align*}
   \sum_{i = j+1}^\infty 2j |T^{-1}_{i,j}| \leq   \sum_{i = j+1}^\infty 2i |T^{-1}_{i,j}| = \mu \sum_{i = j+1}^\infty |T^{-1}_{i-1,j} - T^{-1}_{i+1,j}|.
\end{align*}
But using  Lemma \ref{lem : norm colums of lambda Tinv}, we obtain that 
\begin{align*}
   \sum_{i = j+1}^\infty 2j |T^{-1}_{i,j}| \leq    \mu \sum_{i = j+1}^\infty (-1)^{i+j+1}(T^{-1}_{i-1,j} - T^{-1}_{i+1,j}) = \mu\left(T^{-1}_{j,j} - T^{-1}_{j+1,j}\right).
\end{align*}
Using Lemma \ref{lem : estimates entries of Tinv with dn} and  \eqref{eq : estimates on dn}, we get that $T^{-1}_{j,j} \geq 0$ and $T^{-1}_{j,j} \leq \frac{1}{2j}$. Moreover, the aforementioned lemma yields
\begin{align*}
    \|(T^{-1}\Lambda_0)_{col(j)}\|_1 \leq \frac{1}{\mu}\sum_{i =1}^{j-1} 2j |T^{-1}_{i,j}| + \frac{1}{\mu}-\frac{2j}{\mu\alpha_{j-1}(\mu)} + \frac{1}{\alpha_{j-1}(\mu)}\left(1 + \frac{\mu}{\alpha_j(\mu)} \right). 
\end{align*}
Now, we notice that $j \to \frac{\alpha_{j-1}(\mu)-2j}{\mu}$ is decreasing in $j$ and we get
\begin{align*}
    \|(T^{-1}\Lambda_0)_{col(j)}\|_1 \leq \frac{1}{\mu}\sum_{i =1}^{j-1} 2j |T^{-1}_{i,j}| + \frac{3}{\alpha_{j-1}(\mu)}. 
\end{align*}

It remains to estimate the term $ \sum_{i = 1}^{j-1} 2j |T^{-1}_{i,j}|$. Let $n > \frac{\mu^2}{4}$ and still denote $T(n) = T$ for simplicity.  Then, using \eqref{eq : estimates of Tij with di} and the proof of Lemma \ref{lem : estimates entries of Tinv with dn}, we get
\begin{align*}
    \frac{2j}{\mu}|T^{-1}_{i,j}| \leq \frac{2j}{d_{j-1}d_j} \frac{\mu^{j-i-1}}{d_i \dots d_{j-2}} \leq \frac{2j}{d_{j-1}d_j} \exp\left(g\left(\frac{i-2}{\mu}\right)-g\left(\frac{j-3}{\mu}\right)\right).
\end{align*}
This implies that 
    \begin{align*}
        \frac{2j}{\mu}\sum_{i =1}^{j-1} |T^{-1}_{i,j}|  \leq \frac{2j\mu^{j-2}}{d_1 \dots d_j} +  \frac{2j}{d_{j-1} d_j}  + \frac{2j}{d_{j-1}d_j }\sum_{i = 2}^{j-2} \exp\left(g\left(\frac{i-2}{\mu}\right)-g\left(\frac{j-3}{\mu}\right)\right).  
    \end{align*}
    Now, since $g$ is an increasing function, we obtain
    \begin{align*}
        \sum_{i = 2}^{j-2}  \exp\left(g\left(\frac{i-2}{\mu}\right)\right) \leq  \int_{2}^{j-1}  \exp\left(g\left(\frac{i-2}{\mu}\right)\right) di =  \mu \int_{0}^{x_{j-3}}  \exp\left(g\left(x\right)\right) dx 
    \end{align*}
    where we recall that $x_{j} = \frac{j}{\mu}$.  Let $\alpha \in \left(0, x_{j-3}\right)$, then we consider the following decomposition
      \begin{align}\label{eq : in appendix lemma step 1}
         \int_{0}^{x_{j-3}}  \exp\left(g\left(x\right)\right) dx  &\leq \int_{0}^{\alpha}  \exp\left(g\left(x\right)\right) dx   + \int_{\alpha}^{x_{j-3}}  \exp\left(g\left(x\right)\right) dx. 
    \end{align}
    Notice $g'(x) = \ln\left(\sqrt{x^2+1} + x\right)$, which is an increasing function on $[0,\infty).$ 
    Then, the second term on the right hand side of \eqref{eq : in appendix lemma step 1} is estimated as follows
    \begin{align*}
        \int_{\alpha}^{x_{j-3}}  \exp\left(g\left(x\right)\right) dx = \int_{\alpha}^{x_{j-3}} \frac{g'(x)}{g'(x)} \exp\left(g\left(x\right)\right) dx &\leq \frac{1}{\mu \ln\left(\sqrt{\alpha^2+1} + \alpha\right)} \int_{\alpha}^{x_{j-3}} {g'(x)} \exp\left(g\left(x\right)\right)\\
        & =\frac{\exp\left(g\left(x_{j-3}\right)\right) - \exp\left(g\left(\alpha\right)\right)}{\mu\ln\left(\sqrt{\alpha^2+1} + \alpha\right)}. 
    \end{align*}
    Concerning the first term on the right hand side of \eqref{eq : in appendix lemma step 1}, we first observe that $g(x) \leq \frac{x^2}{2}$ for all $x >0$. This implies 
    \begin{align*}
        \int_{0}^{\alpha}  \exp\left(g\left(x\right)\right) dx \leq \int_{0}^{\alpha}  \exp\left(\frac{x^2}{2}\right)  dx \leq \alpha \exp\left(\frac{\alpha^2}{2}\right).
    \end{align*}
    Consequently, we obtain that 
    {\small{
    \begin{align*}
        \frac{2j}{d_j d_{j-1}}\sum_{i =2}^{j-2} \exp\left(g\left(\frac{i-2}{\mu}\right)-g\left(\frac{j-3}{\mu}\right)\right) &\leq \frac{2j \exp(-g(x_{j-3})) }{d_j d_{j-1}} \left(\frac{\exp\left(g\left(x_{j-3}\right)\right) - \exp\left(g\left(\alpha\right)\right)}{\ln\left(\sqrt{\alpha^2+1} + \alpha\right)} + \alpha \mu \exp\left(\frac{\alpha^2}{2}\right)\right)\\
        &\leq \frac{2j}{d_jd_{j-1}} \frac{1}{\ln\left(\sqrt{\alpha^2+1} + \alpha\right)} + \frac{2j \alpha \mu}{d_jd_{j-1}} \exp(-g(x_{j-3})) \exp\left(\frac{\alpha^2}{2}\right).
    \end{align*}
    }}
    Now, one can prove that 
    \begin{align}\label{eq : estimates on g}
        g(x) \geq \begin{cases}
            \mu\frac{x}{2} &\text{ for all } x \geq 1.2\\
            0.45 \mu x^2   &\text{ for all } x < 1.2.
        \end{cases}
    \end{align}
    Suppose that $x_{j-3} \geq 1.2$. Then, we choose $\alpha =1.2$, and, using Lemma \ref{lem : lemma on sequence dn}, we get
    {\small
     \begin{align*}
       \frac{2j}{ d_{j-1}}\sum_{i =2}^{j-2} \exp\left(g\left(\frac{i-2}{\mu}\right)-g\left(\frac{j-3}{\mu}\right)\right)
        &\leq  \frac{2j}{d_{j-1}\ln\left(1+\sqrt{2.2}\right)} + \frac{2.4j\mu}{d_{j-1}} \exp(-\frac{1}{2}(j-3)) \exp\left(\frac{1}{2}\right)\\
        &\leq  \frac{j}{(j-1)\ln\left(1+\sqrt{2.2}\right)} + \frac{2.4\mu}{d_{j-1}} \exp(-0.6\mu) \exp\left(\frac{1}{2}\right) \\
        &\leq 1.5
    \end{align*}
    }
    for all $\mu \geq 3000$ and using that $x_{j-3} \geq 1.2$.
  Suppose now that $x_{j-2} < 1.2$. Then, we choose $\alpha = \alpha_j \bydef 1-\exp(-x_{j-3})$. Note that by construction $\alpha_j \leq x_{j-3}$. Moreover, since $d_{j-1} \geq j-2 + \sqrt{(j-2)^2 + \mu^2}$, we have
 {\small  \begin{align*}
        \frac{2j}{d_j d_{j-1}} \frac{1}{\ln\left(\sqrt{\alpha^2+1} + \alpha\right)} &=  \frac{2x_{j-3}}{ d_j\frac{d_j-1}{\mu} \ln\left(\sqrt{\alpha^2+1} + \alpha\right)} + \frac{6}{\mu d_jd_{j-1} \ln\left(\sqrt{\alpha^2+1} + \alpha\right)} \\
        &\leq \frac{1}{d_j} \frac{2x_{j-3}}{x_{j-3} + \frac{1}{\mu}+\sqrt{(x_{j-3}+\frac{1}{\mu})^2 + 1} }  \frac{1}{\ln\left(\sqrt{\alpha_j^2+1} + \alpha_j\right)} + \frac{6}{\mu d_j d_{j-1} \ln\left(\sqrt{\alpha^2+1} + \alpha\right)} \\
       &= \frac{1}{d_j} h(x_{j-3})   + \frac{6}{\mu d_j d_{j-1} \ln\left(\sqrt{\alpha^2+1} + \alpha\right)},
    \end{align*}}
    where 
    \[
    h(x) \bydef  \frac{2x}{x + \frac{1}{\mu}+\sqrt{(x+\frac{1}{\mu})^2 + 1} }  \frac{1}{\ln\left(\sqrt{(1-\exp(-x))^2+1} + (1-\exp(-x))\right)}.
    \]
    Studying the variations of $h$, we obtain that $h(x) \leq 2$ for all $0 \leq x \leq 2$. Moreover, notice that 
    \begin{align*}
        \frac{1}{\ln\left(\sqrt{\alpha^2+1} + \alpha\right)} \leq \frac{2}{\alpha} = \frac{2}{1-\exp(-x_{j-3})} \leq \frac{4}{x_{j-3}} = \frac{4\mu}{j-3}.
    \end{align*}
    This implies that 
    \[
    \frac{6}{\mu d_j d_{j-1} \ln\left(\sqrt{\alpha^2+1} + \alpha\right)} \leq \frac{24}{(j-3)d_j d_{j-1}}.
    \]
    Secondly,
    {\small
    \begin{align*}
         \frac{2j \alpha \mu }{d_jd_{j-1}} \exp(-g(x_{j-3})) \exp\left(\frac{\alpha^2}{2}\right) &\leq \frac{2j (1-\exp(-x_{j-3})) \mu }{d_jd_{j-1}} \exp(-0.45 \mu x_{j-3}^2) \exp\left(\frac{1}{2}\right) \\
         &\leq \frac{2x_{j} x_{j-3} \mu }{d_j \frac{d_{j-1}}{\mu}} \exp(-0.45 \mu x_{j-2}^2) \exp\left(\frac{1}{2}\right) \\
         &\leq  \frac{2 x_{j-3}^2 \mu }{d_j} \exp(-0.45 \mu x_{j-3}^2) \exp\left(\frac{1}{2}\right) +  \frac{6 x_{j-3} }{d_j} \exp(-0.45 \mu x_{j-3}^2) \exp\left(\frac{1}{2}\right) \\
         &\leq \frac{1}{d_j}\left(\frac{2e^{-\frac{1}{2}}}{0.45 } + \frac{6}{\sqrt{0.9\mu} }\right).
    \end{align*}}
    This implies that 
    \begin{align*}
        \frac{2j}{d_{j-1}}\sum_{i = 2}^{j-1} \exp\left(g\left(\frac{i-2}{\mu}\right)-g\left(\frac{j-2}{\mu}\right)\right) \leq \frac{2e^{-\frac{1}{2}}}{0.45} + 2 +  \frac{6}{\sqrt{0.9\mu}} + \frac{24}{(j-3)d_{j-1}} \leq 4.82
    \end{align*}
    for all $\mu \geq 3000.$ Since this is true for all $n \geq \frac{\mu^2}{4}$, we conclude the proof for $j \geq 4$.  Now, suppose that $j = 3$. Then we have
    \begin{align*}
        \frac{2j}{\mu}\sum_{i=1}^{2} |T^{-1}_{i,j}| \leq \frac{6}{\mu} \left( \frac{\mu^2}{d_1d_2d_3} + \frac{\mu}{d_2d_3}\right) \leq \frac{12}{\mu^2} < \frac{4.82}{2+\sqrt{4+\mu^2}} 
        \end{align*}
        for all $\mu \geq 3000$. The same reasoning applies for the case $j=2$, which concludes the proof.
\end{proof}

Finally, the next lemma provides an explicit value for $C_0^*$ and $C_1^*$. Combined with Section \ref{ssec : computation for a bounded range of mu}, it allows to compute explicit values for the bounds in \eqref{eq : ineq C0 C1n CA0n} for all $\mu \geq 0.$

\begin{lemma}\label{lem : C0 and C1}
    Suppose that $\mu \geq 3000$ and let $C_0^*, C_1^*$ be defined as 
    \begin{align*}
        C_0^* \bydef 15.71 ~~ \text{ and } ~~
        C_1^* \bydef 15.71.
    \end{align*}
    Then \begin{align*}
        (2 + \mu) \|W_0 \mathcal{L}_\mu^{-1}  {\Lambda} W_0^{-1}\|_1 \leq {C_0^*} ~~ \text{ and } ~~
         \|W_0 \mathcal{L}_\mu^{-1}  {\Lambda} \mathcal{D}_{\mu} W_0^{-1}\|_1 \leq C_1^*.
    \end{align*}
\end{lemma}

\begin{proof}
Using Lemma \ref{lem : general bound Linv} combined with Lemma \ref{lem : norm of first column} and Lemma \ref{lem : operator norm Tinv lambda}, we have 
\begin{align*}
     (\mu+2) \|W_0 \mathcal{L}_\mu^{-1}  {\Lambda} W_0^{-1}\|_1 &\leq \frac{3002}{3000} \max\left\{ 2, ~   \left(\frac{1}{1+2\mu \xi_\mu} + 2\right) \mu\|T^{-1}_\mu\Lambda_0\|_1\right\}\\
     &\leq \frac{3002}{3000} \left(2*7.82 + \frac{ 7.82}{1+2\mu \xi_\mu}\right)\\
 &\leq  \frac{3002}{3000} \left( 2*7.82 + \frac{7.82}{1 +  \frac{2\mu^2 \exp(g(\frac{2}{\mu}))}{1+\sqrt{1+\mu^2}}  \left(\frac{\sqrt{\pi}}{\sqrt{2\mu}} - \frac{2}{\mu}\right) } \right)\\
 & \leq 15.71
\end{align*}
for all $\mu \geq 3000$. Similarly, using Lemma \ref{lem : general bound Linv} and Lemma  \ref{lem : operator norm Tinv lambda}, we get
$
     \| W_0 \mathcal{L}_\mu^{-1}  \Lambda \mathcal{D}_{\mu} W_0^{-1}\|_1 \leq 15.71. 
$
\end{proof}

\section{Applications}\label{sec : applications}

In this section, we illustrate the preceding analysis with specific examples of parabolic partial differential equations. In particular, we establish the constructive existence of solutions to initial value problems (IVPs) for the 2D Navier-Stokes equations, the Swift-Hohenberg equation, and the Kuramoto-Sivashinsky equation. Given a fixed integration time $h$ and an initial condition $b$, our goal is to constructively prove the existence of a solution to the corresponding IVP. All computer-assisted proofs, including the requisite codes, are accessible on GitHub at \cite{julia_cadiot}.

\subsection{The 2D Navier-Stokes equations}\label{ssec : NSE}

Let us consider the 2D Navier-Stokes equations (NSE)
\begin{align*}
   \begin{cases}
       {u}_t =  \nu \Delta {u} - u u_{x_1} - vu_{x_2} -  p_{x_1}  \\
       {v}_t =  \nu \Delta v - v v_{x_2} - uv_{x_1} -  p_{x_2}  \\
       u_{x_1} + v_{x_2} = 0\\
       u(0,{x_1},{x_2}) = b_u({x_1},{x_2}), ~~ v(0,{x_1},{x_2}) = b_v({x_1},{x_2})
   \end{cases} 
\end{align*}
where $\nu >0$ is the viscosity. Then, 
notice that 
\begin{align*}
    \nabla \times (0,u,v) = (0,0,v_{x_1}-u_{x_2})
\end{align*}
and we define $\omega \bydef v_{x_1}-u_{x_2}$ as the $z$ component of the vorticity. In particular, the Navier-Stokes equations can be turned into the vorticity equation, which reads
\begin{align}\label{eq : NS vorticity}
\begin{cases}
    \omega_t = \nu \Delta \omega - (u,v)\cdot \nabla w \\
    u_{x_1} + v_{x_2} = 0\\
    \omega(0,{x_1},{x_2}) = b({x_1},{x_2}) 
\end{cases}
\end{align}
 where $b({x_1},{x_2}) \bydef \partial_{x_1} b_v({x_1},{x_2}) - \partial_{x_2} b_u({x_1},{x_2})$ for all $({x_1},{x_2}) \in (-\pi,\pi)^2$. 


Now, assuming that the mean value of the velocity field is null, that is $u_{0,0}(t) = v_{0,0}(t) = 0$, we can make sense of the  Leray projection 
\begin{align}
    (u,v) = - \Delta^{-1}\nabla \times \omega.
\end{align}
In terms of Fourier coefficients, the above identity yields
\begin{align*}
    u_k = \frac{ik_2}{k_1^2 + k_2^2} w_k ~~ \text{ and } ~~ v_k = -\frac{ik_1}{k_1^2 + k_2^2} w_k 
\end{align*}
for all $k = (k_1,k_2) \in \mathbb{Z}^2\setminus \{0\}$. In particular, we consider sequences indexed on  $\mathbb{Z}^2\setminus \{0\}$ for the above identity to make sense. Moreover, given an integration time $h>0$, we are able to transform \eqref{eq : NS vorticity} into 
\begin{align}\label{eq : NSE full vorticity}
\begin{cases}
    \omega_t = \nu \Delta \omega + \left(\Delta^{-1}\nabla \times \omega\right)\cdot \nabla w, \qquad (t,x) \in (0,h) \times (-\pi,\pi)^2\\
    \omega(0,{x_1},{x_2}) = b({x_1},{x_2}). 
\end{cases}
\end{align}
Note that the divergence free equation $u_{x_1} + v_{x_2} =0$ becomes redundant with the identity $(u,v) = - \Delta^{-1}\nabla \times \omega$. Now, we look for $\omega$ in the form
\begin{align*}
    \omega(t,x_1,x_2) =   \sum_{k \in \mathbb{Z}^2\setminus\{0\}} W_{0,k}e^{ik\cdot x} + 2 \sum_{k \in \mathbb{Z}^2\setminus\{0\}} W_{n,k}e^{ik\cdot x} T_n\left(\frac{2t}{h}-1\right).
\end{align*}
Let $D_1, D_2, \partial_{x_1}$ and  $\partial_{x_2}$  be linear operators defined as 
\begin{align*}
    (D_1W)_{k,n} =  \frac{ik_1}{k_1^2 + k_2^2} W_{k,n}, ~~ ~~ (D_2W)_{k,n} =  -\frac{ik_2}{k_1^2 + k_2^2} W_{k,n}\\
    (\partial_{x_1}W)_{k,n} \bydef ik_1W_{k,n} ~~  \text{ and } ~~ (\partial_{x_1}W)_{k,n} \bydef ik_2W_{k,n}
\end{align*}
for all $(k,n) \in \left(\mathbb{Z}^2\setminus\{0\}\right) \times \mathbb{N}_0$.  In particular, $\partial_{x_1}$ and $\partial_{x_2}$ are abuses of notation and are the representation of partial derivatives in Fourier coefficients spaces. Moreover, $D_1$ and $D_2$ arise from the operator $\Delta^{-1}\nabla \times .$

Then, given $h>0$, we define 
\begin{align*}
    \mu_k \bydef \frac{h}{2} \nu (k_1^2 + k_2^2) 
\end{align*}
for all $(k_1,k_2) \in \mathbb{Z}^2\setminus\{0\}$. In particular, $\mu_k \geq \frac{h}{2} \nu$ for all $(k_1,k_2) \in \mathbb{Z}^2\setminus\{0\}$, implying that we do not need to translate  $\mu_k$ to make it positive (cf. \eqref{eq : translation spectrum}).

Consequently, we can write our zero finding problem of interest as 
\begin{align}\label{eq : zero finding navier stokes}
    F(W) \bydef \mathcal{L}W + \frac{h}{2}   \Lambda Q(W) - \beta, ~~ \text{ where } Q(W) =  (D_2W)*(\partial_x W) + (D_1W)*(\partial_y W).
\end{align}
In practice, the computer-assisted analysis becomes more tractable if the differential operators in the nonlinear term can be factored out. Having this simplification in mind, we derive the following result. 
\begin{lemma}\label{lem : non lin navier stokes}
We have the following property
    \begin{align*}
        (D_2W)*(\partial_x W) +  (D_1W)*(\partial_y W) = \partial_{x_1}\left( (D_2W)*W \right) + \partial_{x_2}\left( (D_1W)*W \right).
    \end{align*}
\end{lemma}

\begin{proof}
    The proof in the case of the 3D Navier-Stokes equations can be found in \cite{navier-stokes}. The same arguments apply in the 2D version.
\end{proof}
Using the above lemma, we can define equivalently $Q(W)$ as 
\begin{align*}
    Q(W) =   \partial_{x_1}\left( (D_2W)*W \right) + \partial_{x_2}\left( (D_1W)*W \right).
\end{align*}

In practice, we will consider initial value problems where $b$ is given as a sine series :
\begin{equation}\label{eq : initial data sine}
    b(x_1,x_2) = 4 \sum_{k \in \mathbb{N}^2}^\infty b_{k}\sin(k_1x_1)\sin(k_2x_2), ~~ \text{ where } (b_k)_{k \in \mathbb{N}^2} \in \ell^1(\mathbb{N}^2).
\end{equation}
Noticing the invariance of $F$ under the sine-sine symmetry, we also look for a solution $\omega$ of the form  
\begin{align*}
    \omega(t,x_1,x_2) =   4\sum_{k \in \mathbb{N}^2} W_{0,k}\sin(k_1x_1)\sin(k_2x_2) + 8 \sum_{k \in \mathbb{N}^2} W_{n,k}\sin(k_1x_1)\sin(k_2x_2) T_n\left(\frac{2t}{h}-1\right).
\end{align*}
Note that, by construction, the zero mean value condition is satisfied and we can make sense of the operator $\Delta^{-1} \nabla \times.$ 

\subsubsection{Global Existence}

Given $b$ as in \eqref{eq : initial data sine}, one may seek the existence of a globally bounded and smooth solution to \eqref{eq : NSE full vorticity} for all $t > 0$. In the specific case of the 2D Navier-Stokes equations (NSE), this problem was long resolved by Lions and Prodi in \cite{lions1959theoreme}. However, for more complex initial value problems, such as the three-dimensional case, the question of global existence remains open. Notably, since the pioneering work of Leray, it has been established that the only stationary solution to the unforced NSE is the trivial solution. Moreover, various results (see, e.g., \cite{Kato_NS_IVP, Kato_NS_Solu_Rm, Navier_Stokes_Lemarie_Gilles, NS_thesis, Arnold_NSE}) have demonstrated that if $b$ is sufficiently small, then the solution to the 3D NSE exists for all time.

In this section, we present similar arguments for the 2D NSE, but which could potentially be applied in the 3D case. Indeed, supposing that $\|b\|_1$ is small enough, we prove that \eqref{eq : NSE full vorticity} possesses a smooth and bounded solution for all times.

\begin{lemma}\label{lem : global existence condition}
Let $h_1 >0$ and $0< \epsilon^* < \nu$. If there exists $\sigma_0, \sigma_1 >0$ satisfying
\begin{equation}\label{eq : conditions for global existence}
\begin{aligned}
     &\|b\|_1  +  0.6383 \sqrt{\frac{h_1}{2\nu}} \sigma_0^2 \leq \sigma_0\\
      &\|b\|_1 e^{(-2\nu +\epsilon^*)h_1} +  0.6383 \sqrt{\frac{h_1}{2\nu}} \sigma_0^2 + \frac{1}{2(\nu -\epsilon^*)}\sigma_1^2 \leq \sigma_1\\
         &0.6383 \sqrt{\frac{2h_1}{\nu}} \sigma_0 + \frac{1}{\nu -\epsilon^*}\sigma_1 <1.
\end{aligned}
\end{equation}
 then there exists a unique solution $\tilde{w}$ to \eqref{eq : NSE full vorticity} and $\tilde{\omega}(t) \to 0$ uniformly as $t \to \infty.$
\end{lemma}

\begin{proof}
Let $0 < \epsilon < \epsilon^*$ and let $X_\epsilon$ be the Banach space defined as 
   \[
   X_\epsilon \bydef \{ U = (u_k)_{k \in \mathbb{N}^2}, u_k : (0,\infty) \to \R \text{ and } \|U\|_{X_\epsilon} < \infty \}  ~~ \text{ where }  \|U\|_{X_\epsilon} \bydef 4\sum_{k \in \mathbb{N}^2} \sup_{t \in [0,\infty)} |u_k(t)|e^{\epsilon t}.
   \]
   Now, notice that solutions to \eqref{eq : NSE full vorticity} satisfy
   \begin{align*}
       w_k(t) = e^{\lambda_k t} b_k + \int_0^t e^{\lambda_k(t-s)} ({Q}(W)(s))_k ds \bydef (\mathcal{T}(W(t)))_k
   \end{align*}
   for all $k \in \mathbb{N}^2$ and $t>0$. In particular, we want to prove that $\mathcal{T} : X_\epsilon \to X_\epsilon$ has a fixed point. Let $h_0 = 0$ and $h_2 = \infty$, then for each $i \in \{0, 1\}$ and all $U \in X_\epsilon$, we define $\|U\|_{X_\epsilon,i}$ as 
   \[
    \|U\|_{X_\epsilon,i} \bydef 4\sum_{k \in \mathbb{N}^2} \sup_{t \in [h_i, h_{i+1})} |u_k(t)|e^{\epsilon t}.
   \]
 Then, let $v_k(t) = e^{\epsilon t} w_k(t)$. Since $Q$ is quadratic, we get
   \begin{align*}
       |e^{\epsilon t}(\mathcal{T}(W(t)))_k| &\leq e^{(\lambda_k+\epsilon)t}|b_k| + \int_0^t e^{\lambda_k(t-s) + \epsilon t} e^{-2\epsilon s}\left|({Q}(V(s)))_k\right|ds.
   \end{align*}
  Now,  suppose that $t \in [0,h_{1})$. Since $\lambda_k = -\nu|k|_2^2$ by definition, we have that $\lambda_k \leq - 2\nu$ for all $k \in \mathbb{N}_0^2$. In particular, since $\epsilon < \nu$ by assumption, we have that $\lambda_k + 2\epsilon <0$ for all $k \in \mathbb{N}_0^2$. Using this, we obtain 
\begin{align*}
       |e^{\epsilon t}(\mathcal{T}(W(t)))_k| 
       &\leq e^{(\lambda_k+\epsilon)t}|b_k| +  \sup_{s \in [0, h_{1})}\left|({Q}(V(s)))_k\right|  \int_{0}^{t} e^{\lambda_k(t-s) + \epsilon t} e^{-2\epsilon s}ds\\
       & \leq |b_k|  - \left(1 - e^{(\lambda_k + 2\epsilon)h_{1}}\right)  \frac{\sup_{s \in [0, h_{1})}\left|({Q}(V(s)))_k\right|}{\lambda_k+2\epsilon}.
   \end{align*}

Let $s \in (0,t)$, then setting $V_{p_1,p_2}(s) = -{V}_{-p_1,p_2}(s) = {V}_{-p_1,-p_2}(s)$ and $V_{0,p_2}(s) = V_{p_1,0}(s) = 0$ for all $p \in \mathbb{N}_0^2$, we get
\begin{align*}
    ({Q}(V(s)))_k 
    & = \sum_{p \in \mathbb{Z}^2 \setminus \{0\}}   \frac{p_1k_2 - p_2k_1}{p_1^2+p_2^2}V_p(s) V_{k-p}(s).
\end{align*}
Using that $|p_1k_2 - p_2k_1| \leq |p|_2 |k|_2$, we get
\begin{align*}
    \left|({Q}(V(s)))_k\right| & \leq |k|_2 \sum_{p \in \mathbb{Z}^2 \setminus \{0\}}   \frac{1}{|p|_2}|V_p(s)V_{k-p}(s)|. 
\end{align*}
Using that $\frac{1}{|p|_2} \leq \frac{1}{\sqrt{2}}$ for all $p \in \mathbb{N}^2$, this implies that 
\begin{align*}
    |e^{\epsilon t}(\mathcal{T}(W(t)))_k| 
       & \leq  |b_k|  +    \frac{|k|_2 \left(1 - e^{(-\nu |k|_2^2 + 2\epsilon)h_{1}}\right)}{\sqrt{2}(\nu |k|_2^2 - 2\epsilon)} \sum_{p \in \mathbb{Z}^2 \setminus \{0\}}  \sup_{s \in [0, h_{1})} |V_p(s)V_{k-p}(s)|
\end{align*}
since $\lambda_k = - \nu |k|_2^2$ by definition. Let us consider the function $f(y) \bydef \frac{ (1-e^{-y^2})}{y}$. Then $f$ has a global maximum on $(0,\infty)$ and one can prove that $f(y) < 0.6382$ for all $y \in (0,\infty)$.  Denoting $y = \sqrt{\nu h_1} |k|_2$. Then  for $\epsilon$ sufficiently small, we have 
\[
\frac{|k|_2 \left(1 - e^{(-\nu |k|_2^2 + 2\epsilon)h_{1}}\right)}{\sqrt{2}(\nu |k|_2^2 - 2\epsilon)} \leq 0.6383 \sqrt{\frac{h_1}{2\nu}}.
\]
for all $|k|_2 \geq 0.$ Choosing such an epsilon, the above implies that 
\[
 |e^{\epsilon t}(\mathcal{T}(W(t)))_k| 
        \leq |b_k|  +    0.6383 \sqrt{\frac{h_1}{2\nu}} \sum_{p \in \mathbb{Z}^2 \setminus \{0\}}  \sup_{s \in [0, h_{1})} |V_p(s)V_{k-p}(s)|.
\]
Taking the supremum for $t \in [0,h_1)$ and summing over $k \in \mathbb{N}^2$, the above yields
\begin{align}\label{eq : ineq 1 global}
    \|\mathcal{T}(W)\|_{X_\epsilon,0} \leq \|b\|_1 +  0.6383 \sqrt{\frac{h_1}{2\nu}} \|W\|_{X_\epsilon,0}^2
\end{align}
since $V(t) = e^{\epsilon t} W(t)$ by definition.  Now, suppose that $t \in [h_1, \infty)$. We have
\small{
\begin{align*}
      &|e^{\epsilon t}(\mathcal{T}(W(t)))_k| \\
       &\leq e^{(\lambda_k+\epsilon)t}|b_k| +  \sup_{s \in [0, h_{1})}\left|({Q}(V(s)))_k\right|  \int_{0}^{h_1} e^{\lambda_k(t-s) + \epsilon t} e^{-2\epsilon s}ds +  \sup_{s \in [h_1, \infty)}\left|({Q}(V(s)))_k\right|  \int_{h_1}^{t} e^{\lambda_k(t-s) + \epsilon t} e^{-2\epsilon s}ds\\
       &\leq e^{(-2\nu +\epsilon^*)h_1}|b_k|   - \sup_{s \in [0, h_{1})}\left|({Q}(V(s)))_k\right|  \frac{  1 - e^{(\lambda_k + 2\epsilon)h_{1}}}{\lambda_k+2\epsilon} -  \sup_{s \in [h_1, \infty)}\left|({Q}(V(s)))_k\right|  \frac{1}{\lambda_k + 2 \epsilon}
\end{align*}
}
\normalsize
since $\lambda_k \leq -2 \nu$ for all $k \in \mathbb{N}^2$.
Using a similar reasoning as for the case $[0,h_1)$, we get
\begin{align}\label{eq : ineq 3 global}
     \|\mathcal{T}(W)\|_{X_\epsilon,1} \leq \|b\|_1 e^{(-2\nu +\epsilon^*)h_1} +  0.6383 \sqrt{\frac{h_1}{2\nu}} \|W\|_{X_\epsilon,0}^2 + \frac{1}{2(\nu -\epsilon^*)}\|W\|_{X_\epsilon,1}^2.
\end{align}
One the other hand, using an analogous analysis, we have
\begin{equation}\label{eq : ineq 2 global}
\begin{aligned}
   \|D\mathcal{T}(W) U\|_{\epsilon,0} &\leq  0.6383 \sqrt{\frac{2h_1}{\nu}} \|W\|_{X_\epsilon,0} \|U\|_{X_\epsilon,0} \\
    \|D\mathcal{T}(W) U\|_{\epsilon,1} &\leq   0.6383 \sqrt{\frac{2h_1}{\nu}} \|W\|_{X_\epsilon,0}\|U\|_{X_\epsilon,0} + \frac{1}{\nu -\epsilon^*}\|W\|_{X_\epsilon,1}\|U\|_{X_\epsilon,1}
    \end{aligned}
\end{equation}
for all $U \in X_\epsilon.$ Let $\sigma_0, \sigma_1>0$ be satisfying \eqref{eq : conditions for global existence}, and let  $X_{box}(\sigma_0,\sigma_1)$ be defined as 
\begin{equation*}
    X_{box}(\sigma_0,\sigma_1) \bydef \{U \in X_{\epsilon} : \|U\|_{X_\epsilon,0} \leq \sigma_0 \text{ and } \|U\|_{X_\epsilon,1} \leq \sigma_1\}.
\end{equation*}
 Using \eqref{eq : ineq 1 global}, \eqref{eq : ineq 3 global} and \eqref{eq : ineq 2 global}, we obtain that $\mathcal{T} : X_{box}(\sigma_0,\sigma_1) \to X_{box}(\sigma_0,\sigma_1)$ is well defined and satisfies
\[
\|D\mathcal{T}(W)\|_{X_\epsilon} < 1
\]
for all $W \in X_{box}(\sigma_0,\sigma_1)$. Consequently, using the Mean Value Inequality and the Banach fixed point theorem, we obtain that $\mathcal{T}$ has a unique fixed point in $X_{box}(\sigma_0,\sigma_1) \subset X_\epsilon$. The convergence to zero of the solution is a direct consequence of the definition of $X_\epsilon$. 
\end{proof}

\begin{corollary}\label{Cor : global existence}
Let $0 < \epsilon^* < \nu$ and let $h_1 >0$. Moreover let $\gamma_0 = 0.6383 \sqrt{\frac{h_1}{2\nu}}$ and $\gamma_1 = \frac{1}{2(\nu -\epsilon^*)}$.  If $\|b\|_1$ satisfies
\begin{equation}\label{eq : condition on the norm of b}
    \begin{aligned}
        &\|b\|_1 \leq  \frac{1}{4}\min\left\{\frac{1}{\gamma_0} ;~ \frac{1}{\gamma_1 e^{(-2\nu +\epsilon^*)h_1}}\right\}\\
        &1- \sqrt{1-4\|b\|_1\gamma_0} + 1- \sqrt{\frac{1-4\|b\|_1 \gamma_1 e^{(-2\nu +\epsilon^*)h_1}}{\frac{\gamma_1}{\gamma_0}+1}} < 1,
    \end{aligned}
\end{equation}
then there exists a unique solution $\tilde{w}$ to \eqref{eq : NSE full vorticity} and $\tilde{\omega}(t) \to 0$ as $t \to \infty.$
\end{corollary}

\begin{proof}
Let us denote $\gamma_0 \bydef 0.6383 \sqrt{\frac{h_1}{2\nu}}$ and $\gamma_1 \bydef \frac{1}{2(\nu -\epsilon^*)}$. Then, suppose that $\sigma_0, \sigma_1$ satisfy the following 
\begin{equation}
    \begin{aligned}
          \|b\|_1  + \gamma_0 \sigma_0^2 \leq  \sigma_0, ~~ \|b\|_1 e^{(-2\nu +\epsilon^*)h_1} +  \gamma_0 \sigma_0^2 +  \gamma_1 \sigma_1^2 \leq  \sigma_1 ~\text{ and }~
         2 \gamma_0 \sigma_0 + 2 \gamma_1 \sigma_1  < 1.
    \end{aligned}
\end{equation}
The first inequality imposes that $
   \|b\|_1 \leq \frac{1}{4\gamma_0}
$
and under that condition we have that $\|b\|_1  + \gamma_0 \sigma_0^2 \leq  \sigma_0$ if and only if
\[
\sigma_0 \in \left[\frac{1-\sqrt{1-4\gamma_0 \|b\|_1}}{2\gamma_0}~,~ \frac{1+\sqrt{1-4\gamma_0 \|b\|_1}}{2\gamma_0}\right].
\]
Now, using that $ \sigma_0 < \frac{1-2\gamma_1 \sigma_1}{2\gamma_0}$, a sufficient condition to satisfy $\|b\|_1 e^{(-2\nu +\epsilon^*)h_1} +  \gamma_0 \sigma_0^2 +  \gamma_1 \sigma_1^2 \leq  \sigma_1$ is that 
\[
\|b\|_1 e^{(-2\nu +\epsilon^*)h_1} + \frac{1}{4\gamma_0} - \left(\frac{\gamma_1}{\gamma_0}+1\right) \sigma_1 +   \left(\frac{\gamma_1^2}{\gamma_0} + \gamma_1\right) \sigma_1^2 \leq 0.
\]
In fact, the above imposes that 
\[
\left(\frac{\gamma_1}{\gamma_0}+1\right)^2 - 4 \left(\|b\|_1 e^{(-2\nu +\epsilon^*)h_1} + \frac{1}{4\gamma_0}\right)\left(\frac{\gamma_1^2}{\gamma_0} + \gamma_1\right)  \geq 0
\]
which is equivalent to
\[
\|b\|_1 \leq \frac{1}{4\gamma_1 e^{(-2\nu +\epsilon^*)h_1}}.
\]
Under such a condition, we have that 
{\small \[
\sigma_1 \in \left[\frac{1- \sqrt{\frac{1-4\|b\|_1 \gamma_1 e^{(-2\nu +\epsilon^*)h_1}}{\frac{\gamma_1}{\gamma_0}+1}}}{2\gamma_1} ~,~ \frac{1 + \sqrt{\frac{1-4\|b\|_1 \gamma_1 e^{(-2\nu +\epsilon^*)h_1}}{\frac{\gamma_1}{\gamma_0}+1}}}{2\gamma_1}\right].
\]}
Consequently, assuming that 
$
\|b\|_1 \leq \frac{1}{4}\min\left\{\frac{1}{\gamma_0}~,~ \frac{1}{\gamma_1 e^{(-2\nu +\epsilon^*)h_1}}\right\},
$
then a sufficient condition for the existence of $\sigma_0,\sigma_1 >0$ satisfying \eqref{eq : conditions for global existence} is that 
\[ 
1- \sqrt{1-4\|b\|_1 \gamma_0} + 1- \sqrt{\frac{1-4\|b\|_1 \gamma_1 e^{(-2\nu +\epsilon^*)h_1}}{\frac{\gamma_1}{\gamma_0}+1}} < 1.
\]
We conclude the proof using Lemma \ref{lem : global existence condition}.
\end{proof}

\begin{remark}
    Numerically, we can determine $h_1$ maximizing the quantity $\min\left\{\frac{1}{\gamma_0} ;~ \frac{1}{\gamma_1 e^{(-2\nu +\epsilon^*)h_1}}\right\}$ in the previous lemma. Then, one can approximate the highest value of $\|b\|_1$ satisfying \eqref{eq : condition on the norm of b} (using bisection for instance). In particular, such a quantity can be made rigorous numerically using the arithmetic on intervals. As an illustration, when $\nu =1$, one can prove that if $\|b\|_1 \leq 0.58$, then $h_1 =0.35$ satisfies \eqref{eq : condition on the norm of b} and we obtain that the associated initial value problem \eqref{eq : NSE full vorticity} possesses solutions that exist for all times and vanish as time goes to infinity. In comparison, a direct application of Gronwall's inequality yields that solutions exist for all times if $\|b\|_1 < \frac{\nu}{2} = 0.5.$ Consequently, the introduction of the parameter $h_1$ might allow to increase the trapping regions in which solutions will exist for all times and vanish at infinity.
\end{remark}

We construct $\overline{W}$, an approximate zero of $F$, and denote $\overline{\omega}$ its function representation (illustrated in Figure \ref{fig : NS}). Applying the above analysis and Theorem \ref{th : radii polynomial}, we obtain the following result.

 \begin{figure}[h!]
  \centering
  \begin{minipage}{.52\textwidth}
   \centering
  \includegraphics[clip,width=1\textwidth]{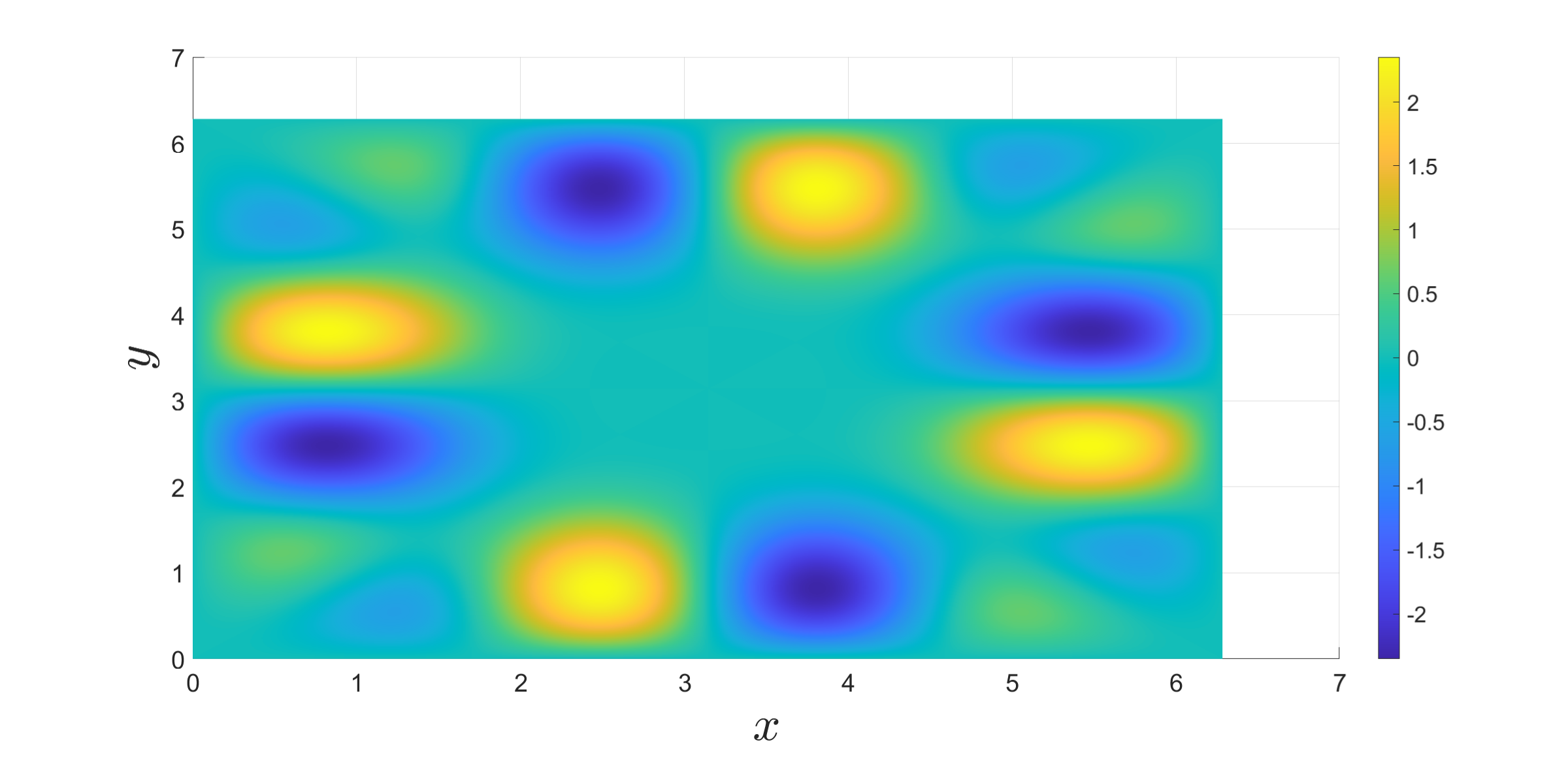}
  \end{minipage}%
  \begin{minipage}{.52\textwidth}
    \centering
   \includegraphics[clip,width=1\textwidth]{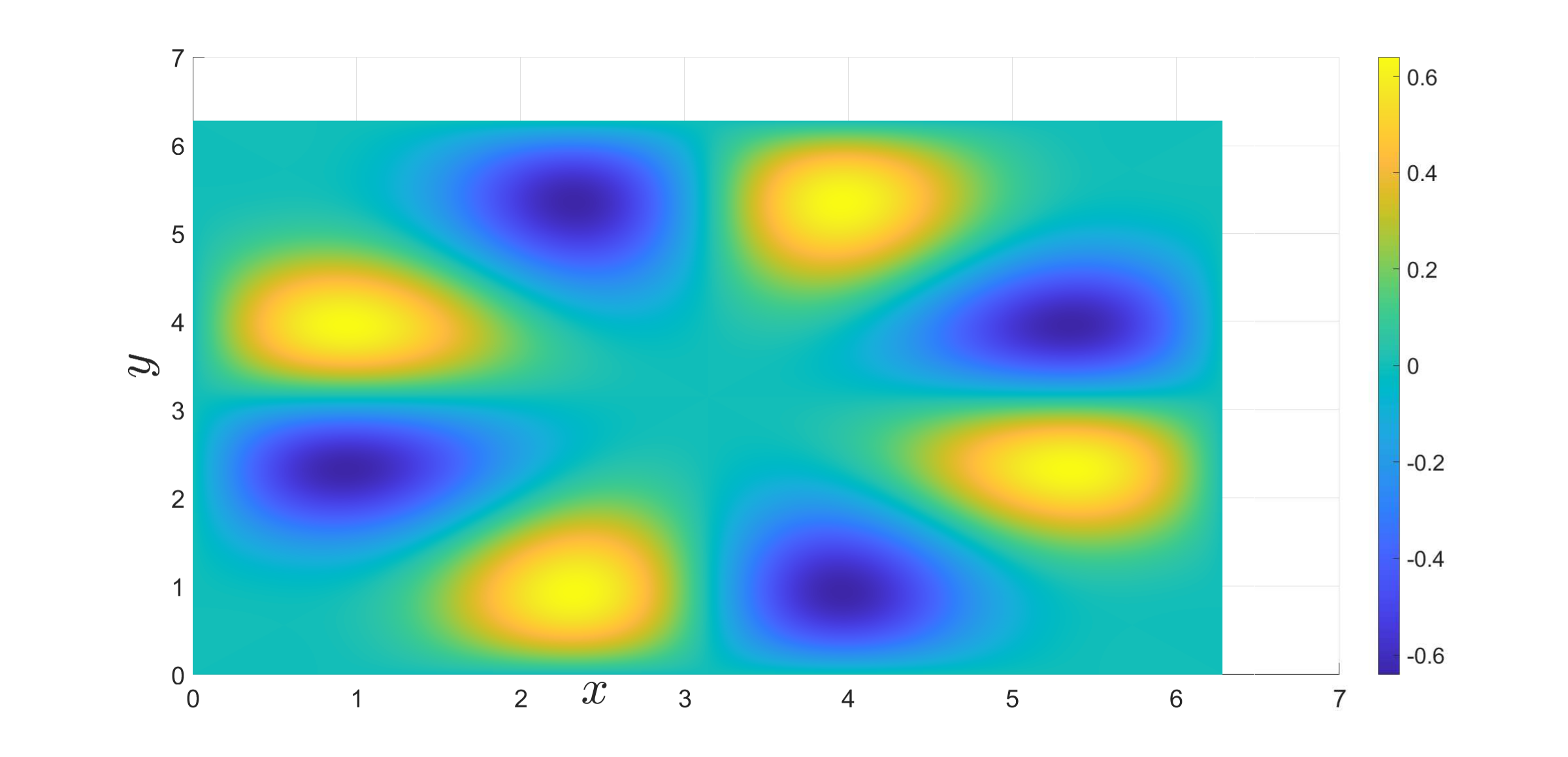}
  \end{minipage}
   \begin{minipage}{.52\textwidth}
   \centering
  \includegraphics[clip,width=1\textwidth]{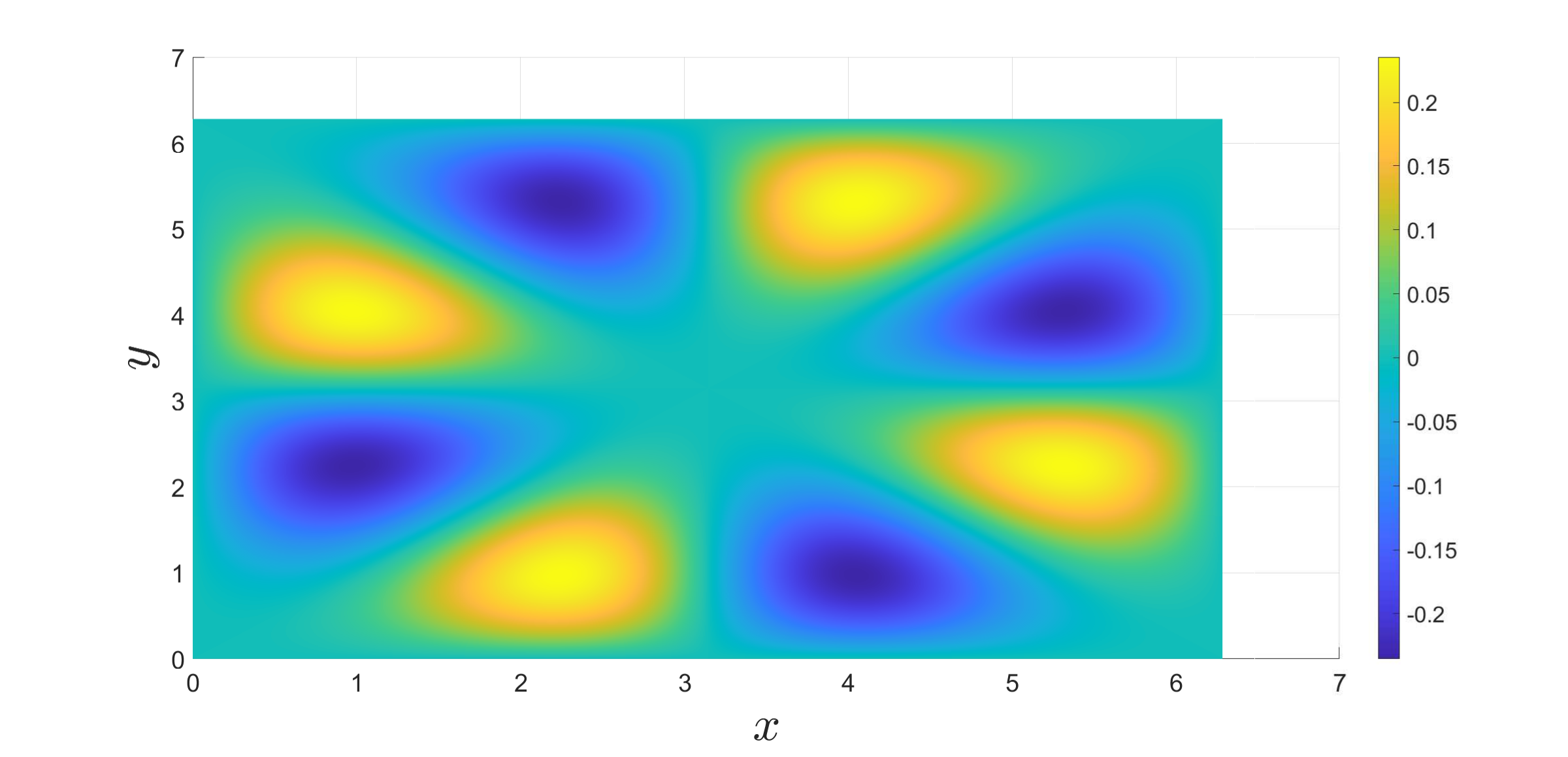}
  \end{minipage}%
  \begin{minipage}{.52\textwidth}
    \centering
   \includegraphics[clip,width=1\textwidth]{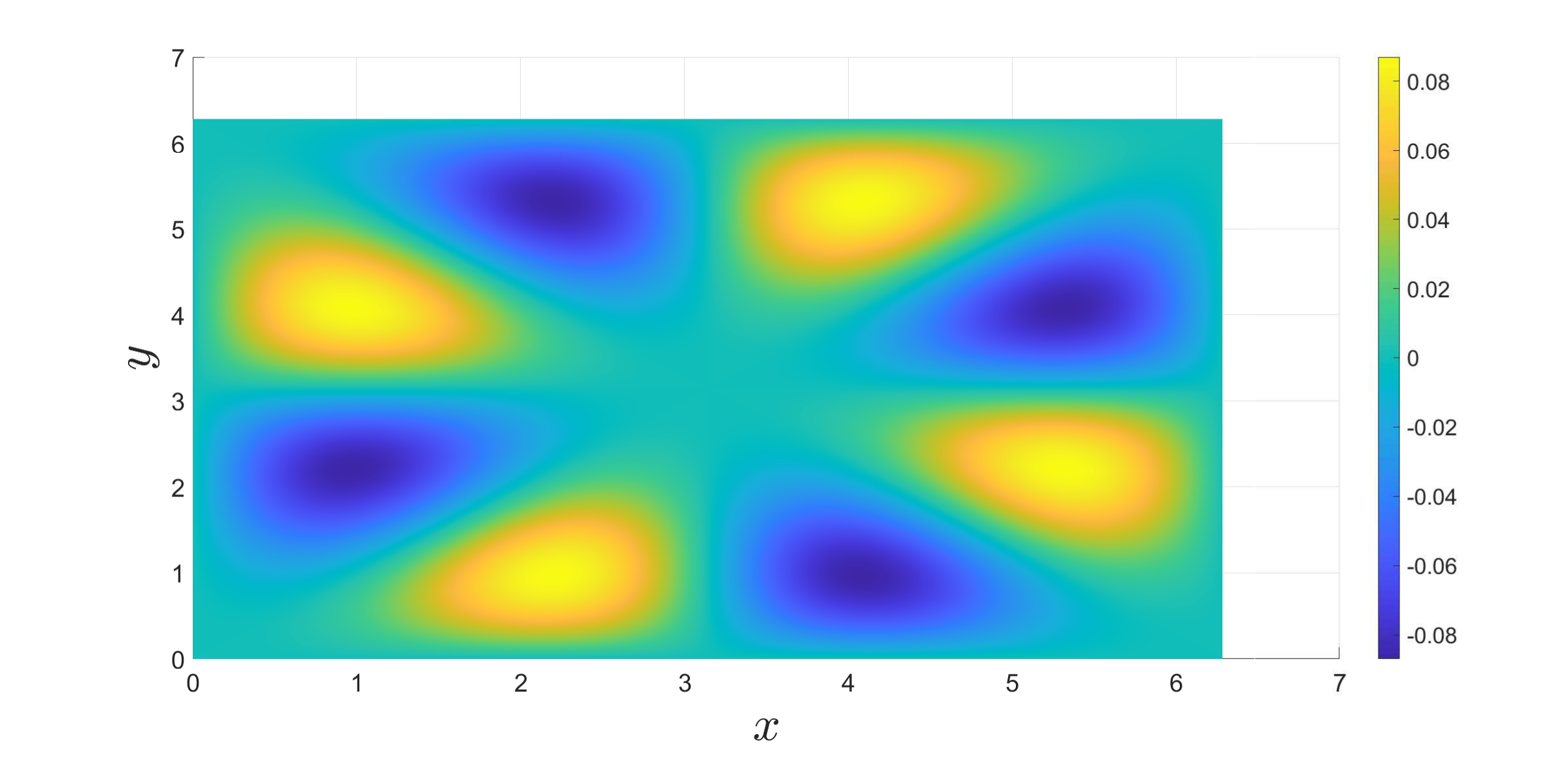}
  \end{minipage}
  \vspace{-.5cm}
  \caption{Numerical approximation for the 2D Navier-Stokes initial value problem at times $t=0$ (up-left), $t=0.96$ (up-right), $t=1.92$ (bottom-left) and $t=2.91$ (bottom-right).}
    \label{fig : NS}
  \end{figure}

\begin{theorem}\label{th : proof in NSE}
   Let $b(x) \bydef 2\sin(2x_1)\sin(x_2) - 2 \sin(x_1)\sin(2x_2) + 0.8\sin(x_1)\sin(3x_2) - 0.8 \sin(3x_1)\sin(x_2)  + 1.2\sin(2x_1)\sin(3x_2) - 1.2 \sin(3x_1)\sin(2x_2)$ and let $\nu =0.2$. Moreover, let $h \bydef 2.91$ and $r_0 \bydef 4.54 \times 10^{-5}$, then there exists a unique smooth solution $\tilde{\omega}$ to \eqref{eq : NS vorticity}  and $\|\tilde{w} - \overline{w}\|_\infty \leq r_0$. In fact, the solution $\tilde{w}$ exists for all positive times and we have that $\tilde{w}(t) \to 0$ as $t \to \infty.$
\end{theorem}

\begin{proof}
    The existence of the solution is obtained combining Theorem \ref{th : radii polynomial} and the estimations derived in Section \ref{sec : computation of the bounds}. We use the explicit formulas \eqref{def : definition of Z1}, \eqref{def : definition of Y0} and \eqref{def : definition of Z2} for computing the bounds $Z_1, Y$ and $Z_2$.  Furthermore, we use the multi-stepping approach presented in Appendix \ref{sec : appendix}. On each time step, the approximate solution is constructed with $N_1 = 15$ and $N_2 =10$, and the space $X$ is chosen to be $\ell^1$. We perform 73 time steps of size $0.03$ and one final step of size $0.72$.  The accumulated error provides a radius of existence  $r_0 =  4.54 \times 10^{-5}$. For the global existence, we prove, using  Corollary \ref{Cor : global existence}, that if $\|b\|_1 \leq 0.015$, then the associated IVP \eqref{eq : NSE full vorticity} possesses a global solution in time that vanishes at infinity.
     Let $\widetilde{W}(h)$ (resp. $\overline{W}(h)$) be the Fourier series representation of $\tilde{w}(h)$ (resp. $\overline{w}(h)$). Then we have that $\|\widetilde{W}(h)\|_1 \leq r_0 + \|\overline{W}(h)\|_1$. 
    Finally, replacing $\|b\|_1$ by $r_0 + \|\overline{W}(h)\|_1$ in Corollary \ref{Cor : global existence}, we verify that  $r_0 + \|\overline{W}(h)\|_1 \leq 0.015$ and obtain the desired conclusion. The regularity of the solution is a  consequence of the bootstrapping argument exposed in Lemma \ref{lem : regularity} combined with Remark \ref{rem : regularity} and Lemma \ref{lem : non lin navier stokes}.
\end{proof}

\begin{remark}
    Note that in the proof of Theorem \ref{th : proof in NSE}, the last time of integration is much bigger than the previous ones. Indeed, as we integrate forward in time, the amplitude of the solution decreases, leading to smaller and smaller bounds for Theorem \ref{th : radii polynomial}. In particular, when the amplitude is close enough to zero, we can reached the desired integration time with a final bigger step.
\end{remark}

\subsection{The Swift-Hohenberg PDE}

Let $\alpha >1$, we consider the following form of the 1D Swift-Hohenberg PDE
\begin{equation}\label{eq : swift hohenberg pde}
      u_t = (\alpha-1)u - 2u_{xx} - u_{xxxx} - u^3, ~~ u(0,x) = b(x)
\end{equation}
where $u$ and $b$ are $2\pi$-periodic in space. Equation \eqref{eq : swift hohenberg pde} is a well-established model for thermal convection, particularly in the context of Rayleigh-Bénard convection. Moreover, it plays a fundamental role in the study of pattern formation across various fields, including phase-field crystals \cite{phase-crystal}, magnetizable fluids \cite{GROVES20171}, and nonlinear optics \cite{Odent:2016gli}. In this section, we seek to reproduce the existence proof presented in Section 5.3 of \cite{jacek_integration_cheb}. This comparison serves to illustrate the ability to perform long-time integrations without requiring a decomposition of the temporal domain. Specifically, we numerically construct an approximate solution $\oU$ (depicted in Figure~\ref{fig : sh}) as given by \eqref{eq : expansion solution} with $N_1 = 24, N_2 = 150$ and obtain the following result.

 \begin{figure}[h!]
\centering
 \begin{minipage}[H]{0.6\linewidth}
  \centering\epsfig{figure=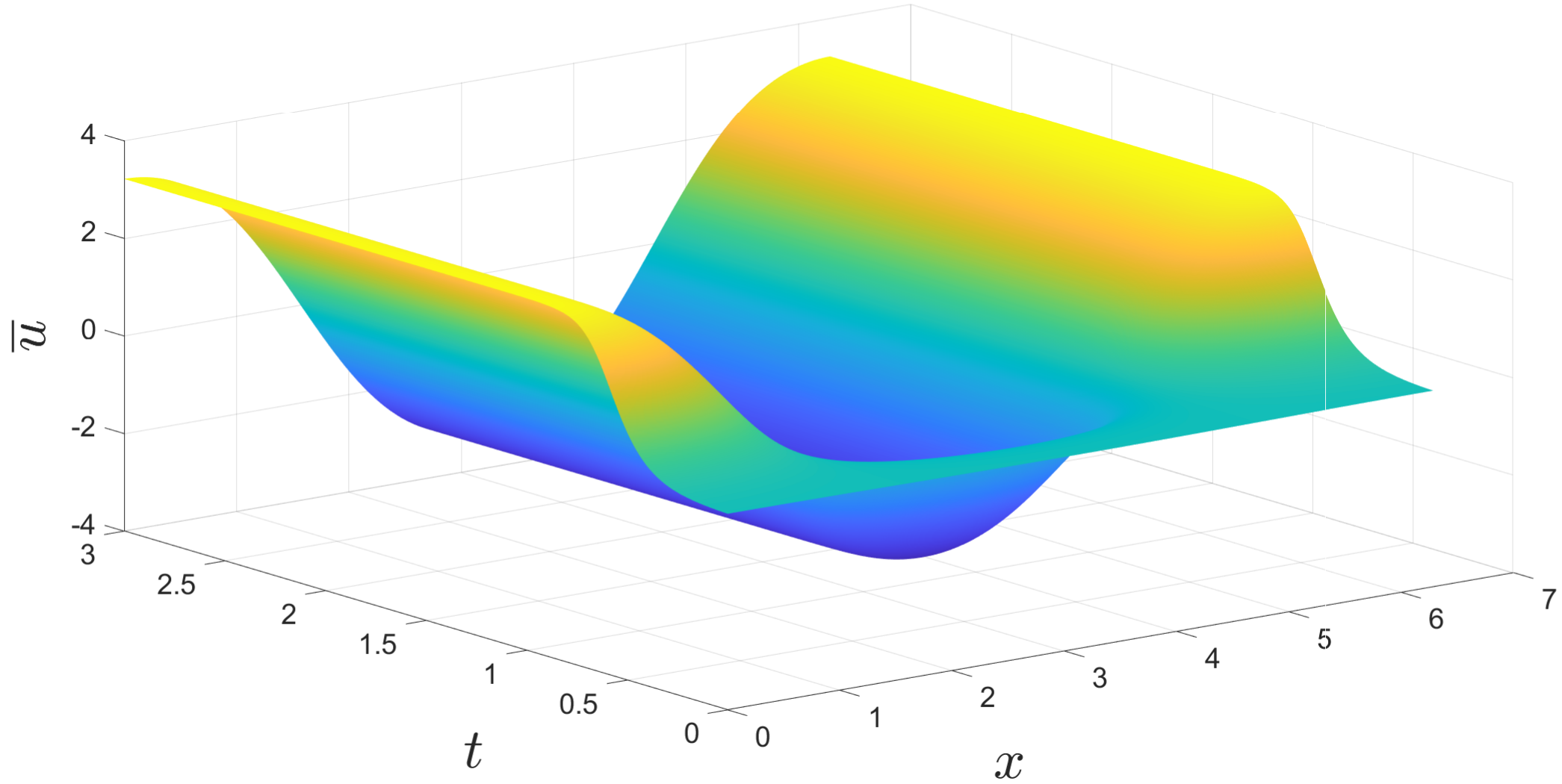,width=
  \linewidth}
  \vspace{-.4cm}
  \caption{Approximation $\overline{u}$ for the Swift-Hohenberg IVP.}
  \label{fig : sh}
 \end{minipage} 
 \end{figure}

\begin{theorem}\label{th : proof in SH}
Let $b$ be given as $b(x) \bydef 0.02\cos(x)$. Moreover, let $h \bydef 3$ and $r_0 \bydef 7.053 \times 10^{-10}$. Then there exists a unique smooth solution $\tilde{u}$ to \eqref{eq : swift hohenberg pde} and $\|\tilde{u} - \overline{u}\|_\infty \leq r_0$.
\end{theorem}

\begin{proof}
   Choosing $X = \ell^1$, the existence of the solution is obtained combining Theorem \ref{th : radii polynomial} and the formulas \eqref{def : definition of Z1}, \eqref{def : definition of Y0}, \eqref{def : definition of Z2} for the bounds $Z_1, Y$ and $Z_2$. In particular, choosing $r_0 = 7.053 \times 10^{-10}$ we obtained the following values
    \begin{align*}
        Y = 7.84 \times 10^{-11},~ Z_1 = 0.89 , ~ Z_2(r_0) = 587 
    \end{align*}
    and verified that \eqref{eq : condition contraction} is satisfied. Finally, the uniqueness and the regularity of the solution are obtained thanks to Remark \ref{rem : regularity}. 
\end{proof}

\subsection{The Kuramoto-Sivashinsky PDE}

Let $\alpha >0$ and consider the one dimensional Kuramoto-Sivashinsky PDE
\begin{equation}\label{eq : kuramoto Sivashinsky pde}
    u_t = -u_{xx} - \alpha u_{xxxx} - 2uu_x, ~~ u(0,x) = b(x)
\end{equation}
where $u$ and $b$ are $2\pi$-periodic in space.  

The objective of this subsection is to demonstrate that the proposed integration method can facilitate existence proofs for periodic solutions in time. Specifically, we aim to integrate over the period of an approximate periodic solution, thereby illustrating the feasibility of establishing existence results. To this end, we consider the approximate periodic solution $u_p$ given in \cite{MR2049869} for $\alpha = 0.127$. Our goal is to integrate forward in time from $u_p(0, \cdot)$ over a full period. It is important to note that we do not rigorously establish the existence of a periodic orbit, as we do not prove an exact return to the initial condition. Rather, we emphasize that the method enables integration over an entire period. To formally prove the existence of a periodic orbit, one would need to formulate a zero-finding problem in which both the period $h$ and the periodic condition $u(0,x) = u(h,x)$ for all $x \in \Omega$ are enforced. It is worth noting the existence of fully spectral (Fourier-Fourier) computer-assisted methods that have been employed to prove the existence of periodic orbits \cite{MR3662023, MR3623202, MR3633778, navier-stokes}.

We fix $b(x) = u_p(0,x)$ for all $x \in (-\pi,\pi)$. 
Then, we construct numerically $\overline{U}$, an approximate zero of $F$, with $N_1 = 31, N_2 = 80$ and denote $\overline{u}$ its function representation (illustrated in Figure \ref{fig : KS}). Finally, we obtain the following result.

 \begin{figure}[h!]
  \centering
  \begin{minipage}{.52\textwidth}
   \centering
  \includegraphics[clip,width=1\textwidth]{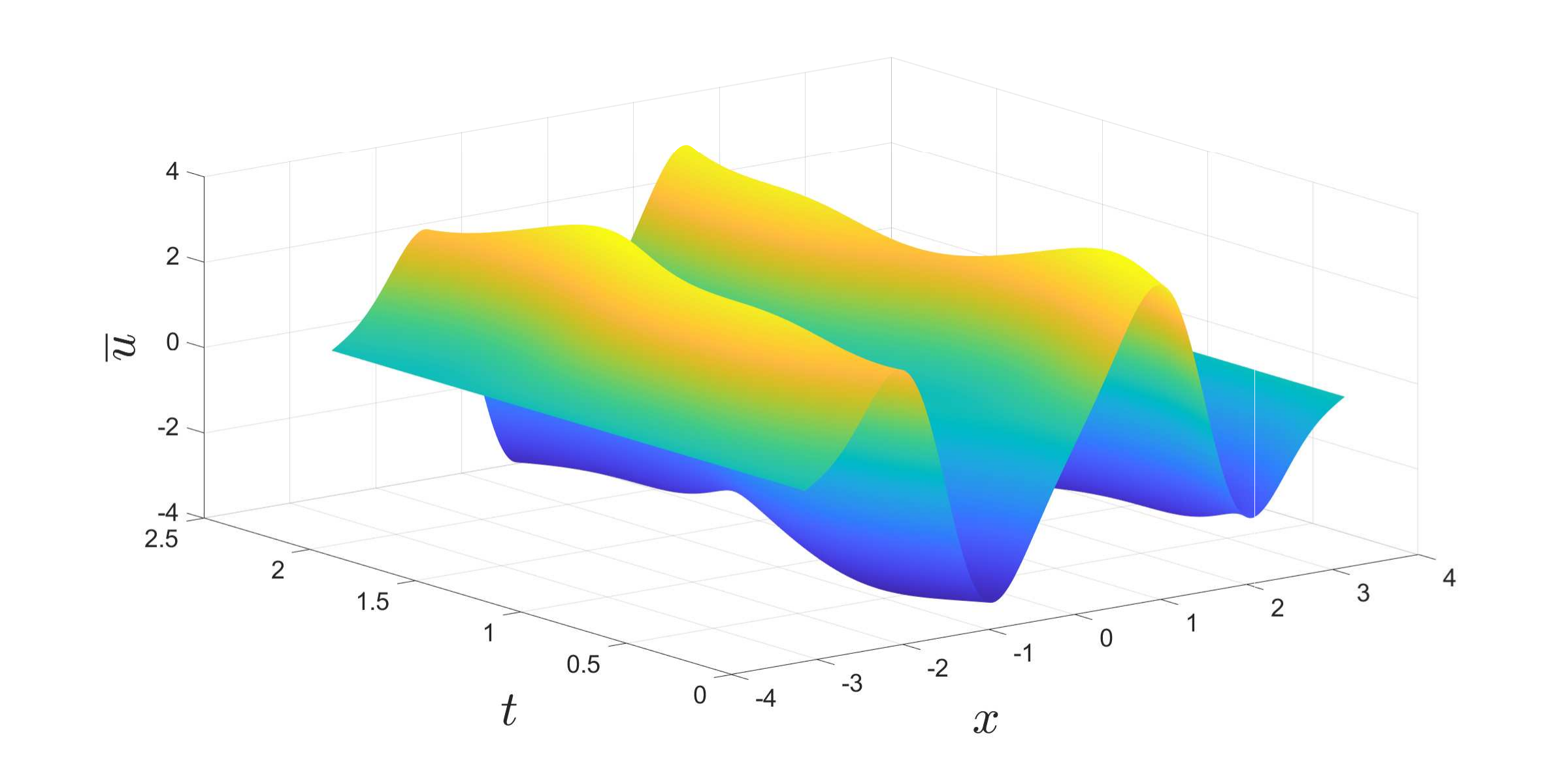}
  \end{minipage}%
  \begin{minipage}{.52\textwidth}
    \centering
   \includegraphics[clip,width=1\textwidth]{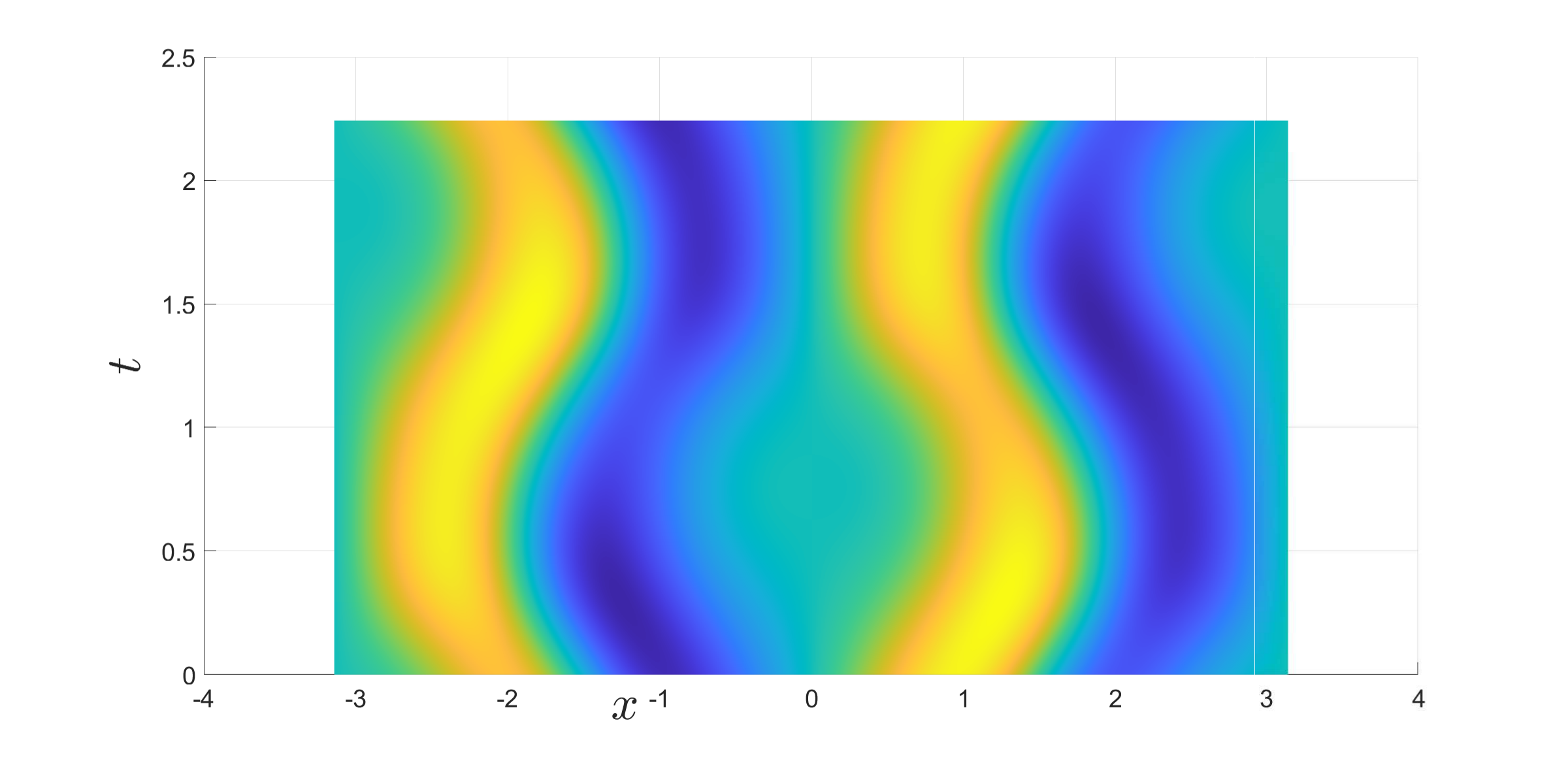}
  \end{minipage}
  \caption{Numerical approximation for the Kuramoto-Sivashinski initial value problem.}
    \label{fig : KS}
  \end{figure}

\begin{theorem}\label{th : proof in KS}
    Let $h \bydef 2.244333563761175$ and $r_0 \bydef 3.49 \times 10^{-10}$, then there exists a unique smooth solution $\tilde{u}$ to \eqref{eq : kuramoto Sivashinsky pde}  and $\|\tilde{u} - \overline{u}\|_\infty \leq r_0$.
\end{theorem}

\begin{proof}
    Choosing $X = \ell^1$, the proof is obtained similarly as the one of Theorem \ref{th : proof in SH}. In particular, choosing $r_0 = 3.49 \times 10^{-10}$ we obtained the following bounds 
    \begin{align*}
        Y = 6.81\times 10^{-11},~ Z_1 = 0.81,~ Z_2(r_0) = 3.57 \times 10^3 
    \end{align*}
    and verified that \eqref{eq : condition contraction} is satisfied.
\end{proof}

\section{Conclusion}


In this manuscript, we demonstrate how a careful analysis of Fourier-Chebyshev series enables the solution of initial value problems in parabolic partial differential equations with long integration times. In particular, we provide a precise analysis of the inverse of the operator $\mathcal{L}$, establishing the necessary notion of compactness for applying Theorem \ref{th : radii polynomial}. This approach serves as a foundational framework for tackling more complex problems. For example, the treatment of the 2D Navier-Stokes equations in Section \ref{ssec : NSE} opens the possibility of addressing IVPs in more intricate dynamical systems. In particular, we are actively investigating a similar framework for the 3D version of the equations, which may lead to the establishment of global existence for given initial data. Another area of interest is the study of connections between steady states of parabolic PDEs. Specifically, given two steady states $u_1$ and $u_2$, we aim to prove the existence of a solution $u$ such that $u(-\infty) = u_1$ and $u(+\infty) = u_2$. This extension is currently under investigation.\textbf{}

\section{Appendix}\label{sec : appendix}

\subsection{Computation of \boldmath$\mathcal{L}_k^{-1}  V$\unboldmath}

Given $k \in \mathbb{Z}^m$, we demonstrate how to compute the quantity $\mathcal{L}_k^{-1}  V$ where $V \in \ell^1(\mathbb{N}_0)$ only has a finite number of non-zero elements. In particular, we assume that $V_n = 0$ for all $n \geq N$, for some $N \in \mathbb{N}$. Using \eqref{eq : Linv times Lambda}, notice that it is enough to compute $T^{-1}_k  \tilde{V}$ for some  $\tilde{V} \in \ell^1(\mathbb{N})$, where $\tilde{V}_n = 0$ for all $n \geq N$. In the rest of the section, we drop the dependency on $k$ for simplicity. Without loss of generality, we can study $T^{-1}V$ for $V \in \ell^1(\mathbb{N})$ and $V_n =0$ for all $n \geq N$ (we drop the tilde notation for simplicity as well).

First, recalling the definition of $T(N)$ in \eqref{def : T(n)}, notice that $T$ can be decomposed as follows
\vspace{-.2cm}
{\small
\begin{equation}
    T = \begin{pmatrix}
        T(N) & \mu M_1\\
        - \mu M_2 & T_\infty
    \end{pmatrix}, 
\end{equation}
}
where 
\vspace{-.3cm}
{\tiny
\[
M_1 = \begin{pmatrix}
        0 & 0 & 0 & 0 & \dots \\
       0 & 0 & 0 & 0 & \dots \\
        \vdots &  \vdots &  \vdots &  \vdots&  ~ \\
        0 & 0 & 0 & 0 & \dots \\
        -1 & 0 & 0 & 0 & \dots 
    \end{pmatrix}, ~~ M_2 = \begin{pmatrix}
        0 & 0 & \dots & 0 & 1\\
        0 & 0 & \dots & 0 & 0 \\
        0 & 0 & \dots & 0 & 0 \\
        \vdots & \vdots & ~ & \vdots & \vdots 
    \end{pmatrix}, ~~ \text{ and }  
   T_\infty \bydef \begin{pmatrix}
     2(N+1)  & -\mu  & ~       & ~      &~ \\
    \mu  & 2(N+2)  & -\mu    & ~      &~ \\
      ~  &  \mu  & 2(N+3)    & -\mu      &~ \\
    ~     & ~ & \ddots  & \ddots & \ddots\\
    \end{pmatrix}.
\]}

Let $U = T^{-1}V$, we decompose $U = (U_N, U_\infty), V = (V_N, V_\infty)$ where $U_n = (U_N)_n$ (resp. $V_n = (V_N)_n$) for all $n \in 1, \dots, N$. In particular, we have that $U$ satisfies
\[
    T(N) U_N + \mu M_1 U_\infty = V_N \quad \text{and} \quad
    - \mu M_2 U_N + T_\infty U_\infty = 0.
\]
Note that $T_\infty$ is invertible (using \eqref{eq : inverse for Tn}) and we have that 
\[
    (T(N) + \mu^2 M_1 T_\infty^{-1} M_2) U_N  = V_N \quad \text{and} \quad 
     U_\infty = \mu   T_\infty^{-1} M_2 U_N. 
\]
By construction of $M_1$ and $M_2$, we have that $\mu^2 M_1 T_\infty^{-1} M_2 = \mu^2 (T_\infty^{-1})_{1,1} e_Ne_N^T$.
Using \eqref{eq : inverse for Tn}, we know that $(T_\infty^{-1})_{1,1} >0$. Moreover, we have that $(T(N)^{-1})_{N,N} >0$. This implies that 
\[
1 + \mu^2 (T_\infty^{-1})_{1,1} e_N^T T(N)^{-1} e_N = 1 + \mu^2 (T_\infty^{-1})_{1,1}(T(N)^{-1})_{N,N} >0.
\]
Therefore, using Sherman-Morrison formula, we get
\begin{align*}
    U_N = \left(I_d - \frac{\mu^2 (T_\infty^{-1})_{1,1}}{1+\mu^2 (T_\infty^{-1})_{1,1} (T(N)^{-1})_{N,N} } T(N)^{-1}e_Ne_N^T\right) T(N)^{-1} V_N.
\end{align*}
Now, notice that 
\[
T(N)^{-1}e_Ne_N^T T(N)^{-1} V_N = \begin{pmatrix}
    0 & \cdots & 0 & (T(N)^{-1})_{col(N)} \left(T(N)^{-1} V_N\right)_{N}
\end{pmatrix},
\]
where the right hand side is a column-wise representation of the matrix. Consequently, we obtain that 
\begin{align*}
    (U_N)_n = (T(N)^{-1}V_N)_n - \frac{\mu^2 (T_\infty^{-1})_{1,1} \left(T(N)^{-1} V_N\right)_{N}}{1+\mu^2 (T_\infty^{-1})_{1,1} (T(N)^{-1})_{N,N} } (T(N)^{-1})_{n,N} 
\end{align*}
for all $n \in \{1, \dots, N\}$. In particular, we have 
\begin{equation}\label{eq : last entry of UN}
    (U_N)_N =  \frac{ \left(T(N)^{-1} V_N\right)_{N}}{1+\mu^2 (T_\infty^{-1})_{1,1} (T(N)^{-1})_{N,N} }.
\end{equation}
Since $T(N)$ is a matrix, if $\mu^2 (T_\infty^{-1})_{1,1}$ are enclosed, the entries of $U_N$ can be estimated using interval arithmetic. A similar reasoning as in Lemma \ref{lem : lemma on sequence dn} can be applied to $T_\infty.$ We obtain the following lemma.
\begin{lemma}\label{lem : estimates for T infinity}
For all $i \geq 1$, we have
    \begin{align*}
       \mu^{i-1}\prod_{l=1}^{i} \frac{1}{\alpha_{N+l}(\mu)}  \leq (T^{-1}_\infty)_{i,1} \leq \mu^{i-1}\prod_{l=0}^{i-1} \frac{1}{\alpha_{N+l}(\mu)}. 
    \end{align*}
\end{lemma}
In particular, the above lemma yields
$\frac{1}{\alpha_{N+1}(\mu)} \leq (T^{-1}_{\infty})_{1,1} \leq \frac{1}{\alpha_{N}(\mu)}$,
where $\alpha_j(\mu)$ is defined in \eqref{eq : def of alpha j}. Consequently, the entries of $U_N$ can be enclosed thanks to the use of rigorous numerics. This is achieved in our code in \cite{julia_cadiot}. Now, recall that $U_\infty = \mu T^{-1}_\infty M_2 U_N$.
By definition of $M_2$, we obtain that 
\begin{align*}
    U_\infty = \mu (T^{-1}_\infty)_{col(1)} (U_N)_{N} =  \mu (T^{-1}_\infty)_{col(1)} \frac{ \left(T(N)^{-1} V_N\right)_{N}}{1+\mu^2 (T_\infty^{-1})_{1,1} (T(N)^{-1})_{N,N} }
\end{align*}
using \eqref{eq : last entry of UN}.  Consequently, this implies that 
\[
\|U_\infty\|_1 \leq \frac{\mu |\left(T(N)^{-1} V_N\right)_{N}|}{1+\mu^2 (T_\infty^{-1})_{1,1} (T(N)^{-1})_{N,N}} \|(T^{-1}_\infty)_{col(1)}\|_1.
\]
Using Lemma \ref{lem : norm of first column} and Lemma \ref{lem : estimates for T infinity}, one can compute a bound for $\|(T^{-1}_\infty)_{col(1)}\|_1$. This implies that an upper bound for $\|U_\infty\|_1$ can be computed explicitly, leading to an explicit approach for enclosing $\mathcal{L}_k^{-1}\Lambda V$.

\subsection{Integration from one time step to the next}

In this section, we show how to solve two successive initial value problems. In particular, this strategy allows to achieve multiple steps of time integration. Given $h_1 >0$, we start by integrating on $[0,h_1]$, i.e.
\begin{equation}\label{eq : equation for u}
    \begin{aligned}
    \partial_t u &= \mathbb{L} u + \mathbb{Q}(u) + \phi, \quad u = u(t,x), ~~ x \in \Omega, ~~ t \in [0,h_1]\\
    u(0,x) &= b(x), \quad \text{for all } x \in \Omega
\end{aligned}
\end{equation}
and we look for $v$ on $[h_1,h_2]$ (with $h_2 > h_1$) such that
\begin{equation}\label{eq : equation for v}
\begin{aligned}
    \partial_t v &= \mathbb{L} v + \mathbb{Q}(v) + \phi, \quad v = v(t,x), ~~ x \in \Omega, ~~ t \in [h_1,h_2]\\
    v(0,x) &= u(h_1,x), \quad \text{for all } x \in \Omega.
\end{aligned}
\end{equation}
Then constructing $\tilde{u} : [0,h_2] \times \Omega \to \R$ as 
\begin{align*}
    \tilde{u}(t,x) = \begin{cases}
        u(t,x) \text{ if } t \in [0,h_1]\\
        v(t,x) \text{ if } t \in (h_1,h_2]
    \end{cases}
\end{align*}
we obtain that $\tilde{u}$ solves 
\begin{align*}
    \partial_t \tilde{u} &= \mathbb{L} \tilde{u} + \mathbb{Q}(\tilde{u}) + \phi, \quad \tilde{u} = \tilde{u}(t,x), ~~ x \in \Omega, t \in [0,h_2]\\
    \tilde{u}(0,x) &= b(x), \quad \text{for all } x \in \Omega.
\end{align*}

Now our goal is to understand how the error is propagating from one time step to the next. First, similarly as the definition of $F$ in \eqref{def : definition of F}, define $F_1$ as 
\begin{align}\label{def : definition of F1}
    F_1(U) \bydef U   + \frac{h}{2}\mathcal{L}^{-1}\mathbf{\Lambda} Q(U)  + \frac{h}{2}\mathcal{L}^{-1} \mathbf{\Lambda} \Phi - \mathcal{L}^{-1}\beta.
\end{align}
Assume that $\oU$ is an approximate zero of $F_1$, as constructed in Section \ref{ssec : Newton Kantorovich} and that, using Theorem \ref{th : radii polynomial}, we proved that there exists $\tilde{U} \in X$ (associated to its function representation $\tilde{u}$) such that $F_1(\tilde{U}) = 0$ with $\tilde{U} \in \overline{B}_{r_1}(\overline{U})$ and $r_1 >0$. Now, note that the evaluation of $\tilde{u}$ at $h_1$ is given in coefficient space by $\gamma(\tilde{U})$, where $\gamma$ is defined as $\gamma(U)_{k,0} = U_{k,0} +  2 \sum_{j \in \mathbb{N}}U_{k,j}$ and $\gamma(U)_{k,n}=0$ for $n \ne 0$.
Then, define $F_2$ as 
\begin{align}\label{def : definition of F2}
    F_2(V) \bydef V   + \frac{h}{2}\mathcal{L}^{-1}\mathbf{\Lambda} Q(V)  + \frac{h}{2} \mathcal{L}^{-1}\mathbf{\Lambda} \Phi - \mathcal{L}^{-1}\gamma(\tilde{U})
\end{align}
and notice that zeros of $V$ equivalently provides zeros of \eqref{eq : equation for v} in function space. In particular, assume that we possess an approximate zero $\overline{V}$ of $F_2$ such that $\overline{V} = \pi^{N} \overline{V}$. Moreover, as in Section \ref{ssec : Newton Kantorovich}, let $A$ be constructed as in \eqref{def : operator A} and be an approximate inverse for $DF_2(\overline{V})$. 

Now, we wish to apply Theorem \ref{th : radii polynomial} to $F_2, \overline{V}$ and $A$. Notice that the bounds $Z_1$ and $Z_2$ can be computed as described in Section \ref{sec : computation of the bounds}. The bound $Y$ differs because of the term $\gamma(\tilde{U})$, which is not known explicitly. However, we can use that $\|\oU - \tilde{U}\|_{1,\omega} \leq r_1$ by definition of $\tilde{U}$. Now, we have 
\begin{align*}
    \|AF_2(\overline{V})\|_{1,\omega} \leq \|A(F_2(\overline{V}) - \mathcal{L}^{-1}\gamma(\oU))\|_{1,\omega} + \|A \mathcal{L}^{-1}\gamma(\tilde{U}-\oU)\|_{1,\omega}.
\end{align*}
since $\gamma$ is linear. Furthermore, $\oU$ only has a finite number of non-zero elements which are known explicitly. This implies that the quantity $\|A(F_2(\overline{V}) - \mathcal{L}^{-1}\gamma(\oU))\|_{1,\omega}$ can be computed as in Lemma \ref{lem : Y bound}.
Now, 

\begin{align*}
    \|A \mathcal{L}^{-1}\gamma(\tilde{U}-\overline{U})\|_{1,\omega} &= \sum_{k \in \mathbb{Z}^m}   w_{0,k} \| W_0 (A\mathcal{L}^{-1})_{col(k,0)}\|_1 |U_{k,n}-\overline{U}_{k,n} +  2\sum_{n \in \mathbb{N}}(U_{k,n}-\overline{U}_{k,n})| \\
    &\leq  \max_{k \in \mathbb{Z}^m}  \|W_0 A(\mathcal{L}^{-1}_k)_{col(0)}\|_1 r_1
\end{align*}
using that $\|\tilde{U}-\oU\|_{1,\omega} \leq r_1$. 
The above analysis allows to perform successive steps of integration, with a propagation of error determined in parts by $\max_{k \in \mathbb{Z}^m}  \|W_0 A(\mathcal{L}_k^{-1})_{col(0)}\|_1$.

There is a second component having an influence on the error propagation, that is $Z_1$. Indeed, from Theorem \ref{th : radii polynomial}, if the bound $Y$ is very small compared to both $Z_2$ and $Z_1$, then one can easily show that $r_0 \approx \frac{Y}{1-Z_1}$. Hence, for the zero finding problem \eqref{def : definition of F2}, we have shown that $Y$ will be of the form
\begin{align*}
    Y = \|A(F_2(\overline{V}) - \gamma(\oU))\|_{1,\omega} + r_1\max_{k \in \mathbb{Z}^m}  \|W_0 A(\mathcal{L}_k^{-1})_{col(0)}\|_1. 
\end{align*}
Then, the new radius of contraction is approximately given by $r_2 \approx \frac{Y}{1-Z_1}$, which leads to 
\begin{align*}
    r_2 \approx \frac{\|A(F_2(\overline{V}) + \mathcal{L}^{-1}\gamma(\tilde{U} - \oU))\|_{1,\omega}}{1-Z_1} + r_1\frac{\max_{k \in \mathbb{Z}^m}  \|W_0 A(\mathcal{L}_k^{-1})_{col(0)}\|_1}{1-Z_1},
\end{align*}
that is, from a step to the next, the error accumulates exponentially at the rate $\frac{\max_{k \in \mathbb{Z}^m}  \|W_0 A(\mathcal{L}_k^{-1})_{col(0)}\|_1}{1-Z_1}$.

\begin{remark}
    If $h_2-h_1$ is  small, then $A \approx I_d$. Since $\|W_0 (\mathcal{L}_k^{-1})_{col(0)}\|_1 = 1$ by Lemma \ref{lem : general bound Linv}, we obtain that $\|W_0 A (\mathcal{L}_k^{-1})_{col(0)}\|_1 r_1 \approx r_1$ if $h_2-h_1$ is small. Moreover, using Remark \ref{rem : remark Z1}, we know that $Z_1 = \mathcal{O}(h^2)$. Hence, the error accumulates almost linearly as the step size gets smaller. Note that one could use the framework of \cite{MR4777935} to perform multiple integrations with a more controlled wrapping effect.
\end{remark}

\bibliographystyle{abbrv}
\bibliography{papers}

\end{document}